%% file: inertialess.tex
\documentclass[11pt,reqno]{amsart}

\usepackage{AAHG}

\title[Bounds for inertialess dynamo]{Bounds for inertialess dynamo}

\author{Ali Arslan}
\address{AA: 
Institute of Geophysics\\
ETH Zürich\\ 
Zürich, CH-8092, Switzerland}
\email{\href{mailto:ali.arslan@eaps.ethz.ch}{\tt ali.arslan@eaps.ethz.ch}}

\author{Hezekiah Grayer II}
\address{HG:
Program in Applied \& Computational Mathematics\\
Princeton University \\ \newline
Princeton, NJ 08544, USA}

\email{\href{mailto:hgrayer@math.princeton.edu}{\tt hgrayer@math.princeton.edu}}

\dedicatory{In memory of Dr.~John Bryan Taylor}

\begin{document}

\begin{abstract}
  We derive necessary conditions for instantaneous dynamo action for rotating convection. A magnetohydrodynamic model is considered in two settings: the  rapidly rotating plane layer where inertia and viscosity are absent, and at an arbitrary rotation rate where  viscosity is finite.
  In contrast to kinematic dynamo bounds, the evolution of the magnetic field is coupled  via an inertialess force balance. The buoyancy-driven part of the flow $\bm{u}^{\mathrm{A}}$ in the event of dynamo action must in fact satisfy, for $3\leq p \leq \infty$
  \begin{equation*}
      \Rm\, A_p\| \bm{u}^{\mathrm{A}}\|_{L^p} \geq 1
  \end{equation*}
where $A_p$ is an explicit constant, and $Rm$ is the magnetic Reynolds number.
In the inviscid model, $\bm{u}^{\mathrm{A}}$ depends only on the horizontal gradients of the vertical primitive of temperature.
A refinement via the poloidal-toroidal decomposition allows us to replace $L^p$ in our constraint with an anisotropic norm for $L^{\infty}_z \dot{H}^1_{x,y}$.  
  For the viscous model, we also derive necessary conditions for the  growth of magnetic enstrophy and a combined thermo-magnetic energy. 
  One branch of our constraints implies that the scaling $\Rmod_\nu \gtrsim Ek^{-3/2}$ is necessary for dynamo action, where $\Rmod_\nu$ is the classical Rayleigh number
  and $Ek$ is the Ekman number. 
\end{abstract}

\maketitle
\tableofcontents

\section{Introduction}
\label{sec:intro}

The Earth, several other planets, and the Sun all possess large scale magnetic fields driven by the motion of an electrically conducting fluid under convection.  
The Earth's magnetic field is aperiodic, undergoing secular variation and reversals, and the other planetary magnetic fields across the Solar System exhibit considerable temporal and morphological diversity; see  \cite{roberts2013genesis,kono2002recent, moffatt2008,davies2015constraints} and \cite{jones2011planetary,schubert2011planetary,tikoo2022,soderlund2025puzzles}. 

Understanding which balance of forces and parameter regimes are necessary for dynamo action remains a central problem in dynamo theory \cite{yadav2016approaching,schwaiger2019}.
A basic model for the force balance is  the magnetohydrodynamic equations under the Boussinesq approximation \cite{braginsky1995equations,dormy2024}. 
Planetary magnetic fields are expected to be generated by a rapidly rotating dynamo, where inertial and viscous forces are  relatively small \cite{schwaiger2019}.
This force hierarchy can be investigated by direct numerical simulations, but the parameters relevant to planets are beyond the current capabilities of computational methods \cite{landeau2022sustaining}. 
In this paper, we address mathematically the basic question:  what conditions are necessary for dynamo action in an inertialess dynamo?

\subsection{Contribution}
\label{sec:contintro}

We first consider the Taylor limit, in which inertia and viscosity are absent from the force balance, in a horizontally periodic plane layer with either boundary or internal heating; see \cite{taylor1963} and \cite{hughes2016strong,cattaneo2006,hughes2019force,cattaneo2017dynamo,jones1991dynamo}. We show that critical to dynamo action is the contribution of  the buoyancy-driven part of the flow, i.e., the Archimedean part of the fluid velocity
$$\bm{u}^{\mathrm{A}}(x,y,z) = \Rmod\, \nabla_{x,y}^{\perp}\int_0^z  T(x,y,z') \, \mathrm{d}z'$$  where $\Rmod$ is the modified Rayleigh number. More precisely, we prove that for $3\leq p\leq\infty$ instantaneous growth of the magnetic energy requires
\begin{equation*}
\label{eq:main_intro}
    \Rm\, A_p\|\bm{u}^{A}\|_{L^p} = \Rm\, \Rmod\,A_p\|\nabla_{x,y}\overline{T}\|_{L^p}\geq1 
\end{equation*}
where $\overline{T}$ is the vertical primitive of the temperature, $A_p$ is a constant, and $\Rm$ is the magnetic Reynolds number. We refine the above bound via the  poloidal-toroidal decomposition of the magnetic field to replace $L^p$ with the anisotropic $L_z^\infty \dot{H}^1_{x,y}$.
Further, we prove absorbing ball estimates for the temperature equation wherein we   determine explicit regions of parameter space in which magnetic energy growth is prohibited in long-time convection.

We then extend the analysis to the viscous inertialess problem with finite the Ekman number  $Ek>0$ and stress-free velocity boundary conditions. The condition for dynamo action is robust and  holds with finite viscosity corrections, showing that the restriction on the Archimedean response is not reliant on singular properties of Taylor's constraint. 
The constraint naturally yields the scaling $Ra_\nu \gtrsim Ek^{-3/2}$, where $Ra_\nu$ is the classical Rayleigh number in the study of convection.
In addition to magnetic energy, we prove a constraint for the growth of magnetic enstrophy and a combined thermo-magnetic energy. The growth condition for the combined energy in a rapidly rotating solution branch again has the scaling $Ra_\nu\gtrsim Ek^{-3/2}$.

\subsection{Magneto-geostrophic models}
\label{sec:mgintro}

The inviscid model considered in this paper was studied by J. B. Taylor in his seminal work on the geodynamo \cite{taylor1963}. Taylor observed that when inertia and viscosity are absent, for a fluid in a self gravitating sphere, the total magnetic torque on every geostrophic cylinder centred on the rotation axis must vanish. This condition is now known as Taylor’s constraint \cite{jones1991dynamo}. Equivalently, the Lorentz force must lie in the range of the Coriolis operator after projection onto solenoidal vector fields \cite{gallagher2017}. The constraint is one of the main reasons why the Taylor dynamo is mathematically delicate. At each instant, the magnetostrophic and buoyancy driven components of the velocity are constrained, whereas the geostrophic component is not determined by the instantaneous force balance alone \cite{walker1998note,moffatt2008,li2018taylor}. Taylor proposed that the requirement of remaining on the constraint manifold determines this geostrophic flow, but the original prescription was later shown to be insufficient in general three-dimensional Taylor states due to an incorrect evaluation of the boundary data in spherical domains \cite{hardy2018}.

At finite viscosity, $Ek>0$, viscous forces regularise the degeneracy of the Taylor limit. 
This viscous inertialess problem has formed the basis of several plane layer studies, including computations observing weak- and strong-field dynamo states \cite{jones2000convection,cattaneo2017dynamo,hughes2019force,hughes2016strong}. Realistic $Ek \ll 1$ remain computationally inaccessible by direct numerical simulations. Instead, a `Taylorisation coefficient' has been used to measure the degree to which Taylor's constraint is satisfied as $Ek$ decreases \cite{stellmach2004cartesian,rotvig2002rotating}.
We consider both the singular Taylor limit and its finite viscosity extension. 

The solution theory of the Taylor model is incomplete. Well-posedness is subtle even for  perturbative variants \cite{moffatt1994,friedlander2011,friedlander2011b,friedlander2012,friedlander2015,gallagher2017}. However, the instantaneous force balance still imposes useful constraints. In particular, identities associated with the Coriolis operator allow magnetic energy production to be related to the horizontal Archimedean flow generated by buoyancy, without determining the subsequent evolution of the geostrophic flow. We can therefore ask: for smooth solutions of Taylor's model, what quantitative conditions are necessary for magnetic energy or magnetic enstrophy to grow?

\subsection{Kinematic dynamo theory}
Previous rigorous studies of dynamo action are restricted to kinematic dynamos.
Kinematic dynamos are simplifications in which magnetic feedback on the flow is neglected \cite{moffatt1978field}. In this setting, the velocity field, $\bm{u}$, is determined independently of the magnetic field. The induction equation then becomes a linear evolution equation for the magnetic field, and dynamo action is formulated as the existence of magnetic fields with a positive asymptotic growth rate under the action of the induction operator. The study of kinematic dynamos is diverse \cite{dudley1989time,gubbins2008,alexakis2011,Galloway2012,Rincon_2019,klapper1992,friedlander1991}, here we review the results relevant to our subsequent nonlinear study. 


The first major antidynamo result is due to Cowling \cite{cowling1933magnetic} and states that an axisymmetric magnetic field vanishing at infinity cannot be maintained by dynamo action. A Cartesian analogue of Cowling’s theorem states that a magnetic field invariant in one Cartesian direction and decaying appropriately at infinity cannot be maintained. This restriction on magnetic field geometry is distinct from Zel’dovich’s theorem, which prohibits dynamo action driven by a planar two dimensional velocity field in an unbounded domain \cite{zeldovich57}. A related result is that a purely toroidal flow cannot generate a dynamo; however, if a small poloidal flow exists, a kinematic dynamo requires a sufficiently large toroidal flow \cite{proctor2004}. Nuances to the axisymmetric theorems have also been shown to exist \cite{kaiser2014axisymmetric,kaiser2018approximately,arslan2021dynamo}.

Antidynamo theorems exclude exact geometries. A complementary approach excludes flows that are too weak, by deriving necessary lower bounds on $Rm$.
The first result due to Childress \cite{childress1969theorie} applies to velocity fields normalised by their maximum, in which case, $Rm \geq \pi$, whereas if the maximum strain, $\nabla \bm{u}$, of the flow is chosen as the velocity scale, $Rm \geq \pi^2$ \cite{backus1958}. Both results can be improved for insulating boundary conditions requiring continuity of the wall normal magnetic field to a potential field in the exterior, thereby increasing the minimal necessary $Rm$ to $3.51$ and $12.29$ \cite{proctor1977}. 
In spherical domains, with a radially weighted maximum norm, Kaiser \& Tilgner demonstrate that for a selection of velocity fields, the lower bounds are improved by a factor of up to 2.6  \cite{kaiser2025radially}.

By contrast, if the velocity is normalised by an $L^2$ norm, there is no positive universal lower bound on $Rm$. A result demonstrated by constructing a class of flows with decreasing minimum $Rm$ under a fixed energy norm \cite{proctor2015}. However, when the velocity is normalised in an $L^p$ norm with $3\leq p\leq\infty$, then Luo et al. \cite{luo2020optimal} proved that $Rm \geq \pi^{\frac{6-n}{2n}}(3^{-1/2} (2 \pi^{-1})^{2/3} ) ^{\frac{6-3n}{2n}}$ where $2\leq n \leq 6$ and $1/n + 1/p = 1/2$. Their result for $n=2$ and $p=\infty$ is equivalent to the original result of Childress, while for $n=6$ and $p=3$ it gives $2.34$.
The authors also find that, under an $L^2$ enstrophy normalisation, $Rm \geq 3.1$ \cite{luo2020optimal}.

Kinematic dynamo theory has limited  application to the magneto-geostrophic models discussed in \S\ref{sec:mgintro}  because the geostrophic  component of the flow is typically difficult to compute. Crucially, our constraints for dynamo action do not require  information from  geostrophic component.


\subsection{Outline}
The paper is organised as follows.  In \S 2 we introduce the dynamo model and Taylor’s constraint.  \S3 states the main results, first for the Taylor limit and then for the viscous inertialess model. In \S4 we present bounds for the temperature equation and use them to simplify the results of the previous section in terms of "classical" Rayleigh numbers.
\S 5, 6 and 7 contain the proofs of the inviscid, viscous and temperature equation results. \S8 provides concluding remarks. The appendices collect constants and identities of the poloidal-toroidal decomposition.

\section{Governing Equations}
\label{sec:gov_eqs}

We consider an electrically conducting fluid confined by two horizontal parallel plates in a rotating frame of reference. The frame rotates at a rate $\Omega$ and gravity acts antiparallel to the  vertical direction with strength $g$. 

The fluid is considered incompressible and has density $\rho$. The fluid has kinematic viscosity $\nu$. For thermal properties, the fluid has thermal diffusivity $\kappa$, specific heat capacity  $c_p$ and thermal expansion  coefficient $\alpha$. As a conductor, the fluid has electrical conductivity $\sigma$ and permeability $\mu$.
The plates are separated by a distance $d$, and are periodic in the horizontal directions with periods $L_x d$ and $L_y d$. The fluid is  subject to internal or boundary heating. 

\subsection{Inertialess Magnetohydrodynamics}
\label{sec:taylor_eqs}

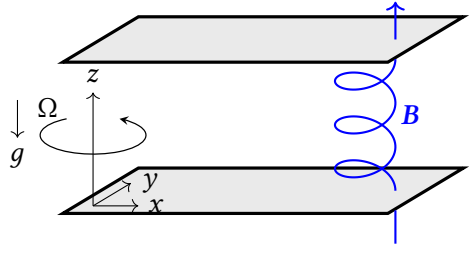
\begin{figure}[t]
\centering
\begin{tikzpicture}[every node/.style={scale=1.00}]
    \draw[black,very thick, fill=mygrey, fill opacity = 0.2] (-2,2) -- (-1,2.6) -- (3.3,2.6) -- (2.3,2) -- cycle;
    \draw[black,very thick, fill=mygrey, fill opacity = 0.2] (-2,0) -- (-1,0.6) -- (3.3,0.6) -- (2.3,0) -- cycle;
    \draw[->] (-2.6,1.5) -- (-2.6,1) node [anchor=north] {$g$};
    \draw[->] (-1.6,0.1) -- (-1,0.1) node [anchor=west] {$x$};
    \draw[->] (-1.6,0.1) -- (-1.1,0.4) node [anchor=west] {$y$};
    \draw[->] (-1.6,0.1) -- (-1.6,1.6) node [anchor=south] {$z$};
    \node at (-1.6,1) {\AxisRotator[rotate=-90]};
   \def\radius{0.4}
  \def\height{3}
  \def\nturns{3}
  \draw[thick, blue, domain=0:\nturns*360, samples=200, variable=\t]
    plot ({2+\radius*cos(\t)}, {0.3 + 0.6*\radius*sin(\t) + 0.58*\height*\t/(\nturns*360)});
    \draw [-, thick, matlabblue] (2.4,0.04) -- (2.4,-0.4);
    \draw [->, thick, matlabblue] (2.4,2.3) -- (2.4,2.8);
    \node at (-2.2,1.4) {$\Omega$};
    \node at (2.6,1.3) {${\color{matlabblue}\bm{B}}$};
    \end{tikzpicture}
\caption{Schematic diagram  of our dynamo configuration.  The fluid is bounded between two horizontal plates and the domain is periodic in the $x$ and $y$ directions. The frame rotates about the $z$ axis at a constant rate $\Omega$ and $g$ is the acceleration due to gravity. The magnetic field $\bm{B}$ generated by the fluid is illustrated with blue lines.}
\label{fig:schema}
\end{figure}

To nondimensionalise the problem, 
we fix $d$ as the characteristic length scale, $U$ as the velocity scale,  $B_0$ as the magnetic field scale, and $T_0$ as the temperature scale. In  nondimensional terms, the electrically conducting fluid fills a domain $V = [0,L_x]\times [0,L_y]\times [0,1]$ and  is described by the magnetohydrodynamic equations
\begin{subequations}
\label{eq:gov_eqs}
    \begin{align}
     \bm{e}_z \times \bm{u} + \nabla P  - \Ek \nabla^2 \bm{u} &=  \Lambda (\nabla \times\bm{B}) \times \bm{B} + \Rmod\, T \bm{e}_z, \label{eq:momentum_eq}\\
            \nabla \cdot \bm{u} &= 0, \label{eq:incompressibility}\\
     \partial_t \bm{B} - \nabla\times(\bm{u}\times \bm{B})  &= \frac{1}{\Rm}\nabla^2 \bm{B},
     \label{eq:induction_eq} \\
          \nabla \cdot \bm{B}&=0,  \label{eq:Bdiv} \\
          \partial_t T + \nabla \cdot (\bm{u} T) &= \frac{1}{\Pe}\nabla^2 T +  \frac{Q}{\Pe} .\label{eq:temperature_eq} 
    \end{align}
\end{subequations}
The equations are inertialess having taken the limit $\Ro \to 0$, where $Ro$ is the Rossby number. The nondimensional parameters, and their role in the model, are detailed below.

The unknowns are the fluid velocity, the magnetic field, and the temperature:
\begin{equation}
       \bm{u} = \sum_{i \in \{ x, y , z\}}u_i(\bm{x},t) \bm{e}_{i}, \quad  \bm{B} = \sum_{i \in \{ x, y , z\}}B_i(\bm{x},t) \bm{e}_i, \quad T = T(\bm{x}, t).
\end{equation}
The internal heating
\begin{equation}
    Q = Q(\bm{x}),
\end{equation}
 is nondimensionalized by  $T_0$.

Imposed on these functions are periodicity in $x$ and $y$, and boundary conditions on $\partial V = \{ z = 0\} \cup \{ z = 1\}$. On the boundaries $S_{\pm} = \{ z=1/2 \pm 1/2\}$ the outward unit normal is $\bm{n}_{\pm} = \pm \bm{e}_z$. The boundaries are impenetrable, ferromagnetic and isothermal. In the inviscid limit $\Ek \to 0$,  the boundary conditions  are
\begin{subequations}
\label{eq:bc}
\begin{align}
    \bm{u}\cdot\bm{n}_{\pm}\,|_{\,S_{\pm}} &= 0, \label{eq:noslip} \\
  \bm{B} \times \bm{n}_{\pm}\,|_{\,S_{\pm}} &=0, \label{eq:insulating} \\
        T\,|_{\,S_{\pm}} &= T_{\pm}. \label{eq:isothermal}
\end{align}
\end{subequations}
In the viscous case $\Ek > 0$, the fluid  satisfies the stress-free condition
\begin{equation}
    \label{bc:stress_free}
     (\nabla \times \bm{u})\times \bm{n}_{\pm}\,|_{\,S_{\pm}} = 0.
\end{equation}
The ferromagnetic boundary condition for the magnetic field \eqref{eq:insulating} implies  
\begin{subequations}
\begin{align}
 (\nabla \times \bm{B}) \cdot \bm{n}_{\pm}\,|_{\,S_{\pm}} &=0, \label{eq:insulating2} \\
\partial_z \bm{B} \cdot  \bm{n}_{\pm}\,|_{\,S_{\pm}} &=0.
 \end{align}
\end{subequations}
The boundary values $T_{+}$ and $T_{-}$ are constant temperatures.

The model is completed with initial conditions for $\bm{B}$ and $T$. As our analysis is independent of this data, given it is compatible, we do not detail this point.

\subsubsection{Nondimensional parameters} 
From the dimensional parameters defined in the previous section the magnetic diffusivity can be defined as $\eta = 1/(\sigma \mu)$.

The nondimensional numbers are 
\begin{equation}
\label{eq:nondim_nos}
\renewcommand{\arraystretch}{1.8} 
\begin{array}{r l}
\text{Rossby:} & \Ro = U / (2 \Omega d)
\\
\text{Ekman:}  & \Ek = \nu / ( 2 \Omega d^2)
\\
\text{Elsasser:}  & \Lambda = B_0^2 / (2 \Omega \rho \mu  Ud ) 
\\
\text{(modified) Rayleigh:}  & \Rmod = T_0 \alpha g   / (2 \Omega U)
\\
\text{magnetic Reynolds:}     & \Rm = Ud / \eta
\\
\text{Peclet:}     & \Pe = Ud / \kappa.
\end{array}
\end{equation}
For arbitrary $U$, one observes that the six parameters are independent. The rapid rotation limit $\Ro \to 0$ is taken to obtain \eqref{eq:gov_eqs}. Five parameters remain free. 

By specifying $\Rm$ or $\Pe$, the timescale may be fixed to obtain various rotating  magnetohydrodynamic and convective settings considered in the literature: 
\begin{enumerate}
    \item magnetic diffusive timescale
    \begin{equation}
        \Rm = 1 \quad \Rightarrow \quad 
        \begin{cases}
            \Lambda =  B_0^2 / (2 \Omega \rho \mu  \eta ) \\
            \Rmod = T_0 \alpha g d  / (2 \Omega \eta)\\
            \Pe  = \eta / \kappa  
        \end{cases},
    \end{equation}
    \item thermal diffusive timescale
        \begin{equation}
        \,\,\Pe = 1 \quad \Rightarrow \quad 
        \begin{cases}
            \Lambda = B_0^2 / (2 \Omega \rho \mu  \kappa )\\
            \Rmod = T_0 \alpha g d  / (2 \Omega \kappa)\\
            \Rm  = \kappa / \eta 
        \end{cases}.
    \end{equation}
\end{enumerate}
The ratio between these two time scales is the Roberts number defined as
\begin{equation}
\begin{array}{r l}
    \text{Roberts:}     & q = \kappa / \eta.
    \end{array}
\end{equation}
\subsubsection{Convection settings}

We consider two scenarios to drive the temperature field in  our governing equations.
\begin{enumerate}
    \item Rayleigh--B{\'e}nard convection
    \begin{equation}
        T_0 = \tilde{T}_{+} - \tilde{T}_{-}, \quad T_+=0, \quad  T_- =1 .
    \end{equation}
    \item Internally heated convection
    \begin{equation}
        T_0 = \frac{\tilde{Q}d^2}{\rho c_p \kappa}, \quad T_+=0, \quad  T_- =0.
    \end{equation}
\end{enumerate}
$ \tilde{T}_{+} - \tilde{T}_{-}$ is the dimensional temperature difference between the two boundaries, and $\tilde{Q}$ is the dimensional volumetric heating rate such that $Q$ is equal to $\tilde{Q}$ normalised by its norm.
In Rayleigh--B\'enard, one may  consider perturbations of temperature about the conductive state so that the boundary conditions are homogeneous.

\subsection{Taylor's Constraint}
\label{sec:taycon}
\subsubsection{Mechanical setting}
The main result of \cite{taylor1963} was the demonstration of a constraint on the flow for regular solutions to the inviscid equations, $ \Ek \to 0$. Denote the Lorentz force
\begin{equation}
    \bm{F}^{\mathrm{M}} =( \nabla \times \bm{B}) \times \bm{B}.
\end{equation}
Taylor's constraint is a direct consequence of the balance of torques obtained from  the curl of the momentum balance \eqref{eq:momentum_eq} at $\Ek = 0$,
\begin{equation}
\label{eq:torque_eq}
    - \nabla \times \bm{F}^{\mathrm{M}}  =  \frac{1}{\Lambda}\left (\Rmod\, \nabla T \times \bm{e}_z    + \partial_z \bm{u} \right).
\end{equation}
In particular, there is zero torsion on geostrophic surfaces. For the spherical domain originally considered in  \cite{taylor1963}, the geostrophic surfaces are cylinders centered on the axis of rotation. 

For the plane layer, the geostrophic surfaces are the vertical extension of any simple closed plane curve $C$. More precisely, for any simple closed contractible plane curve $C$ in $[0,L_x]\times [0,L_y]$ any regular solution to equations \eqref{eq:gov_eqs} at $\Ek = 0$ must satisfy
\begin{equation}
\label{eq:taylor_constraint}
    \int_0^1 \mathrm{d}z \oint_C \bm{F}^{\mathrm{M}} \cdot \mathrm{d}l = 0.
\end{equation}
where in our notation $\mathrm{d}l = (X', Y', 0) \, \mathrm{d}s$ for a parameterization $s \mapsto (X(s), Y(s))$ of the curve $C$ by arc length $s$. To see this, first integrate the $z$ component of the torque equation \eqref{eq:torque_eq} with respect to $\mathrm{d}z$ and recall that $u_z$ vanishes at the boundaries to find
\begin{equation}
\label{eq:vert_constraint}
    \int_0^1   \bm{e}_z \cdot \nabla \times \bm{F}^{\mathrm{M}} \, \mathrm{d}z = 0.
\end{equation}
Integrating \eqref{eq:vert_constraint} with respect to $\mathrm{d}x \,\mathrm{d}y $ the planar region $A \subset [0,L_x]\times [0,L_y]$ bounded by $C$ such that $\partial A = C$, the planar Kelvin--Stokes theorem then implies \eqref{eq:taylor_constraint}. In this regard, the torque balance in equation \eqref{eq:torque_eq} is the differential version of Taylor's constraint. Our main results are derived from this torque balance.


 \subsubsection{Functional setting}
The basic framework for the study of the dynamo problem is the space of finite energy solenoidal fields. For two vector fields $\bm{v}$ and $\bm{w}$ (not necessarily solenoidal), we denote the inner product
\begin{equation}
    \langle \bm{v} , \bm{w} \rangle = \int \bm{v} \cdot \bm{w} \, \mathrm{d}V.
\end{equation}
The space $L^2_{\mathrm{sol}} = L^2_{\mathrm{sol}}(\mathrm{d}V)$  consists of finite energy solenoidal vector fields, i.e. vector fields $\bm{v}$ such that $\nabla \cdot \bm{v} = 0$ and which  are finite in the norm
\begin{equation}
    \| \bm{v} \|_{2} = \sqrt{\langle \bm{v} , \bm{v}\rangle}.
\end{equation}
With the inner product $\langle \cdot, \cdot \rangle$, the function space $L^2_{\mathrm{sol}}$ is a natural Hilbert space.

Taylor's constraint has an abstract  interpretation via Hilbert space structure. First outlined by Gallagher--Gérard-Varet  \cite{gallagher2017}, we review briefly the function analytic implications.  Denote the Coriolis operator on $L^2_{\mathrm{sol}}$
\begin{equation}
    \mathscr{C} u = \mathbb{P}_{\mathrm{sol}} ( \bm{e}_z \times \bm{u})
\end{equation}
where $\mathbb{P}_{\mathrm{sol}}$ is the Leray projector onto solenoidal vector fields. Taylor's constraint \eqref{eq:taylor_constraint} is equivalent to the condition 
\begin{equation}
\label{eq:abstay}
    \mathbb{P}_{\mathrm{sol}} \bm{F}^{\mathrm{M}} \in \mathrm{Range }\,\mathscr{C}.
\end{equation}
The operator $\mathscr{C}$ is skew-symmetric and linear. A  solution $\bm{u} \in L^2_{\mathrm{sol}}$ 
therefore admits a decomposition into magnetostrophic, Archimedean, and geostrophic parts. Each part as follows is in $L^2_{\mathrm{sol}}$, i.e.~is solenoidal and tangent to $\partial V$.
\begin{itemize}
    \item The \emph{magnetostrophic} part $\bm{u}^{\mathrm{M}}$ solves
    \begin{equation}
        \frac{1}{\Lambda } (\bm{e}_z \times \bm{u}^{\mathrm{M}} + \nabla P^{\mathrm{M}}) = (\nabla \times \bm{B} )\times \bm{B} = \bm{F}^{\mathrm{M}}. \label{eq:magpart}
    \end{equation}
        \item The \emph{Archimedean} part $\bm{u}^{\mathrm{A}}$ solves
    \begin{equation}
        \frac{1}{\Rmod } (\bm{e}_z \times \bm{u}^{\mathrm{A}} + \nabla P^{\mathrm{A}}) = T\bm{e}_z = \bm{F}^{\mathrm{A}}. \label{eq:arcpart}
    \end{equation}
            \item The \emph{geostrophic} part $\bm{u}^{\mathrm{g}}$ solves
    \begin{equation}
        \bm{e}_z \times \bm{u}^{\mathrm{g}} + \nabla P^{\mathrm{g}} = 0. \label{eq:geopart}
    \end{equation}
\end{itemize}
The solenoidal constraint on each part is enforced by their respective pressures. 

By the linearity of $\mathscr{C}$, it is clear that $\bm{u} = \bm{u}^{\mathrm{M}} + \bm{u}^{\mathrm{A}} + \bm{u}^{\mathrm{g}}$
satisfies the equation \eqref{eq:momentum_eq} in the inviscid limit $\Ek \to 0$. 
At any given instant of time,  Taylor's constraint guarantees the uniqueness of 
\begin{equation}
\bm{u}^\mathrm{M}, \bm{u}^\mathrm{A} \in   \{ \bm{v} \in L^2_{\mathrm{sol}} : \bm{v}(x,y,0) = 0\} = \mathcal{G}
\end{equation}
given their forces are sufficiently regular. The geostrophic part
\begin{equation}
    \bm{u}^{\mathrm{g}} \in \mathrm{Ker}\, \mathscr{C}
\end{equation}
is not determined by the requirement \eqref{eq:abstay} at a particular time.  While the space of particular solutions $\mathcal{G}$ is not exactly the orthogonal complement of $\mathrm{Ker}\, \mathscr{C}$, it is the case $\mathcal{G} \cap  \mathrm{Ker}\, \mathscr{C} = 0$ and our decomposition is well-defined. Taylor argues that demanding the evolution remains on the (Hilbert) manifold defined by \eqref{eq:abstay} is sufficient to determine $\bm{u}^{\mathrm{g}}$. The argument was revisited in \cite{hardy2018}  and was found to be incomplete.

By the skew-symmetry of $\mathscr{C}$, it follows that the Lorentz and Archimedean forces cannot do work on the flows they generate, nor on geostrophic flows. One verifies that
\begin{equation}
\label{eq:nsw_1}
    \langle \bm{u}^{\mathrm{M}}, \bm{F}^{\mathrm{M}} \rangle = 0, \qquad \langle \bm{u}^{\mathrm{A}}, \bm{F}^{\mathrm{A}} \rangle = 0,
\end{equation}
and
\begin{equation}
\label{eq:nsw_2}
    \langle \bm{u}^{\mathrm{g}}, \bm{F}^{\mathrm{M}} \rangle = 0, \qquad \langle \bm{u}^{\mathrm{g}}, \bm{F}^{\mathrm{A}} \rangle = 0.
\end{equation}
In the sequel, we shall use pointwise no-work requirements  derived  from  equations \eqref{eq:magpart}, \eqref{eq:arcpart} and \eqref{eq:geopart}. We refer to this procedure as the \emph{MAC decomposition}.

\section{Inertialess Dynamo Theorems}
\label{sec:bounds}


The governing equations \eqref{eq:gov_eqs} are considered both with and without the effects of viscosity.  In particular, we discover requirements for instantaneous magnetic energy non-decay. For nontrivial magnetic field $ \bm{B}$, we specify two notions:
\begin{definition}
\label{def:dynamo_action}
    \emph{Dynamo action} occurs  if
    \begin{equation*}
        \frac{\mathrm{d}}{\mathrm{d}t} \left ( \frac{1}{2}\int |\bm{B}|^2 \, \mathrm{d}V\right) \geq 0.
    \end{equation*}
\end{definition}
\begin{definition}
Magnetic flux is \emph{exact} if 
    \begin{equation*}
\int \bm{n}_{\pm} \cdot \bm{B} \, \mathrm{d}S_{\pm} = 0.
\end{equation*}
\end{definition}
\noindent The exact condition that the magnetic flux through the boundaries is net-zero is specified to simplify the bounds. It holds automatically in spherical domains.

The section contains three main results: the first is a lower bound on $\Rmod$ and $\Rm$ that does not differentiate the structure of the magnetic field, the second is a bound that does, and the third is a result where $Ek>0$.


The bounds we discover depend on the quantity
\begin{equation}
    \overline{T}(x,y,z) = \int_0^z  T(x,y, z') \, \mathrm{d}z'.
\end{equation}
\subsection{Taylor Dynamo Theory}
\label{sec:taylor_bounds}
In the first subsection, we present our requirements for dynamo action in the inviscid limit $Ek \to 0$, i.e.~the Taylor model for the Earth's dynamo. In the plane layer, the Taylor model  is
\begin{subequations}
\label{eq:inviscid_gov_eqs}
    \begin{align}
     \bm{e}_z \times \bm{u} + \nabla P  &=  \Lambda (\nabla \times\bm{B}) \times \bm{B} + \Rmod\, T \bm{e}_z, \label{eq:inviscid_momentum_eq}\\
            \nabla \cdot \bm{u} &= 0, \label{eq:inviscid_incompressibility}\\
     \partial_t \bm{B} - \nabla\times(\bm{u}\times \bm{B})  &= \frac{1}{\Rm}\nabla^2 \bm{B},
     \label{eq:inviscid_induction_eq} \\
          \nabla \cdot \bm{B}&=0,  \label{eq:inviscid_Bdiv} 
    \end{align}
\end{subequations}
and the boundary conditions  are
\begin{subequations}
\label{eq:inviscid_BCs}
\begin{align}
    \bm{u}\cdot\bm{n}\,|_{\,S_{\pm}} &= 0, \label{eq:inviscid_noslip} \\
\bm{B}\times \bm{n}\,|_{\,S_{\pm}} &=0, \label{eq:inviscid_insulating}
\end{align}
\end{subequations}
The requirements for dynamo described in this subsection are independent of the temperature model, and so we omit its description here. Taylor's constraint and the functional setting in Section \ref{sec:taycon} hold in the inviscid model.

The inviscid model \eqref{eq:inviscid_gov_eqs} admits the basic energy identity
\begin{equation}
\label{eq:basicid}
            \frac{1}{2} \frac{\mathrm{d}}{\mathrm{d}t} \int |\bm{B}|^2 \, \mathrm{d}V + \frac{1}{\Rm} \int | \nabla \times \bm{B}|^2 \, \mathrm{d}V = \frac{\Rmod}{\Lambda} \int \bm{u} \cdot T \bm{e}_z \, \mathrm{d}V.
\end{equation}
%
The production of magnetic energy  is strictly determined by the term on the right hand side. Dynamo action is impossible in the pure magnetostrophic limit $\Rmod \to 0$. Therefore, buoyancy must be significant relative to the Coriolis force for dynamo action to occur, and this ratio is measured by the modified Rayleigh number $\Rmod > 0$. 

Our main results are quantitative lower bound on $\Rmod$.

\subsubsection{Childress-type Bounds}
\label{sec:childress_bounds}

Our bounds are in terms of $L^p$ norms and a constant $A_p$. The constant $A_p$ depends on the power $p$ and interpolates between  $C_{\mathrm{S}} = A_3$ from the Sobolev embedding and $C_{\mathrm{P}} = A_{\infty }$ from the Poincaré inequality.
\begin{theorem}
\label{theorem_1} 
Consider the Taylor model, and suppose that the magnetic flux is exact. Then, for  dynamo action it is necessary that, for $3 \leq p \leq \infty $,
    \begin{equation*}
         \Rmod \, A_p \| \nabla_{x,y} \overline{T}\|_{L^p} \geq \frac{1}{\Rm}.
    \end{equation*}
\end{theorem}
See Section \ref{proof:theorem_1} for the proof.
This bound and the ones to follow are in terms of the horizontal gradient $\nabla_{x,y}$ of the field $\overline{T}$. 
In fact, $\overline{T}$ is related to the streamfunction of the Archimedean part of velocity $\bm{u}^{\mathrm{A}}$ (see Section \ref{sec:taycon}). Indeed, $\bm{u}^{\mathrm{A}}$ is an entirely horizontal flow and has the explicit form
    \begin{equation}
    \label{eq:archvel}
        \bm{u}^{\mathrm{A}} = \Rmod \,\bm{e}_z \times \nabla \overline{T} =  \Rmod \, \nabla_{x,y}^{\perp} \overline{T}.
    \end{equation}
\begin{theorem}
\label{corollary_0} 
Consider the Taylor model, and suppose that the magnetic flux is exact. 
    Then, for  dynamo action it is necessary that, for $3 \leq p \leq \infty $,
    \begin{equation*}
      A_p \| \bm{u}^{\mathrm{A}}\|_{L^p} \geq 
      \frac{1}{\Rm}.
    \end{equation*}
\end{theorem}
    In the basic energy identity, it is immediate that production of the magnetic energy depends only on the vertical (poloidal) component of the velocity;
however, our Theorems claim the significance of an entirely horizontal
(toroidal) part of the flow. Rather than conflict, the two observations
are in fact commensurate via Taylor's constraint (see Section \ref{sec:taycon}). The production of magnetic energy is exactly
\begin{equation}
    \frac{\Rmod}{\Lambda} \langle \bm{u}, \bm{F}^{\mathrm{A}}\rangle =      \frac{\Rmod}{\Lambda} \langle \bm{u}^{\mathrm{M}}, \bm{F}^{\mathrm{A}}\rangle =     \frac{1}{\Lambda} \langle \bm{u}^{\mathrm{M}}, \bm{e}_z \times \bm{u}^{\mathrm{A}} \rangle = - \langle \bm{u}^{\mathrm{A}}, \bm{F}^{\mathrm{M}}\rangle.
\end{equation}
The identity starts with the right hand side of \eqref{eq:basicid}, then uses the skew-symmetry of the Coriolis operator and properties of the MAC decomposition; 
see Section \ref{proof:theorem_1} for the proof. This identity also clarifies why our bounds are independent of $\Lambda$.

\begin{remark}
Theorem \ref{theorem_1} and Theorem \ref{corollary_0} are equivalent in statement, but not in proof. There is an explicit correction to Theorem \ref{theorem_1} when viscosity is finite.  On the other hand, Theorem \ref{corollary_0} holds verbatim in a spherical domain for the analogous  Archimedean velocity. 
\end{remark}

\subsubsection{Poloidal-Toroidal Decomposition}
\label{sec:pol_tor_result}
In this section, we decompose the magnetic field into its poloidal and toroidal components and refine our Childress-type bounds. Our refined bounds are in terms of the anisotropic norm
\begin{equation}
\| \bm{v}\|_{L^{\infty}_z L^2_{x,y}} = \sup_z \left(\int |\bm{v}(x,y,z)|^2 \, \mathrm{d}x \, \mathrm{d}y \right)^{1/2}.
\end{equation}
We state the result first and then discuss the poloidal-toroidal decomposition.
\begin{theorem}
\label{thm:main}
Consider the Taylor model, and suppose that the magnetic flux is exact.   For dynamo action, it is necessary that
  \begin{equation*}
     \Rmod\, A_2\, \| \nabla^2_{x,y}\overline{T}\|_{L^{\infty}_z L^2_{x,y}} \geq  \frac{1}{\Rm}.
  \end{equation*}
\end{theorem}
\begin{corollary}
\label{cor:main}
Under the hypotheses of Theorem \ref{thm:main},   dynamo action requires that
  \begin{equation*}
    A_2\, \| \nabla_{x,y}\bm{u}^{\mathrm{A}}\|_{L^{\infty}_z L^2_{x,y}} \geq  \frac{1}{\Rm}.
  \end{equation*}
\end{corollary}
In the plane layer, the magnetic field $\bm{B}$  decomposes into a poloidal component $\bm{B}^{\pol}$, a toroidal component $\bm{B}^{\tor}$, and a base flow component $\bm{B}^{0}$. In particular,
\begin{equation}
  \bm{B} =  \bm{B}^{\pol} + \bm{B}^{\tor} + \bm{B}^{0}.
\end{equation}
The poloidal-toroidal (PT) decomposition, in particular  its proper orthogonality  and representation formulae for $\pol$ and $\tor$, is reviewed in Appendix \ref{sec:polotoro}. By design, it is a decomposition into orthogonal subspaces of $L^2_{\mathrm{sol}}$ on which the curl operator is chiral: if $\alpha \neq \beta  \in \{ \pol, \tor, 0\}$ then
\begin{equation}
    \langle \bm{B}^{\alpha}, \bm{B}^{\beta} \rangle = 0, \qquad \langle \nabla \times \bm{B}^{\alpha},  \nabla \times \bm{B}^{\beta} \rangle = 0.
\end{equation}
In the energy identity, the ohmic dissipation term then splits into poloidal-toroidal components.
The magnetic energy production splits into  interactions (Poisson brackets) of components in the PT decomposition. Estimating the interactions results in Theorem \ref{thm:main};
the complete proof is given in Section \ref{proof:torpol_thm}.
\begin{remark}
    Unlike the Childress-type bounds, the three-dimensional Poincaré inequality is not used here. Instead, a two-dimensional Ladyzhenskaya inequality is applied on  each horizontal surface.
\end{remark}

\subsection{Viscous Dynamo Theory}
\label{sec:viscosity}

Inertialess dynamos in the plane layer have only been studied numerically with finite viscosity  due to  computational accessibility \cite{cattaneo2017dynamo,hughes2016strong,hughes2019force}. With  finite viscosity $Ek>0$,  there is classical local regularity theory for the dynamo model \eqref{eq:gov_eqs}: local-in-time smooth solutions exist for smooth initial conditions, Taylor's constraint notwithstanding. Though the set of possible solutions is larger, it is not immediate that the threshold for dynamo action is lower--viscosity introduces dissipation into the dynamics. 

 Our quantitative  antidynamo theory is robust to finite viscosity $Ek>0$. Our condition in terms of the temperature holds with viscous corrections 
 \begin{theorem}
\label{thm:visc_main}
Fix $\Ek > 0$ and consider the viscous inertialess model. Suppose that the magnetic flux is exact.
    Let $3 \leq p \leq \infty $ and $0< \gamma < 1$. Then, for  dynamo action it is necessary that 
    \begin{equation*}
         \Rmod \, A_p \| \nabla_{x,y} \overline{T}\|_{L^p} +  \Rmod^2 \, \Ek^{1 - \gamma}\| \nabla \nabla_{x,y}\overline{T}\|^2_{L^2} \, \geq \frac{1}{\Rm}
    \end{equation*}
    or 
        \begin{equation*}
4\Lambda \int | \nabla \bm{B}|^2 \, \mathrm{d}V \leq \Ek^{\gamma}.
 \end{equation*}
\end{theorem}

The statement holds for any real number $\gamma$. However,  for  $ 0< \gamma< 1$ in particular the  result is nontrivial in the inviscid limit: both exponents $\gamma$ and $1 - \gamma$ are positive such that Theorem \ref{theorem_1} is recovered as $\Ek \to 0$. Our refinement via the poloidal-toroidal decomposition holds with  identical corrections for finite viscosity.
 

Our condition for dynamo action in terms of Archimedean fluid velocity holds also for the viscous inertialess model. For $\Ek > 0$, the Archimedean fluid velocity $\bm{u}^{\mathrm{A}}$  is the solution to 
    \begin{equation}
        - \Ek \nabla^2 \bm{u}^{\mathrm{A}} + \bm{e}_z \times \bm{u}^{\mathrm{A}} + \nabla P^{\mathrm{A}} = \Rmod \, T\bm{e}_z \label{eq:arcpart_v}, \quad \nabla \cdot \bm{u}^{\mathrm{A}} = 0,
    \end{equation}
    subject to no-penetration and stress-free boundary conditions; $\bm{u}^{\mathrm{A}}$  is guaranteed to be unique and smooth for $\Ek > 0$, with continuous dependence on $T$.
\begin{theorem}
\label{corollary_0_visc} 
Fix $\Ek > 0$ and consider the viscous inertialess model. Suppose that the magnetic flux is exact. Then, for  dynamo action it is necessary that, for $3 \leq p \leq \infty $,
    \begin{equation*}
      A_p \| \bm{u}^{\mathrm{A}}\|_{p} \geq \frac{1}{\Rm}.
    \end{equation*}
\end{theorem}


Our final result for the inertialess viscous model concerns the magnetohydrodynamic evolution of the electric current 
\begin{equation}
    \bm{J}=\nabla \times \bm{B}.
\end{equation}
Magnetic enstrophy $\int |\nabla \times \bm{B}|^2 \, \mathrm{d}V$ is the energy of the current, and is proportional to ohmic dissipation. Differentiating the induction equation \eqref{eq:inviscid_induction_eq} results in 
\begin{equation}
\label{eq:J_gov_eq}
    \partial_t \bm{J}  - \frac{1}{Rm} \nabla^2 \bm{J} = \nabla (\nabla \cdot ( \bm{u} \times \bm{B})) - \nabla^2(\bm{u} \times \bm{B}) .
\end{equation}
In the  current energy identity, the gradient term vanishes because $\bm{J}$ is divergence-free. For ferromagnetic boundary conditions, sufficiently smooth solutions satisfy
\begin{equation}
\label{eq:J_energy_id}
    \frac12 \frac{\mathrm{d}}{\mathrm{d}t} \int |\bm{J}|^2\mathrm{d}V + \frac{1}{Rm}\int |\nabla \bm{J}|^2 \mathrm{d}V = -\int \bm{J} \cdot \nabla^2 (\bm{u} \times \bm{B}) \mathrm{d}V. 
\end{equation}
We identify a region of parameters in solution space where current dissipation necessarily dominates all  mechanisms of (magnetic) enstrophy production.


\begin{theorem}
\label{theorem:ohmic}
Fix $\Ek > 0$ and consider the viscous inertialess model. Suppose that the magnetic flux is exact and that the temperature satisfies isothermal boundary conditions. Let $0 < \gamma < 1$, and suppose that
    \begin{equation*}
        Rm < \frac{2Ek^{1-\gamma}}{C_m^2C_{\mathrm{P}}^2 }.
    \end{equation*}
Then, for  \emph{magnetic enstrophy growth},
    \begin{equation*}
        \frac{\mathrm{d}}{\mathrm{d}t}\left(\frac12  \int |  \bm{J}|^2\mathrm{d}V \right) \geq 0,
    \end{equation*}
    it is necessary that 
        \begin{equation*}
        C_n   \Rmod \,\|\overline{\nabla \times \nabla \times T \bm{e}_z} \|_{L^2} + Ra^2 Ek^{1-\gamma} \| \nabla \times \overline{\nabla \times \nabla \times T \bm{e}_z} \|_{L^2}^2  \geq \frac{1}{2 \Rm},
    \end{equation*}
or it does not hold
    \begin{equation*}
        Ek^{\gamma} \leq 2\Lambda\int |\nabla \bm{J}|^2 \mathrm{d}V \leq Ek^{\gamma}\left(\frac{4}{C_m^2 C_{\mathrm{P}}^2} \frac{1}{Rm Ek^{\gamma - 1}} - 1\right) .
    \end{equation*}
\end{theorem}
The result relies on a nonlinear interaction between the ferromagnetic boundary conditions for $\bm{B}$ and the stress-free boundary conditions for $\bm{u}$. In fact, the current $\bm{J}$ here satisfies an analogue of the stress-free condition.

\begin{remark}
    The smallness condition on $\Rm$ is used to control a fluid enstrophy term. If one forgoes this control, one can fix $\Rm > 0$ and take $\Ek \to 0$ such that, in the inviscid limit, the first magnetic enstrophy growth  requirement is 
    \begin{equation}
    \label{eq:ohmic_taylor}
        C_m \| \nabla  \bm{u}\|_{L^2} + C_n   \Rmod \,\|\overline{\nabla \times \nabla \times T \bm{e}_z} \|_{L^2} \geq \frac{1}{\Rm}.
    \end{equation}
\end{remark}
\begin{remark}
Limited enstrophy growth does not guarantee dynamo action.
\end{remark}

 The proofs of the results in this subsection are given in Section \ref{sec:proof_viscous_results}. 

\section{On Convection Driven Dynamo}
\label{sec:temp_classical_R}
In the preceding section, dynamo conditions are presented in terms of the modified Rayleigh number $\Rmod$ and are  independent of the model for temperature $T$. In this section, we use the temperature equation to obtain constraints on the nondimensional parameters for sustained dynamo action. 

To match the convection literature, the conditions here shall be reported in terms of ``classical'' Rayleigh numbers. 
Classical here refers to a Rayleigh number defined strictly as a ratio between heating and diffusion due to material properties. The modified Rayleigh number $\Ra$ contains the rotation rate, and for the case where we permit viscosity and $Ek>0$, would be coupled to $Ek$ and so it is appropriate to separate the two. 

In this section, we shall analyze coupling of the inertialess force balance with basic temperature dynamics. The temperature equation is
\begin{equation}
\label{eq:temp2}
    \partial_t T + \nabla \cdot (\bm{u} T) = \frac{1}{\Pe}\nabla^2 T +  \frac{Q}{\Pe}.
\end{equation}
Recall the Peclet number $\Pe$ measures the ratio of thermal advection to thermal diffusion. In the inviscid case $\Ek = 0$, we report conditions for dynamo action for absorbing balls in the long-time dynamics of \eqref{eq:temp2}. In the viscous case, we report requirements for (instantaneous) dynamo action for small $\Pe$.
\subsection{Inviscid Case}
In its long time dynamics, we identify an absorbing ball for the critical quantity necessary for (instantaneous) dynamo action. The radius of the absorbing ball is expressed in terms of the long-time quantity
\begin{equation}
    U_{\infty} = C_{\mathrm{S},0} \limsup_{t \to \infty} \| \bm{u}\|_{L^3} +  C_{\infty} \limsup_{t \to \infty } \| \nabla \bm{u}\|_2
    \label{eq:Mdef}
\end{equation}
where $C_\infty$ and $C_{\mathrm{S},0}$ are positive constants from Sobolev embeddings.
\begin{theorem}
\label{lemma:nablaT_L3}
 Suppose $T$ satisfies the temperature equation and homoegenous isothermal boundary conditions.  
 Suppose that $Q \in H^1$ is time independent and let
 \begin{equation*}
     U_\infty Pe < \frac{\pi}{\sqrt{1+\pi^2}},
 \end{equation*}
then,
\begin{equation*}
     \limsup_{t \to \infty}\| \horg \overline{T}(t) \|_{L^3} \leq  R_3 : = \frac{ b_3 \,C_{\mathrm{S}}^{1/2}\|Q \|^{1/4}_2}{\sqrt{\pi}}\left(  \frac{ U_{\infty}\Pe\,  \| Q\|_{H^1}}{  \pi  - \sqrt{1+\pi^2}  U_{\infty} \Pe  } + \| Q\|_2 \right)^{3/4}.
\end{equation*}
\end{theorem}

The results of \S\ref{sec:taylor_bounds} are stated without specifying a timescale and information from the temperature equation. We start by specifying a more classical  Rayleigh number to separate the parameters responsible for heating from rotation in the inviscid and viscous problems.
We define the nondimensional parameters:
\begin{equation}
\label{eq:R_classical_taylor_def}
\renewcommand{\arraystretch}{1.8} 
\begin{array}{r l}
\text{Rayleigh:} &  \Rmod_\eta = g \alpha T_0 d^3/(\eta \kappa) \\
\text{magnetic Ekman:}     & \Ek_\eta = \eta/(2 \Omega d^2).
\end{array}
\end{equation}

The relation to the parameters \eqref{eq:nondim_nos} is given by
   \begin{equation}
        \label{eq:R_classical_taylor}
      \Rmod_\eta\, Ek_\eta = \Pe\, \Rmod.
   \end{equation}
Theorem \ref{theorem_1}, can be recast in terms of this more standard Rayleigh number $\Rmod_\eta$ which is independent of the rotation rate.

\begin{corollary}
\label{cor1}
Consider the Taylor model, and fix a uniform heat source $Q=1.$ If $T$ is in the absorbing ball defined by Theorem \ref{lemma:nablaT_L3}, then dynamo action requires
\begin{equation*}
    \Rmod_\eta \geq\frac{2^{4/3}\pi^{3/2}}{3\sqrt{3}  C^{3/2}_{\mathrm{S}} b_3 }  \left( \frac{U_{\infty} \Pe }{\pi   - \sqrt{1+\pi^2} U_{\infty} \Pe  } +1 \right)^{-3/4}    \frac{1}{q \,\Ek_\eta} .
\end{equation*}
\end{corollary}
\begin{remark}
    Formally, if $U_{\infty}\Pe \ll1$ then the requirement is
\begin{equation}
     \Rmod_\eta \geq \left(\frac{2^{4/3}\pi^{3/2}}{3\sqrt{3}C_{\mathrm{S}}^{3/2} b_3}\right) \frac{1}{q \,\Ek_\eta}.
\end{equation}
\end{remark}

Next, we reinterpret the result from the poloidal-toroidal decomposition of the magnetic dissipation and Lorentz force presented in \S\ref{sec:pol_tor_result}. In this case for the mix-norm we do not have to use Theorem \ref{lemma:nablaT_L3}. Instead, the critical quantity has a simpler absorption radius
\begin{theorem}
\label{lemma:poltorT}
 Suppose $T$ satisfies the temperature equation and isothermal boundary conditions.  
 Suppose that $Q\in H^1$ is time independent and let
 \begin{equation*}
     U_\infty Pe < \frac{\pi}{\sqrt{1+\pi^2}},
 \end{equation*}
 then,
\begin{equation*}
     \limsup_{t \to \infty}\| \hord \overline{T}(t) \|_{L^\infty_z L^2_{x,y}} \leq  R_2 : =  \left(  \frac{ U_{\infty}\Pe\,  \| Q\|_{H^1}}{  \pi  - \sqrt{1+\pi^2}  U_{\infty} \Pe  } + \| Q\|_{L^2} \right).
\end{equation*}
\end{theorem}

\begin{corollary}
\label{cor3}
Consider the Taylor model, and fix a uniform heat source $Q=1.$ If $T$ is in the absorbing ball defined by Theorem \ref{lemma:poltorT}, then dynamo action requires
    \begin{equation*}
        \Rmod_\eta \geq \frac{ C_{\mathrm{P}}^{1/2}}{ C_{\mathrm{S}}^{3/2} } \left(\frac{U_{\infty} \Pe }{\pi  - \sqrt{1+\pi^2} U_{\infty} \Pe}+ 1\right)^{-1} \frac{1}{q \,\Ek_\eta}.
    \end{equation*}
\end{corollary}
\begin{remark}
        Formally, if $U_{\infty}\Pe \ll1$, the requirement is
    \begin{equation}
        \Rmod_\eta \geq \left(\frac{C^{1/2}_P}{C_{\mathrm{S}}^{3/2}} \right)\frac{1}{q \,\Ek_\eta} .
    \end{equation}
\end{remark}


\subsection{Viscous Case}
For the viscous model, the classical Rayleigh $\Rmod_\nu$ number often used in  studies of convection provides a natural decoupling of heating from the rotation rate. Explicitly
    \begin{equation}
\label{eq:OG_rayleigh}
\renewcommand{\arraystretch}{1.8} 
\begin{array}{r l}
\text{Rayleigh:} &  \Rmod_\nu = g \alpha T_0 d^3/(\nu \kappa)
\end{array}
\end{equation}
The relation to the parameter \eqref{eq:nondim_nos} is given by 
    \begin{equation}
    \label{eq:OG_rayleigh_relation}
        \Rmod_{\nu} \,Ek = \Rmod \,\Pe .
    \end{equation}
    In terms of the classical Rayleigh number, our constraint for (instantaneous) dynamo action has the alternative form
    \begin{theorem}
    \label{thm:viscous_classical}
        Fix $\Ek > 0$ and consider the viscous inertialess model. Suppose that the magnetic flux is exact. Then, for  dynamo action it is necessary that, for $0< \gamma < 1 $, 
        \begin{equation*}
            \Rmod_\nu \| \nabla^2 T\|_{L^2} \geq \frac{A_2}{2 C^2_{\mathrm{P},2}} \Pe \,\Ek^{-2+ \gamma } \left( -1 + \sqrt{1 + \frac{4 C^2_{\mathrm{P},2} \Ek^{1-\gamma}}{A_2^2 \Rm}} \right).
        \end{equation*}
    \end{theorem}
     \begin{remark}
     \label{rem:scaling1}
        The choice $\gamma = \frac12$ is natural, since in the original Theorem \ref{thm:visc_main}, both constraints tend to the inviscid limit at the same rate in $Ek$.
        For $\delta >0$, suppose that
        \begin{equation}
            \Rm \leq \frac{4 C^2_{\mathrm{P},2}}{A_2^2  (2+\delta) \delta} Ek^{1/2}  ,
            \label{eq:cond_viscous_final_res}
        \end{equation}
        then the requirement for dynamo action is
        \begin{equation}
            Ra_\nu \| \nabla^2 T \|_{L^2} \geq \frac{A_2 \delta}{2C^2_{P,2}} Pe \,Ek^{-3/2} .
        \end{equation}
        If  \eqref{eq:cond_viscous_final_res} is an equality and since $Rm=q\,Pe $, the requirement for dynamo action is
        \begin{equation}
            Ra_\nu \|\nabla^2 T \|_2 \geq \frac{2}{A_2(2+\delta)} q^{-1} Ek^{-1}.
        \end{equation}
        Note that these constraints are instantaneous and are independent of the dynamics of temperature $T$. See \S\ref{sec:visccombined} for the proof of \cref{thm:viscous_classical}. A long-time prediction  can be obtained by referring to the temperature equation and substituting  the radius of the absorbing ball for $\|\nabla^2 T \|_2$ identified in Proposition \ref{lemma:pointwise_t_laplace_T} into Theorem \ref{thm:viscous_classical}.
     \end{remark}

A significant feature is that the parameter $\Rmod_{\nu}$ has a critical value such that the conductive state of the system is linearly unstable. If $Q = Q(z)$, the conductive state $\tau = \tau(z)$ is given by 
\begin{equation}
    \label{eq:conductive_temperature_problem}
    \tau''+Q=0,\qquad \tau \,|_{\,S_{\pm}} = T_{\pm}.
\end{equation}
It is only perturbation about the conductive state $\theta = T- \tau$ which may drive dynamo action. Indeed,  $\theta = 0$ on the boundaries $S_{\pm}$ and it holds $\horg \overline{\theta}=\horg \overline{T}$; the magnetic energy production term has the form
\begin{equation}
\int u_zT\,\mathrm{d}V=\int u_z\theta\,\mathrm{d}V.
\label{eq:combined_energy_conductive_reduction}
\end{equation}

\begin{theorem}
\label{thm:viscous_combined_energy}
Fix $\Ek>0$ and consider the viscous inertialess model. Suppose that the magnetic flux is exact and let $\tau$ be the conductive state for heating $Q=Q(z)$ and $T_{\pm}$.
For non-trivial fields, combined energy growth,
\begin{equation*}
\frac{\mathrm{d}}{\mathrm{d}t}\left( \frac{\Lambda}{2}\int |\bm{B}|^2 \, \mathrm{d}V + \frac{\Pe}{2} \int |\nabla \theta |^2 \, \mathrm{d}V\right)\geq0,
\label{eq:relative_combined_energy_growth}
\end{equation*}
requires that, for $ 3\leq p \leq \infty $
\begin{equation*}
Ra\,A_p \,\|\nabla_{x,y}\overline{T}\|_{L^p}\geq \frac{1}{\Rm},
\label{eq:first_combined_energy_condition}
\end{equation*}
or it holds
\begin{equation*}
\frac{C_{\mathrm{P},2}\,Ra\sqrt{Ek}}{2}+\frac{\Pe \,\|\tau'\|_{L^2_z}}{2\sqrt{Ek}}\geq  \sqrt{1-\Pe\, C_{\mathrm{S},0}\|\bm{u}\|_{L^3}},
\label{eq:second_combined_energy_condition}
\end{equation*}
provided $Pe C_{S,0} \| \bm{u} \|_{L^3} < 1$.
\end{theorem}
See \S\ref{sec:visccombined} for a proof of Theorem \ref{thm:viscous_combined_energy}. 
\begin{remark}
\label{rem:scaling2}
    Formally, if $\| \bm{u}\|_{L^3}\Pe \ll 1 $, the second requirement is
    \begin{equation}
        C_{\mathrm{P},2}\,Ra\sqrt{Ek} +\frac{\Pe \,\|\tau'\|_{L^2_z}}{\sqrt{Ek}}\geq  2.
        \label{eq:smallPereq}
    \end{equation}
\end{remark}
Notably, in the case of small $\Pe$, combined energy growth is essentially dynamo action. In terms of the classical Rayleigh number $\Rmod_{\nu} $ defined in \eqref{eq:OG_rayleigh}, the condition for energy growth \eqref{eq:smallPereq} at small $\Pe$ is
    \begin{equation}
        \frac{C_{\mathrm{P},2}\,Ra_\nu\, Ek^{3/2}}{ Pe} + \frac{Pe\|\tau'\|_{L^2_z}}{ \sqrt{Ek}}
         \geq 2.
         \label{eq:Combined_energy_second_classical_R}
    \end{equation}

Let us consider formal asymptotics for small $\Pe$ in the inviscid limit. If we posit $Pe \sim  Ek^n$, then up to independent constants the above expression becomes
    \begin{equation}
\label{eq:Combined_energy_second_classical_R2}
            Ra_\nu \gtrsim  Ek^{-1}( Ek^{n-\frac12} - \|\tau' \|_2 Ek^{2n-1}).
    \end{equation}
    If $\frac12 < n < \frac32$,
        both exponents $n - \frac{1}{2}$ and $2n - 1$ are positive; thus in the inviscid limit $\Ek \to 0$, the requirement \eqref{eq:Combined_energy_second_classical_R2} imposes a nontrivial restriction for combined energy growth, i.e., dynamo action. In particular, one has $\Rmod_\nu \gtrsim \Ek^{n-3/2 }$. In the critical case, $n=\frac12$, the asymptotic restriction is $Ra_\nu \gtrsim Ek^{-1}(1-\|\tau'\|_2)$.

\begin{remark}
    One can predict a scaling of $Ek^{-3/2}$ for $\Rmod_\nu$ without an  asymptotic restriction on  $Pe > 0$. If $\| \tau' \|_{L^2_z} \Ek^{-1/2} \ll 1$, then the condition in Theorem \ref{thm:viscous_combined_energy} is 
    \begin{equation}
        Ra_\nu \gtrsim  Ek^{-3/2} Pe (1- C_{\mathrm{S},0} Pe \|\bm{u}\|_{L^3})^{1/2}.
        \label{eq:Ek32_scaling}
    \end{equation}
    Sufficiently oscillatory or concentrated $Q$ satisfy this condition. Explicit examples are $Q(z) =  \sin(   \pi z / Ek^{1/2 + \epsilon}  )$ or $ Q(z)= \bm{1}_{[0,Ek^{1+\epsilon}]} Ek^{-(1+\epsilon)} $ for any $\epsilon>0$.
\end{remark}

\section{Proof of Theorems: Inviscid Case}
Our antidynamo Theorems rely on functional identities, decompositions and  inequalities. The analysis begins with the fact that our governing equations admits a basic (magnetic) energy identity
\begin{lemma}[Magnetic Energy Balance]
\label{prop:1}
Consider the Taylor model \eqref{eq:inviscid_gov_eqs} and suppose that the nondimensional parameters are all finite. Then for a smooth solution,
    \begin{equation*}
    \frac{1}{2} \frac{\mathrm{d}}{\mathrm{d}t} \int |\bm{B}|^2 \, \mathrm{d}V  + \frac{1}{\Rm}\int |\nabla \bm{B}|^2 \mathrm{d}V= \frac{Ra}{\Lambda} \int u_z T \mathrm{d}V .
\end{equation*}
\end{lemma}
\begin{proof}
Applying $\langle \bm{B}, \cdot\, \rangle$ to the induction equation \eqref{eq:inviscid_induction_eq} and integrating by parts results in
\begin{equation}
    \label{eq:induct_energy_0}
    \frac{1}{2} \frac{\mathrm{d}}{\mathrm{d}t} \int |\bm{B}|^2 \mathrm{d}V  + \frac{1}{\Rm}\int |\nabla \bm{B}|^2 \mathrm{d}V= \langle \bm{u}, \bm{B}  \times (\nabla \times \bm{B}) \rangle .
\end{equation}
Indeed, integrating by parts the ohmic dissipation term yields
\begin{equation}
\langle \bm{B}, \nabla^2 \bm{B} \rangle + \int \bm{n}_{\pm} \cdot
(\nabla \times \bm{B}) \times  \bm{B}\, \mathrm{d} S_{\pm} = - \int | \nabla \bm{B}|^2\, \mathrm{d}V ,
\end{equation}
and the boundary term vanishes because \eqref{eq:inviscid_insulating} implies $\bm{B}$ is normal to the boundary. Integrating by parts the production term yields
\begin{equation}
\langle \bm{B}, \nabla \times ( \bm{u} \times \bm{B}) \rangle +  \int \bm{n}_{\pm} \cdot \bm{B} \times ( \bm{u} \times \bm{B})\, \mathrm{d}{S_{\pm}} =  \langle \bm{u}, \bm{B} \times  (\nabla \times \bm{B}) \rangle.
\end{equation}
and the boundary term vanishes because \eqref{eq:inviscid_noslip} implies $\bm{u}$ is tangent to the boundary.

Applying $\langle \bm{u}, \cdot\, \rangle$ to the momentum equation \eqref{eq:inviscid_momentum_eq} results in
    \begin{equation}
    \label{eq:mom_energy}
\langle \bm{u} ,  \bm{B} \times (\nabla \times \bm{B})  \rangle     = \frac{Ra}{\Lambda}  \int u_z T \mathrm{d}V. 
\end{equation}
%
Substituting \eqref{eq:mom_energy}  into \eqref{eq:induct_energy_0} produces the desired identity.
\end{proof}

\subsection{Childress-type Bounds}
\label{proof:theorem_1}

\begin{proof}[\underline{Proof of Theorem \ref{theorem_1}}]
The curl of the momentum equation is 
\begin{equation}
  \label{eq:curlmom_1}
  -\partial_z \bm{u} = \Lambda\, \nabla \times \bm{F}^{\mathrm{M}} -
  \Rmod \, \bm{e}_{z} \times \nabla T,
\end{equation}
where we recall the notation
\begin{equation}
\bm{F}^{\mathrm{M}} = (\nabla \times  \bm{B}) \times \bm{B}.
\end{equation}
Integrating by parts in $\mathrm{d}z$
\begin{equation}
\begin{aligned}  
   \int u_{z} T \, \mathrm{d}V &= - \int \partial_{z}u_z  \left( \int_0^{z} T(x,y,z')
\, \mathrm{d}z' \right)\, \mathrm{d}x \,
  \mathrm{d}y \,  \mathrm{d}z, 
\end{aligned}
\end{equation}
 the boundary terms vanished because $u_{z} = 0$ on the boundaries.
Therefore,  after substituting \eqref{eq:curlmom_1}, we have the
expression
\begin{equation}
\label{eq:ibpdz_2}
 \int u_{z} T \, \mathrm{d} V =  -\langle
\partial_{z} \bm{u}, \overline{T} \bm{e}_z \rangle 
= \Lambda \langle
\nabla \times \bm{F}^{\mathrm{M}},  \overline{T} \bm{e}_{z}
\rangle.
\end{equation}
Integrating by parts the right hands side yields
\begin{equation}
\label{eq:ibpprod}
 \langle
\nabla \times \bm{F}^{\mathrm{M}},  \overline{T} \bm{e}_{z}
\rangle + \int \bm{n}_{\pm} \cdot \bm{F}^{\mathrm{M}}  \times \overline{T} \bm{e}_{z} \, \mathrm{d}S_{\pm} = - \langle
\bm{F}^{\mathrm{M}}, \bm{e}_{z} \times \nabla \overline{T} \rangle,
\end{equation}
where the boundary term vanishes. Observe, we have simply integrated by parts
the perpendicular horizontal gradient
\begin{equation}
  \bm{e}_{z} \times \nabla f = - \partial_{y} f \bm{e}_{x} +
  \partial_{x} f \bm{e}_{y} = \nabla^{\perp}_{x,y} f.
\end{equation}
Substituting \eqref{eq:ibpprod} into the expression \eqref{eq:ibpdz_2},
we discover a new identity for the production of magnetic energy
\begin{equation}
\label{eq:newprod}
  \frac{\Rmod}{\Lambda} \int T u_{z} \, \mathrm{d}V = -
  \Rmod \int
  \nabla^{\perp}_{x,y} \overline{ T } \cdot (\nabla \times \bm{B}) \times
  \bm{B} \, \mathrm{d} V.
\end{equation}
To conclude, we require
\begin{lemma}
\label{lemma:production1}
Let a magnetic field $\bm{B}$ satisfy ferromagnetic boundary
conditions and
\begin{equation*}
\int \bm{n}_{\pm} \cdot \bm{B} \, \mathrm{d}S_{\pm} = 0.
\end{equation*}
 Then  for $3 \leq p \leq \infty$ there is a constant  $A_{p}$  such that it holds
\begin{equation*}
   \left| \int \bm{v} \cdot (\nabla \times \bm{B}) \times \bm{B} \, \mathrm{d}V \right|
   \leq A_{p} \| \bm{v} \|_{L^p} \int |\nabla \bm{B}|^2 \mathrm{d}V.
\end{equation*}
\end{lemma}
Applying Lemma \ref{lemma:production1} yields
\begin{equation}
\label{eq:inviscidest}
\begin{aligned}
  \Rmod \int
  \nabla^{\perp}_{x,y} \overline{ T } \cdot (\nabla \times \bm{B}) \times
  \bm{B} \, \mathrm{d} V - &\frac{1}{\Rm} \int |\nabla \bm{B}|^2 \,
  \mathrm{d}V \\
  &
  \leq \left( A_{p} \Rmod\,  \| \nabla_{x,y} \overline{T} \|_{L^{p}} -
  \frac{1}{\Rm} \right) \int |\nabla \bm{B}|^2 \, \mathrm{d}V.  
\end{aligned}
\end{equation}
By Lemma \ref{prop:1}, our identity \eqref{eq:newprod} and estimate
\eqref{eq:inviscidest}, we have
thus
established
\begin{equation}
  \frac{1}{2} \frac{\mathrm{d}}{\mathrm{d}t} \int |\bm{B}|^2 \,
  \mathrm{d}V \leq  \left( A_{p} \Rmod\,  \| \nabla_{x,y} \overline{T} \|_{L^{p}} -
  \frac{1}{\Rm} \right) \int |\nabla \bm{B}|^2 \, \mathrm{d}V.  
\end{equation}
Dynamo
action therefore requires that the expression in the parentheses above is
nonnegative.
\end{proof}
\begin{proof}[\underline{Proof of Lemma \ref{lemma:production1}}]
By H{\"o}lder's inequality with exponents
\begin{equation}
\frac{1}{p} + \frac{1}{r} = \frac{1}{2},
\end{equation}
it holds
\begin{equation}
\label{eq:holderestimate}
\int \bm{v} \cdot (\nabla \times \bm{B}) \times \bm{B} \, \mathrm{d}V \leq \| \bm{v}\|_{L^p} \| \bm{B} \|_{L^r} \| \nabla \times \bm{B} \|_{L^2}.
\end{equation}
Because $3 \leq p \leq \infty$, it follows that $2 \leq  r \leq 6$ are the endpoints. We then bound $L^r$ by an interpolation of the endpoint norms
    \begin{equation}
    \label{eq:interpestimate}
        \| \bm{B} \|_{L^r} \leq \| \bm{B} \|^{\frac{6-r}{2r}}_{L^2} \| \bm{B} \|^{\frac{3r-6}{2r}}_{L^6}.
    \end{equation}

With ferromagnetic boundary conditions, the kernel of the curl operator in the space of solenoidal vector fields are exactly constant vectors in the $z$ direction. Because we assumed  $\int B_z \,\mathrm{d}S_{\pm} = 0 $, these constant vectors must be zero in our case and the kernel is trivial. Therefore,  from the Poincaré inequality it holds
\begin{equation}
    \label{eq:poincareestimate}
\| \bm{B}\|_{L^2} \leq C_{\mathrm{P}} \| \nabla \times \bm{B} \|_{L^2},
\end{equation}
and from the Sobolev embedding it holds
\begin{equation}
    \label{eq:sobolevestimate}
\| \bm{B}\|_{L^6} \leq C_{\mathrm{S}} \| \nabla \times \bm{B} \|_{L^2}.
\end{equation}

Combining \eqref{eq:poincareestimate} and \eqref{eq:sobolevestimate} with \eqref{eq:interpestimate}, we have established that
    \begin{equation}
    \label{eq:interpestimatenew}
        \| \bm{B} \|_{L^r} \leq C_{\mathrm{P}}^{\frac{6-r}{2r}} C_{\mathrm{S}}^{\frac{3r-6}{2r}}  \| \nabla \times \bm{B}\|_{L^2}.
    \end{equation}
    From the boundary conditions on $\bm{B}$ we have that $\| \nabla \times \bm{B}\|_2 = \| \nabla \bm{B}\|_2$.
    Applying \eqref{eq:interpestimatenew} to \eqref{eq:holderestimate}, we have proven the claim for the constant
    \begin{equation}
    A_{p} = C_{\mathrm{P}}^{\frac{p-3}{p} } C_{\mathrm{S}}^{\frac{3}{p}}.
    \end{equation}
\end{proof}
\begin{proof}[\underline{Proof of Theorem \ref{corollary_0}}]
Taylor's constraint implies
\begin{equation}
    \frac{\Rmod}{\Lambda} \langle \bm{u}, \bm{F}^{\mathrm{A}}\rangle =      \frac{\Rmod}{\Lambda} \langle \bm{u}^{\mathrm{M}}, \bm{F}^{\mathrm{A}}\rangle =     \frac{1}{\Lambda} \langle \bm{u}^{\mathrm{M}}, \bm{e}_z \times \bm{u}^{\mathrm{A}} \rangle = - \langle \bm{u}^{\mathrm{A}}, \bm{F}^{\mathrm{M}}\rangle.
\end{equation}
The first equality follows from the no work properties \eqref{eq:nsw_1} and \eqref{eq:nsw_2} derived from the skew-symmetry of the Coriolis operator. The penultimate and last equality follow from equations \eqref{eq:magpart} and \eqref{eq:arcpart}.

We have discovered an identity for the production of magnetic energy
\begin{equation}
\label{eq:newprod2}
  \frac{\Rmod}{\Lambda} \int T u_{z} \, \mathrm{d}V = - 
   \int
  \bm{u}^{\mathrm{A}} \cdot (\nabla \times \bm{B}) \times
  \bm{B} \, \mathrm{d} V.
\end{equation}
Lemma \ref{prop:1} and Lemma \ref{lemma:production1} then imply
\begin{equation}
  \frac{1}{2} \frac{\mathrm{d}}{\mathrm{d}t} \int |\bm{B}|^2 \,
  \mathrm{d}V \leq  \left( A_{p} \,  \| \bm{u}^{\mathrm{A}} \|_{L^{p}} -
  \frac{1}{\Rm} \right) \int |\nabla \bm{B}|^2 \, \mathrm{d}V.  
\end{equation}
Therefore, dynamo
action requires that the expression in the parentheses above is
nonnegative.
\end{proof}
\subsection{Poloidal-Toroidal Decomposition}
\label{proof:torpol_thm}

The basic properties and representation formula for the PT decomposition  are reviewed in Appendix \ref{sec:polotoro}. In particular, any magnetic field in the plane layer is given by
\begin{equation}
\bm{B} = \sum_{\alpha \in \{ \pol, \tor, 0\}}\bm{B}^{\alpha} ,
\end{equation}
where the poloidal and toroidal parts are
\begin{equation}
    \bm{B}^{\pol}  = \nabla \times \nabla \times  \pol \bm{e}_{z}, \qquad \bm{B}^{\tor} = \nabla \times
  \tor \bm{e}_{z},
\end{equation}
and the base flow is
\begin{equation}
    \bm{B}^0 = B^0_x (z) \, \bm{e}_x + B^0_y (z) \, \bm{e}_y.
\end{equation}
The poloidal potential $\pol$, toroidal potential $\tor$ and the horizontal base flow $\bm{B}^0$ have explicit representations.

\begin{proof}[\underline{Proof of Theorem \ref{thm:main}}]
Because for $\alpha \neq \beta$
\begin{equation}
    \langle \nabla \times \bm{B}^{\alpha},  \nabla \times \bm{B}^{\beta} \rangle = 0,
\end{equation}
it follows that
\begin{equation}
\label{eq:ptod}
    \int | \nabla \times \bm{B}|^2 \, \mathrm{d}V = \sum_{\alpha \in \{ \pol, \tor, 0\}}\int | \nabla \times \bm{B}^{\alpha }|^2 \, \mathrm{d}V.
\end{equation}
From integrating by parts in $z$ and the curl of the momentum equation \eqref{eq:curlmom_1},
    \begin{equation}
    \label{eq:ptprodid}
        \frac{\Rmod}{\Lambda} \int T u_{z} \, \mathrm{d}V=     \Rmod 
   \int \overline{T}\, 
 \bm{e}_z \cdot  \nabla \times \bm{F}^{\mathrm{M}} \, \mathrm{d} V.
    \end{equation}
    where we recall the notation
        \begin{equation}
        \bm{F}^{\mathrm{M}} = (\nabla \times  \bm{B}) \times \bm{B}.
    \end{equation}
    We shall show that the ferromagnetic boundary conditions imply the estimate
     \begin{equation}
     \label{eq:ptest}
     \int f \bm{e}_{z} \cdot \nabla \times \bm{F}^{\mathrm{M}} \, \mathrm{d}V \leq 
    \sum_{\alpha \in \{ \pol, \tor,0 \}} h_{\alpha }(\| \hord f\|_{L^\infty_z L^{2}_{x,y}})\int | \nabla \times \bm{B}^{\alpha}|^2 \, \mathrm{d}V,
 \end{equation}
 where the factors $h_\alpha = h_\alpha(\chi)$ are 
 \begin{equation}
         h_{\pol}(\chi ) =   \frac{3C_1}{2} \chi , \qquad h_{\tor}(\chi ) =  3C_1 \chi , \qquad h_{0}(\chi ) =  C_1 C_2 \chi .
 \end{equation}
    Applying the estimate \ref{eq:ptest} to the identity \eqref{eq:ptprodid}, and  recalling \eqref{eq:ptod} gives
\begin{equation}
\begin{aligned}
\frac{\Rmod}{\Lambda} \int T u_{z} \, &\mathrm{d}V - \frac{1}{\Rm} \int |\nabla \times \bm{B}|^2 \,
  \mathrm{d}V \\
  &
  \leq  \sum_{\alpha \in \{ \pol, \tor, 0\}} \left( \Rmod \, h_\alpha \left( \| \nabla^2_{x,y}\overline{T}\|_{L^{\infty}_z L^2_{x,y}}\vphantom{\frac{1}{2}}\right) -
  \frac{1}{\Rm} \right) \int |\nabla  \times \bm{B}^{\alpha }|^2 \, \mathrm{d}V.  
\end{aligned}
\end{equation}
By the basic energy identity, Lemma \ref{prop:1}, it holds
\begin{equation}
        \frac{1}{2} \frac{\mathrm{d}}{\mathrm{d}t} \int |\bm{B}|^2 \, \mathrm{d}V   \leq  \sum_{\alpha \in \{ \pol, \tor, 0\}} \left( \Rmod \, h_\alpha \left( \| \nabla^2_{x,y}\overline{T}\|_{L^{\infty}_z L^2_{x,y}}\vphantom{\frac{1}{2}}\right) -
  \frac{1}{\Rm} \right) \int |\nabla \times \bm{B}^{\alpha }|^2 \, \mathrm{d}V,
\end{equation}
and dynamo action requires at least one of the expressions in the parentheses to be nonnegative. Specifically, $ \Rm\,  \Rmod \, h_{\alpha} \geq 1 $ holds for at least one $\alpha \in \{ \pol, \tor , 0 \}.$

Let us now prove the estimate \eqref{eq:ptest}.   The PT decomposition applies to the ohmic dissipation \eqref{eq:ptod}. 
    
    \begin{proposition}
    \label{prop:ptod}
        Let a magnetic field $\bm{B}$ satisfy ferromagnetic boundary conditions. Then, for the ohmic dissipation, the poloidal part is
    \begin{equation*}
      \int | \nabla \times \bm{B}^{\pol}|^2 \, \mathrm{d}V =     \int | \horg  \hord \pol|^2 \, \mathrm{d}V +2 \int |
    \hord \partial_{z} \pol|^2 \, \mathrm{d}V +
  \int | \horg  \partial_{z}^2 \pol|^2\, \mathrm{d}V,
    \end{equation*}
     the toroidal part is
        \begin{equation*}
      \int | \nabla \times \bm{B}^{\tor}|^2 \, \mathrm{d}V =     \int
                                  |\horg  \partial_{z}\tor|^2 \, \mathrm{d}V + \int
                                  |\hord \tor|^2\, \mathrm{d}V,
    \end{equation*}
    and the base flow part is
    \begin{equation*}
              \int | \nabla \times \bm{B}^{0}|^2 \, \mathrm{d}V =     L_x L_y   \int_0^1 | \partial_z \bm{B}^{0}|^2 \, \mathrm{d}z.
    \end{equation*}
    \end{proposition}

For the Lorentz force $\bm{F}^{\mathrm{M}} = (\nabla \times \bm{B}) \times \bm{B}$, the components in the PT decomposition interact via the (horizontal) Poisson bracket
\begin{equation}
    \left\{ f, g \right\} = \partial_{x}f
\partial_{y}g - \partial_{y} f \partial_{x}g =  \horg  f \cdot  \horg ^{\perp}g .
\end{equation}
The end result of the splitting is recorded in
\begin{proposition}
\label{prop:ptfm}
    The Lorentz force has the representation
    \begin{equation*}
    \begin{aligned} 
        \bm{e}_{z} \cdot \nabla \times \bm{F}^{\mathrm{M}} =&\, \{ \partial^2_{z}\pol,\hord\pol \} + \{ \tor, \hord\tor \} \\
        & + \horg \cdot ( \hord \pol \horg \partial_z \tor - \hord \tor \horg \partial_z \pol)\\
        & + (\partial_z \bm{B}^0 \cdot \horg^{\perp}) \hord \pol - (\bm{B}^0 \cdot \horg) \hord \tor.
        \end{aligned}
    \end{equation*}
\end{proposition}
Let us introduce notation: for $1 \leq p < \infty $, the horizontal $L^p_{x,y}$  norms are
\begin{equation}
    \| g (z) \|_{L^p_{x,y}} = \left(\int |g(x,y,z)|^p \, \mathrm{d}x\, \mathrm{d}y \right)^{1/p}.
\end{equation}
such that, in particular, 
\begin{equation}
    \| g\|_{L^p} = \left( \int_0^1 \| g(z)\|^p_{L^p_{x,y}}  \mathrm{d}z \right)^{1/p}.
\end{equation}
For each $z$, it is clear that the $L^p_{x,y}$  norms satisfy basic norm inequalities. Further, it holds 
\begin{lemma}[2D Inequalities]
\label{surfp}
  Let $g$ satisfy
  \begin{equation*}
    \int g   \, \mathrm{d}x\, \mathrm{d}y = 0.
  \end{equation*}
  Then, Poincaré's inequality holds
   \begin{equation*}
    \| g\|_{L^{2}_{x,y}} \leq C_{\mathrm{P}}  \|
    \horg g\|_{L^2_{x,y}},
  \end{equation*}
  and Ladyzhenskaya's inequality holds 
  \begin{equation*}
    \| g\|_{L^{4}_{x,y}} \leq C_{\mathrm{L}}  \|
     g\|_{L^2_{x,y}}^{1/2} \|
    \horg g\|_{L^2_{x,y}}^{1/2}.
  \end{equation*}
\end{lemma}

We are now prepared to estimate $\langle f \bm{e}_z , \nabla \times \bm{F}^{\mathrm{M}}\rangle $ in view of Proposition \ref{prop:ptod} and Proposition \ref{prop:ptfm}. For the purely poloidal term, we shall first integrate by parts, and then estimate with H{\"o}lder's, Lemma \ref{surfp} and Young's inequality:
\begin{equation}
\begin{aligned}
    \int \{ \partial^2_{z}\pol,\hord\pol \}f \, \mathrm{d}x \, \mathrm{d}y  &=       \int \{ f, \partial^2_{z}\pol \} \hord\pol \, \mathrm{d}x \, \mathrm{d}y \\
    &\leq \| \horg f\|_{L^{4}_{x,y}} \|
    \horg  \partial_{z}^2 \pol\|_{L^2_{x,y}} \| \hord \pol\|_{L^{4}_{x,y}} \\
    &\leq C_1\| \hord f\|_{L^{2}_{x,y}} \|
    \horg  \partial_{z}^2 \pol\|_{L^2_{x,y}} \| \horg \hord \pol\|_{L^{2}_{x,y}} \\
    &\leq \frac{C_1}{2}\| \hord f\|_{L^{2}_{x,y}} \left(\|
    \horg  \partial_{z}^2 \pol\|_{L^2_{x,y}}^2 +  \| \horg \hord \pol\|_{L^{2}_{x,y}}^2\right),
\end{aligned}    
\end{equation}
where $C_1 = (C_{\mathrm{L}}C^{1/2}_{\mathrm{P}})^2$.
Integrating the last inequality with respect to $\mathrm{d}z$, it follows that the purely poloidal term has the estimate
\begin{equation}
\label{eq:polbound}
      \int \{ \partial^2_{z}\pol,\hord\pol \}f \, \mathrm{d}V  \leq   \frac{C_{1}}{2}\| \hord f\|_{L^\infty_z L^{2}_{x,y}}  \int | \nabla \times \bm{B}^{\pol}|^2 \, \mathrm{d}V.
\end{equation}
For the purely toroidal part, it follows by a similar procedure 
\begin{equation}
\begin{aligned} 
  \int \left\{ \tor, \hord\tor \right\}  f \, \mathrm{d}x \, \mathrm{d}y &= \int \left\{ f, \tor \right\}  \hord \tor  \, \mathrm{d}x \, \mathrm{d}y \\
                                                     &\leq \| \horg f\|_{L^{4}_{x,y}} \|
    \horg   \tor\|_{L^4_{x,y}} \| \hord \tor\|_{L^{2}_{x,y}} \\ 
    &\leq C_1\| \hord f\|_{L^{2}_{x,y}} \| \hord \tor\|_{L^{2}_{x,y}}^2 .
\end{aligned}
\end{equation}
Then integrating the last inequality with respect to $\mathrm{d}z$ yields
\begin{equation}
\label{eq:torbound}
     \int \left\{ \tor, \hord\tor \right\}  f \, \mathrm{d}x \, \mathrm{d}y \leq C_1 \|\hord f\|_{L^\infty_z L^{2}_{x,y}}  \int | \nabla \times \bm{B}^{\tor}|^2 \, \mathrm{d}V.
\end{equation}
For the mixed poloidal-toroidal term, we integrate by parts and then use H{\"o}lder's inequality and Lemma \ref{surfp}  to find
\begin{equation}
\label{eq:mixprebound}
\begin{aligned}
    &\int f \horg \cdot ( \hord \pol \horg \partial_z \tor - \hord \tor \horg \partial_z \pol) \, \mathrm{d}x\, \mathrm{d}y \\
    &= \int \horg f  \cdot ( \hord \tor \horg \partial_z \pol - \hord \pol \horg \partial_z \tor ) \, \mathrm{d}x\, \mathrm{d}y \\
    &\leq \| \horg f\|_{L^{4}_{x,y}} \left( \| \hord \tor\|_{L^{2}_{x,y}}\| \horg \partial_z \pol\|_{L^{4}_{x,y}} + \| \hord \pol\|_{L^{4}_{x,y}}\| \horg \partial_z \tor\|_{L^{2}_{x,y}} \right) \\
&\leq C_{1} \| \hord f\|_{L^{2}_{x,y}} \left( \| \hord \tor\|_{L^{2}_{x,y}}\| \hord \partial_z \pol\|_{L^{2}_{x,y}} + \| \horg \hord \pol\|_{L^{2}_{x,y}}\| \horg \partial_z \tor\|_{L^{2}_{x,y}} \right).
    \end{aligned}
\end{equation}
Then, applying Young's inequality and integrating with respect to $\mathrm{d}z$ yields
\begin{equation}
\begin{aligned}
\label{eq:mixbound}
    \int f \horg \cdot &( \hord \pol \horg \partial_z \tor - \hord \tor \horg \partial_z \pol) \, \mathrm{d}V \\
&\leq \frac{C_{1}}{2} \| \hord f\|_{L^\infty_z L^{2}_{x,y}} \left(\int | \nabla \times \bm{B}^{\pol}|^2 \, \mathrm{d}V + \int | \nabla \times \bm{B}^{\tor}|^2 \, \mathrm{d}V \right).
    \end{aligned}
\end{equation}
The ferromagnetic boundary conditions on the magnetic field $\bm{B}$ imply for the base flow component $\bm{B}^0 = 0$ on $S_\pm$. For the mixed base flow terms, integrating by parts and applying Poincaré's inequality in $\mathrm{d}x\, \mathrm{d}y$ and $\mathrm{d}z$ then yield
\begin{equation}
\label{eq:baseprebound}
    \begin{aligned}
        \int f \lbrack (\partial_z \bm{B}^0 \cdot & \horg^{\perp}) \hord \pol - (\bm{B}^0 \cdot \horg) \hord \tor \rbrack \, \mathrm{d}x\, \mathrm{d}y \\
        &= \int (\hord \tor \, \bm{B}^0 \cdot \horg  -\hord \pol \, \partial_z \bm{B}^0 \cdot \horg^{\perp}) f \, \mathrm{d}x\, \mathrm{d}y \\
        &\leq  (| \bm{B}^0|\| \hord \tor \|_{L^2_{x,y}} + | \partial_z \bm{B}^0|  \| \hord \pol   \|_{L^2_{x,y}}) \| \horg f\|_{L^2_{x,y}} \\
                &\leq  C_{\mathrm{P}}(| \bm{B}^0|\| \hord \tor \|_{L^2_{x,y}} + C_{\mathrm{P}} | \partial_z \bm{B}^0|  \| \horg \hord \pol   \|_{L^2_{x,y}}) \| \hord f\|_{L^2_{x,y}}.
    \end{aligned}
\end{equation}
Applying Young's inequality, then integrating with respect to   $\mathrm{d}z$ yields, after the Cauchy--Schwarz and  Poincaré inequalities 
\begin{equation}
\begin{aligned}
\label{eq:basebound}
    & \int f \lbrack (\partial_z \bm{B}^0 \cdot \horg^{\perp}) \hord \pol - (\bm{B}^0 \cdot \horg) \hord \tor \rbrack \, \mathrm{d}V \\
    & \leq C_1 \| \hord f\|_{L^\infty_z L^{2}_{x,y}} \left( \frac{1}{2} \int | \nabla \times \bm{B}^{\pol}|^2 \, \mathrm{d}V   + \frac{3}{2} \int | \nabla \times \bm{B}^{\tor}|^2 \, \mathrm{d}V + C_2 \int | \nabla \times \bm{B}^{0}|^2 \, \mathrm{d}V \right),
    \end{aligned}
 \end{equation}
 where the constant is
 \begin{equation}
     C_2 =  \left( \frac{1}{6\pi} + \frac{C_{\mathrm{P}}}{2} \right) \left( \frac{1}{ L_x L_y C_\mathrm{L}^2}\right).
 \end{equation}

 In sum, \eqref{eq:polbound}, \eqref{eq:torbound}, \eqref{eq:mixbound} and \eqref{eq:basebound} imply the key estimate of Theorem \ref{thm:main}, where
 \begin{equation}
     A_2 = C_1 \max \left\{ 3, C_2 \right\}.
 \end{equation}
\end{proof}
\begin{proof}[\underline{Proof of Proposition \ref{prop:ptod}}]
    In the PT decomposition, the explicit forms of the  toroidal and base parts are 
    \begin{equation}
        \bm{B}^{\tor}
= (\horg^{\perp}\tor , 0), \qquad \bm{B}^{0}
= (B^0_x, B^0_y , 0).
    \end{equation}
    The curl of these components are
    \begin{equation}
\nabla \times \bm{B}^{\tor} = (\horg  \partial_{z} \tor,
-\hord \tor), \qquad \nabla \times \bm{B}^{0} = ( - \partial_z B_y^0, \partial_z B^0_x, 0),
\end{equation}
and their norm in $L^2$ is exactly the ohmic dissipation in the claim. To realize the poloidal dissipation, we use the ferromagnetic boundary conditions. The poloidal part and its curl are explicitly
\begin{equation}
    \bm{B}^{\pol} = ( \horg  \partial_{z} \pol, - \hord \pol), \qquad \nabla \times \bm{B}^{\pol} = (- \horg ^{\perp} \Delta \pol, 0).
\end{equation}
Then the ferromagnetic boundary conditions  \eqref{eq:insulating} imply
\begin{equation}
\label{eq:poltorboundary}
     \ \horg \partial_z \pol = - \horg^{\perp} \tor, \quad \bm{B}^0 = 0, \qquad \text{on }S_{\pm}.
\end{equation}

By the inner product structure of $L^2$, it holds
\begin{equation}
\begin{aligned}
        &\int |\horg  \nabla^{2} \pol|^2 \, \mathrm{d}V 
        \\
        &=\int |\horg  \hord\pol|^2 \, \mathrm{d}V + 2 \langle \horg  \partial_{z}^2 \pol, \horg  \hord\pol
    \rangle +  \int |\horg  \partial_{z}^2 \pol |^2\, \mathrm{d}V .
    \end{aligned}
\end{equation}
For the cross term, we integrate by parts in $\mathrm{d}z$
   \begin{equation}
  \begin{aligned}  
    \langle \horg  \partial_{z}^2 \pol, \horg  \hord \pol \rangle =  \int
    \horg  \partial_{z}\pol \cdot \horg  \hord\pol\,\mathrm{d}S_{\pm} -
    \langle \horg  \partial_{z}\pol, \horg  \hord \partial_{z} \pol \rangle.
    \end{aligned}
  \end{equation}
  Because of the ferromagnetic boundary conditions, the integral on the boundary vanishes:
  \begin{equation}
  \int
    \horg  \partial_{z}\pol \cdot \horg  \hord\pol\,\mathrm{d}S_{\pm}= 
    \int
    \{\tor,  \hord\pol\}\, \mathrm{d}S_{\pm} =0.
  \end{equation}
  The first equality follows from \eqref{eq:poltorboundary} and the second equality is a general property of the Poisson bracket. Further,
 we  integrate by parts in $\horg $  for each $z$ to deduce
  \begin{equation}
      -
    \langle \horg  \partial_{z}\pol, \horg  \hord \partial_{z} \pol \rangle = \langle
    \hord \partial_{z} \pol, \hord \partial_{z} \pol \rangle.
  \end{equation}
  Therefore,
  \begin{equation}
    \langle \horg  \partial_{z}^2 \pol, \horg  \hord\pol \rangle = \int 
    |\hord \partial_{z} \pol |^2 \, \mathrm{d}V,
  \end{equation}
  as claimed.
\end{proof}
\begin{proof}[\underline{Proof of Proposition \ref{prop:ptfm}}]
Applying the PT decomposition, the representation holds
\begin{equation}
    \bm{e}_{z} \cdot \nabla \times  ((\nabla \times \bm{B}) \times
\bm{B}) = \sum_{\beta \in \{\pol,\tor,0\}} \left( \sum_{\alpha \in \{\pol,\tor,0\} } \Pi^{\alpha
\beta} \right),
\end{equation}
 for the product terms
\begin{equation}
    \Pi^{\alpha \beta}  = \bm{e}_{z} \cdot \nabla \times  ((\nabla \times \bm{B}^{\alpha}) \times
\bm{B}^{\beta}).
\end{equation}
Enumerating the terms, we shall show that
    \begin{enumerate}[label=\arabic*.]
    \item $\Pi^{\pol \pol} = \{ \partial^2_{z}\pol,\hord\pol \}$ ,\label{term:pp}
    \item \label{term:pt} $\Pi^{ \pol \tor} =  0$ ,
    \item $\Pi^{\pol 0 } = 0$, \label{term:p0}
    \item \label{term:tp} $\Pi^{ \tor \pol} = \horg \cdot ( \hord \pol \horg \partial_z \tor - \hord \tor \horg \partial_z \pol)
    $  , 
    \item $\Pi^{\tor \tor} = \{ \tor, \hord\tor \}$, \label{term:tt}
    \item $\Pi^{\tor 0} = - (\bm{B}^0 \cdot \horg) \hord \tor$ ,\label{term:t0}
    \item $\Pi^{0\pol  } = (\partial_z \bm{B}^0 \cdot \horg^{\perp}) \hord \pol $, \label{term:0p}
    \item $ \Pi^{0\tor } = 0$ ,\label{term:0t}
    \item $ \Pi^{00} = 0$ \label{term:00}.
    \end{enumerate}
It remains to identify each term and  discover expressions with poloidal and toroidal potentials $\pol$ and $\tor$. Recall the explicit form of the PT decomposition for the poloidal part
\begin{equation}
\label{eq:ptpartp}
\bm{B}^{\pol} = ( \horg  \partial_{z} \pol, - \hord \pol), \qquad  \nabla \times \bm{B}^{\pol} = (- \horg ^{\perp} \Delta \pol, 0), 
\end{equation}
for the toroidal part
\begin{equation}
\label{eq:ptpartt}
\bm{B}^{\tor}
= (\horg^{\perp}\tor , 0)  \qquad \nabla \times \bm{B}^{\tor} = (\horg  \partial_{z} \tor,
-\hord \tor),
\end{equation}
and for the base flow
\begin{equation}
\label{eq:ptpart0}
\bm{B}^{0}
= (B^0_x, B^0_y , 0)  \qquad \nabla \times \bm{B}^{0} = ( - \partial_z B_y^0, \partial_z B^0_x, 0).
\end{equation}
We shall make repeated use of the vector calculus identity: for solenoidal $\bm{v}$ and $\bm{w}$, it holds
\begin{equation}
\label{eq:keyvec}
    \bm{e}_{z} \cdot \nabla \times (\bm{v} \times \bm{w}) = \bm{w} \cdot \nabla v_z - \bm{v} \cdot \nabla w_z.
\end{equation}
\paragraph{\emph{Term \ref{term:pp}}}
  By the identity, 
  \begin{equation}
  \label{eq:one}  
  \begin{aligned}
\Pi^{\pol\pol} &=  \bm{B}^\pol
  \cdot \nabla (\nabla \times \bm{B}^{\pol})_{z} - (\nabla \times \bm{B}^{\pol}) \cdot
  \nabla
  B^{\pol}_{z} \\
  &= - \horg ^{\perp} \nabla^{2}\pol
  \cdot \horg  \hord \pol  \\
  &=\{ \nabla^{2}\pol, \hord \pol\}. 
  \end{aligned}
  \end{equation}
  The result then follows from the anticommutativity and bilinearity of the
  Poisson bracket, recalling that $\nabla^{2}  = \hord + \partial^{2}_{z}$.
\paragraph{\emph{Term \ref{term:pt}}}
 By the identity, 
  \begin{equation}
  \begin{aligned}
\Pi^{\pol \tor} =  \bm{B}^\tor
  \cdot \nabla (\nabla \times \bm{B}^{\pol})_{z} - (\nabla \times \bm{B}^{\pol}) \cdot \nabla B^{\tor}_{z}=0.
  \end{aligned}
  \end{equation}
  \paragraph{\emph{Term \ref{term:p0}}}
 By the identity, 
  \begin{equation}
  \begin{aligned}
\Pi^{\pol 0} &=  \bm{B}^0
  \cdot \nabla (\nabla \times \bm{B}^{\pol})_{z} - (\nabla \times \bm{B}^{\pol}) \cdot
  \nabla
  B^{0}_{z} =0.
  \end{aligned}
  \end{equation}
  \paragraph{\emph{Term \ref{term:tp}}}
   By the identity, 
  \begin{equation}
  \begin{aligned}
\Pi^{\tor \pol }&=
    \bm{B}^\pol   \cdot \nabla (\nabla \times \bm{B}^{\tor})_{z} - (\nabla \times \bm{B}^{\tor}) \cdot
  \nabla
  B^{\pol}_{z} \\
&=(\nabla \times \bm{B}^{\tor}) \cdot
  \nabla
  \hord \pol -\bm{B}^\pol
\cdot
\nabla 
\hord\tor.
  \end{aligned}
  \end{equation}
  Explicitly,
  \begin{equation}
(\nabla \times \bm{B}^{\tor}) \cdot
  \nabla
  \hord \pol  = \horg  \partial_{z} \tor \cdot \horg  \hord\pol - (\hord \tor)
  (\hord \partial_{z}\pol),
  \end{equation}
  and
  \begin{equation}
  \,\,\bm{B}^{\pol}\cdot \nabla  \hord\tor = \horg  \partial_{z}\pol \cdot \horg 
  \hord\tor - (\hord\pol) (\hord  \partial_{z} \tor).
  \end{equation}
  Then,
  \begin{equation}
      \begin{aligned}
          \Pi^{\tor \pol} = (\hord \partial_{z} \tor + \horg  \partial_{z} \tor \cdot \horg )
    \hord\pol - (\hord \partial_{z}\pol + \horg  \partial_{z}\pol \cdot \horg ) \hord \tor .
      \end{aligned}
  \end{equation}
  and the result follows from the identity
  \begin{equation}
      \horg \cdot ( f \bm{v}) = (\horg \cdot \bm{v} + \bm{v} \cdot \horg) f .
  \end{equation}
  \paragraph{\emph{Term \ref{term:tt}}}
  By the identity, 
  \begin{equation}
  \begin{aligned}
\Pi^{\tor \tor} &=  \bm{B}^\tor
  \cdot \nabla (\nabla \times \bm{B}^{\tor})_{z} - (\nabla \times \bm{B}^{\tor}) \cdot
  \nabla
  B^{\tor}_{z}\\
  &=- \horg^{\perp} \tor \cdot \horg \hord \tor \\
  &= \left\{ \tor, \hord\tor \right\},
  \end{aligned}
  \end{equation}
    \paragraph{\emph{Term \ref{term:t0}}}
 By the identity, 
  \begin{equation}
  \begin{aligned}
\Pi^{\tor 0} &=  \bm{B}^0
  \cdot \nabla (\nabla \times \bm{B}^{\tor})_{z} - (\nabla \times \bm{B}^{\tor}) \cdot
  \nabla
  B^{0}_{z}\\
  &= - \bm{B}^0 \cdot \horg \hord \tor.
  \end{aligned}
  \end{equation}
   \paragraph{\emph{Term \ref{term:0p}}}
   By the identity, 
  \begin{equation}
  \begin{aligned}
\Pi^{0 \pol }&=
    \bm{B}^\pol   \cdot \nabla (\nabla \times \bm{B}^{0})_{z} - (\nabla \times \bm{B}^{0}) \cdot
  \nabla
  B^{\pol}_{z} \\
&= \partial_z \bm{B}^0 \cdot \horg^{\perp} \hord \pol.
  \end{aligned}
  \end{equation}
  \paragraph{\emph{Term \ref{term:0t}}}
   By the identity, 
  \begin{equation}
  \begin{aligned}
\Pi^{0 \tor }&=
    \bm{B}^\tor   \cdot \nabla (\nabla \times \bm{B}^{0})_{z} - (\nabla \times \bm{B}^{0}) \cdot
  \nabla
  B^{\tor}_{z} \\
&= 0.
  \end{aligned}
  \end{equation}
\paragraph{\emph{Term \ref{term:00}}}
Because $\bm{B}^0$ depends only on $z$, it follows
\begin{equation}
    \Pi^{00} = 0,
\end{equation}
and we are done.
\end{proof}

\section{Proof of Theorems: Viscous Case}
\label{sec:proof_viscous_results}

\begin{proof}[\underline{Proof of Theorem~\ref{thm:visc_main}}]
Following Lemma \ref{prop:1}, the induction equation \eqref{eq:induction_eq} has the energy identity
\begin{equation}
    \label{eq:induct_energy_v}
    \frac{1}{2} \frac{\mathrm{d}}{\mathrm{d}t} \int |\bm{B}|^2 \mathrm{d}V  + \frac{1}{\Rm}\int |\nabla \bm{B}|^2 \mathrm{d}V= \langle \bm{u}, \bm{B}  \times (\nabla \times \bm{B}) \rangle .
\end{equation}
Then applying $\langle \bm{u}, \, \cdot \, \rangle$ to the momentum equation \eqref{eq:momentum_eq} after integrating by parts
\begin{equation}
    \Lambda\,  \langle \bm{u}, \bm{B}  \times (\nabla \times \bm{B}) \rangle = \Rmod \int u_z T \, \mathrm{d}V - \Ek \int | \nabla \times \bm{u}|^2 \, \mathrm{d}V.
\end{equation}
Indeed, the stress-free condition on $\bm{u}$ ensures that the vorticity $\omega = \nabla \times \bm{u}$ is normal to the boundary. Overall the energy identity in the viscous setting is
\begin{equation}
    \label{eq:induct_energy_v2}
    \frac{1}{2} \frac{\mathrm{d}}{\mathrm{d}t} \int |\bm{B}|^2 \mathrm{d}V  + \frac{1}{\Rm}\int |\nabla \bm{B}|^2 \mathrm{d}V= \frac{\Rmod}{\Lambda} \int u_z T \, \mathrm{d}V - \frac{\Ek}{\Lambda} \int | \nabla \times \bm{u}|^2 \, \mathrm{d}V.
\end{equation}
Following Theorem \ref{theorem_1}, the curl of the momentum equation is
\begin{equation}
    -\partial_z \bm{u} = \Lambda \nabla \times \bm{F}^{\mathrm{M}} - \Rmod\,  \bm{e}_z \times \nabla T + \Ek \nabla\times \nabla^2 \bm{u}.
\end{equation}
Then substituting the expression above, the identity \eqref{eq:ibpdz_2} when viscosity is finite is
\begin{equation}
\label{eq:prodid_v}
 \frac{\Rmod }{\Lambda }\int u_{z} T \, \mathrm{d} V 
= \Rmod \, \langle
\nabla \times \bm{F}^{\mathrm{M}},  \overline{T} \bm{e}_{z}
\rangle + \frac{\Ek \Rmod }{\Lambda}  \langle
\nabla \times \nabla^2 \bm{u},  \overline{T} \bm{e}_{z}
\rangle.
\end{equation}
On the right hand side, the first term  may be treated by the integration by the parts  \eqref{eq:ibpprod} and Lemma \ref{lemma:production1} such that
\begin{equation}
\label{eq:prodest_v}
    \langle
\nabla \times \bm{F}^{\mathrm{M}},  \overline{T} \bm{e}_{z}
\rangle \leq A_p \| \nabla_{x,y}\overline{T}\|_{L^p}\int | \nabla \bm{B}|^2 \, \mathrm{d}V.
\end{equation}
For the second term, integration by parts $\nabla \times $ yields
    \begin{equation}
        \begin{aligned}
          &  \langle
\nabla \times \nabla^2 \bm{u},  \overline{T} \bm{e}_{z}
\rangle  + \int  \bm{n}_{\pm } \cdot (\nabla^2 \bm{u} \times \overline{T}\bm{e}_z  - (\nabla \times  \bm{u})  \times ( \nabla \times \overline{T}\bm{e}_z ))\, \mathrm{d}S_{\pm} \\
&= \langle
\nabla \times  \bm{u}, \nabla\times \nabla \times  \overline{T} \bm{e}_{z}
\rangle.
        \end{aligned}
    \end{equation}
    The first vanishes because $\bm{n}_{\pm} \times \overline{T}\bm{e}_z = 0$ and the second vanishes because of the stress-free boundary conditions. Thus by Cauchy-Schwarz it holds
    \begin{equation}
    \label{eq:pre_ensest_v}
        \langle
\nabla \times \nabla^2 \bm{u},  \overline{T} \bm{e}_{z}
\rangle \leq \| \nabla \times \nabla \times \overline{T}\bm{e}_{z}\|_{L^2} \| \nabla \times \bm{u}\|_{L^2}.
    \end{equation}
    Because $\nabla \times \overline{T}\bm{e}_z$ is divergence free, it holds
    \begin{equation}
    \label{eq:DDTid}
        \| \nabla \times \nabla \times \overline{T}\bm{e}_{z}\|_{L^2} = \| \nabla \nabla_{x,y}\overline{T}\|_{L^2}.
    \end{equation}
    From \eqref{eq:prodest_v}, \eqref{eq:pre_ensest_v} and \eqref{eq:DDTid}, we have the estimate for the production term
    \begin{equation}
\label{eq:prodest_visc}
 \frac{\Rmod }{\Lambda }\int u_{z} T \, \mathrm{d} V 
\leq  \Rmod \, A_p \| \nabla_{x,y}\overline{T}\|_{L^p}\int | \nabla \bm{B}|^2 \, \mathrm{d}V + \frac{\Ek \Rmod}{\Lambda } \| \nabla \nabla_{x,y}\overline{T}\|_{L^2} \| \nabla \times \bm{u}\|_{L^2}
\end{equation}
Applying Young's inequality, it holds
\begin{equation}
    \label{eq:ensest_v}
    \begin{aligned}
     &\| \nabla \nabla_{x,y}\overline{T}\|_{L^2} \| \nabla \times \bm{u}\|_{L^2} \\
&\leq \frac{ \Rmod  \,\Lambda   }{  \, \Ek^{\gamma }} \| \nabla \nabla_{x,y}\overline{T}\|_{L^2}^2\int | \nabla \bm{B}|^2 \, \mathrm{d}V + \frac{  \Ek^{\gamma }}{4 \Lambda \Rmod \| \nabla \bm{B}\|_{L^2 }^{2}}  \int | \nabla \times \bm{u}|^2 \, \mathrm{d}V.
\end{aligned}
\end{equation}
    Therefore, combining the identities  \eqref{eq:induct_energy_v} and \eqref{eq:prodid_v} with estimates \eqref{eq:prodest_v} and \eqref{eq:ensest_v}, it holds
\begin{align}
  \frac{1}{2} \frac{\mathrm{d}}{\mathrm{d}t} \int |\bm{B}|^2 \,
  \mathrm{d}V &\leq  \left( A_{p} \Rmod\,  \| \nabla_{x,y} \overline{T} \|_{L^{p}} + \Ek^{(1-\gamma)} \Rmod^2 \,\| \nabla \nabla_{x,y} \overline{T}\|_{L^2}^2 -
  \frac{1}{\Rm} \right) \int |\nabla \bm{B}|^2 \, \mathrm{d}V \nonumber \\
  &\quad + \frac{\Ek}{\Lambda } \left( \frac{ \Ek^{\gamma }}{4\Lambda  \| \nabla \bm{B}\|_{L^2}^2 } - 1\right) \int |\nabla \times \bm{u}|^2 \, \mathrm{d}V.
  \end{align}
 Dynamo action  requires that at least one of the expressions inside parentheses is nonnegative. Finally, notice that we are interested in $Ek^\gamma \to 0$ and $Ek^{1-\gamma}\to 0$ as $Ek \to 0$, to recover the inviscid limit with a non-trivial result. These two considerations imply the desired range $0<\gamma < 1$. 
\end{proof}

\begin{proof}[\underline{Proof of Theorem~\ref{corollary_0_visc}}]
\label{proof_2nd_viscous}
For $\Ek > 0$, the linear operator  $-\Ek \nabla^2 + \bm{e}_z \times \vphantom{\cdot} $ subject to no-penetration and stress-free boundary conditions has trivial kernel in $L^2_{\mathrm{sol}}$. The unique solution to the momentum equation \eqref{eq:momentum_eq} thus admits the decomposition
\begin{equation}
    \bm{u} =  \bm{u}^{\mathrm{M}} + \bm{u}^{\mathrm{A}},
\end{equation}
where $\bm{u}^{\mathrm{M}}$ is the magnetostrophic part of the velocity, and $\bm{u}^{\mathrm{A}}$ is the Archimedean part of the velocity. $\bm{u}^{\mathrm{A}}$ solves \eqref{eq:arcpart_v}, and $\bm{u}^{\mathrm{M}}$ solves
    \begin{equation}
        - \Ek \nabla^2 \bm{u}^{\mathrm{M}} + \bm{e}_z \times \bm{u}^{\mathrm{M}} + \nabla P^{\mathrm{M}} = \Lambda \, (\nabla \times \bm{B}) \times \bm{B} \label{eq:magpart_v}, \quad \nabla \cdot \bm{u}^{\mathrm{M}} = 0.
    \end{equation}
    Observe that, for $\Ek > 0$, smooth solutions $\bm{u}^{\mathrm{M}}$ exist  for smooth $\bm{F}^{\mathrm{M}} = (\nabla \times \bm{B}) \times \bm{B}$, even if $\bm{F}^{\mathrm{M}}$  does not satisfy Taylor's constraint.

    The right hand side of \eqref{eq:induct_energy_v} satisfies
    \begin{equation}
         \langle \bm{u}, \bm{B}  \times (\nabla \times \bm{B}) \rangle = -  \langle \bm{u}^{\mathrm{M}}, \bm{F}^{\mathrm{M}} \rangle -  \langle \bm{u}^{\mathrm{A}}, \bm{F}^{\mathrm{M}} \rangle.
    \end{equation}
    The first term is signed; applying $\langle \bm{u}^{\mathrm{M}}, \cdot\, \rangle$ to equation \eqref{eq:magpart_v} and integrating by parts yields
    \begin{equation}
        \langle \bm{u}^{\mathrm{M}},\bm{F}^{\mathrm{M}}\rangle =  \frac{\Ek}{\Lambda } \int | \nabla \bm{u}^{\mathrm{M}}|^2 \, \mathrm{d}V.
    \end{equation}
    The second term may be estimated by Lemma \ref{lemma:production1},
    \begin{equation}
        |\langle \bm{u}^{\mathrm{A}}, \bm{F}^{\mathrm{M}}\rangle| \leq A_{p}\, \|\bm{u}^{\mathrm{A}}\|_{L^p} \int |\nabla \bm{B}|^2 \mathrm{d}V.
    \end{equation}
Consequently,
\begin{equation}
\frac{1}{2}\frac{\mathrm{d}}{\mathrm{d}t}\int |\bm{B}|^2 \mathrm{d}V\leq\left( A_p \,\|\bm{u}^{\mathrm{A}}\|_{p}-\frac{1}{Rm}\right)\int |\nabla \bm{B}|^2 \mathrm{d}V - \frac{ \Ek}{\Lambda}\int |\nabla\bm{u}^{\mathrm{M}}|^2 \mathrm{d}V.
\label{eq:viscous_magnetic_energy_bound}
\end{equation}
Magnetic energy growth therefore requires the  expression in parentheses is nonnegative, completing the proof.
\end{proof}

\subsection{Magnetic Enstrophy Bounds}
\label{proof:ohmic_theorem}
Theorem \ref{theorem:ohmic} requires first deriving equations for the current evolution for our rapidly rotating settings. In fact, an analogue of stress-free boundary conditions holds for the current
\begin{equation}
\label{eq:BCextraJ}
     (\nabla \times  \bm{J}) \times   \bm{n}_{\pm}\,|_{\,S_{\pm}} = 0.
\end{equation}
Pointwise, the induction equation \eqref{eq:inviscid_induction_eq} and ferromagnetic boundary conditions imply that on the boundaries it holds
\begin{equation}
      \left(\nabla \times  \left( \frac{1}{\Rm}  \bm{J} - \bm{u} \times \bm{B}\right) \right) \times \bm{n}_{\pm } = 0 \quad \text{on } S_{\pm}.
\end{equation}
Then, one sees that the boundary conditions \eqref{eq:bc}, \eqref{bc:stress_free} and \eqref{eq:insulating2} imply
\begin{equation}
    \bm{n}_{\pm } \times \nabla \times (\bm{u} \times \bm{B}) = B_z  ( \bm{n}_{\pm } \times \partial_z \bm{u)}  = 0  \quad \text{on } S_{\pm}.
\end{equation}
The first equality follows from vector calculus, the ferromagnetic  conditions for $\bm{B}$ and the no-penetration conditions for $\bm{u}$ on the boundary, and the second follows from the stress-free boundary conditions for $\bm{u}$.

These boundary conditions allow us to isolate magnetic enstrophy production
  due to temperature dynamics:
\begin{proposition}
\label{prop2}
Fix $\Ek > 0$ and consider the viscous inertialess model. Suppose that
\begin{equation*}
\int \bm{n}_{\pm} \cdot \bm{B} \, \mathrm{d}S_{\pm} = 0.
\end{equation*}
Assume that  the solution is smooth and that $\bm{u}$ satisfies stress-free boundary conditions. Then, it holds
\begin{equation*} 
\begin{aligned}
    \frac{1}{2} \frac{\mathrm{d}}{\mathrm{d}t} &\int |\bm{J}|^2 \, \mathrm{d}V  + \frac{1}{\Rm}\int |\nabla \bm{J}|^2 \mathrm{d}V + \frac{\Ek}{\Lambda}\int | \nabla^2 \bm{u}|^2 \, \mathrm{d}V\\
    &\leq C_m \| \nabla \bm{u}\|_{L_2} \int  | \nabla \bm{J}|^2 \, \mathrm{d}V  +  \frac{Ra}{\Lambda} \int \bm{u} \cdot \nabla \times \nabla \times T \bm{e}_z\, \mathrm{d}V.
    \end{aligned}
\end{equation*}
\end{proposition}
\begin{proof}

The right-hand side of the magnetic enstrophy identity \eqref{eq:J_energy_id} becomes after integration by parts
\begin{align}
    \label{eq:rearranging}
    &-\int \bm{J} \cdot \nabla^2 (\bm{u}\times \bm{B})\, \mathrm{d}V  \pm \int  (\bm{J} \cdot \partial_z (\bm{u}\times \bm{B}) -  \partial_z\bm{J} \cdot  (\bm{u}\times \bm{B}))\, \mathrm{d}S_{\pm} \nonumber  \\
     &= \int \bm{u} \cdot (\nabla^2 \bm{J} \times \bm{B}) \, \mathrm{d}V.
\end{align}
The boundary conditions combine such that the boundary term above is zero,
\begin{equation}
    \int  (\bm{J} \cdot \partial_z (\bm{u}\times \bm{B}) -  \partial_z\bm{J} \cdot  (\bm{u}\times \bm{B}))\, \mathrm{d}S_{\pm} = 0.
\end{equation}
Indeed, ferromagnetic boundary conditions guarantee $\bm{B}$ is normal to the boundary and \eqref{eq:insulating2}; the stress-free boundary conditions on $\bm{u}$ and the analogue \eqref{eq:BCextraJ} for $\bm{J}$   then cause the integrand to vanish.
Therefore,
\begin{equation}
     \frac12 \frac{\mathrm{d}}{\mathrm{d}t} \int |\bm{J}|^2\,\mathrm{d}V + \frac{1}{\Rm} \int |\nabla \bm{J}|^2 \, \mathrm{d}V = \int  \bm{u} \cdot (\nabla^2 \bm{J} \times \bm{B}) \, \mathrm{d}V .
     \label{eq:J_eq_2}
\end{equation}
Pointwise, by the Leibniz rule and rearranging
\begin{equation}
\begin{aligned}
     \nabla^2 (\bm{J} \times \bm{B})
    & =  \nabla^2 \bm{J}\times \bm{B} - \bm{J}\times  \nabla^2\bm{B} + 2 \nabla \cdot  (\bm{J} \times \nabla \bm{B}) .\label{eq:lorentz_term}
\end{aligned}
\end{equation}
The right hand side of \eqref{eq:J_eq_2} is then equal to
\begin{equation}
\label{eq:lorentz_term3}
    \int  \bm{u} \cdot (\nabla^2 \bm{J} \times \bm{B}) \, \mathrm{d}V  = \int  \bm{u} \cdot (\bm{J} \times \nabla^2  \bm{B} - 2 \nabla \cdot (\bm{J} \times \nabla \bm{B})) \, \mathrm{d}V  + \langle \bm{u}, \nabla^2 ( \bm{J} \times \bm{B}) \rangle.
\end{equation}
The first two terms admit an estimate recorded in
\begin{lemma}
\label{lem:mixedtermbd}
Suppose the magnetic field $\bm{B}$ satisfies ferromagnetic boundary conditions and suppose that $\bm{u}$ satisfies no-penetration and stress-free boundary conditions. Then,
    \begin{equation*}
        \left|\int  \bm{u} \cdot (\bm{J} \times \nabla^2  \bm{B} - 2 \nabla \cdot (\bm{J} \times \nabla \bm{B})) \, \mathrm{d}V\right| \leq C_m \| \nabla \bm{u}\|_{L_2} \int  | \nabla \bm{J}|^2 \, \mathrm{d}V.
    \end{equation*}
\end{lemma}

For the last term in \eqref{eq:lorentz_term3}, we shall use the momentum equation  to rewrite it. Testing the curl of the momentum equation  \eqref{eq:inviscid_momentum_eq} against $\nabla \times \bm{u}$ results in 
\begin{equation}
\label{eq:sourcetermid}
\begin{aligned}
 & - \Lambda \, \langle \nabla \times \bm{u}, \nabla \times  (\bm{J} \times \bm{B}) \rangle \\
 &= \Rmod \,\langle \nabla \times \bm{u},  \nabla \times T \bm{e}_z \rangle   + \Ek \langle \nabla \times \bm{u}, \nabla \times \nabla^2 \bm{u}\rangle  -  \langle \nabla \times \bm{u}, \nabla \times ( \bm{e}_z \times \bm{u}) \rangle.
  \end{aligned}
\end{equation}
On the right hand side, the last term  vanishes upon integration by parts
\begin{equation}
\label{eq:ohmic_preCid}
\begin{aligned}
 - \langle \nabla \times \bm{u}, \nabla \times ( \bm{e}_z \times \bm{u}) \rangle &=   \langle \nabla^2\bm{u},   \bm{e}_z \times \bm{u} \rangle   \\
  &= -\langle  \bm{u} , \bm e_z\times\nabla^2\bm u
 \rangle  = \sum_{i=1}^3 \langle  \partial_i\bm {u}, \bm e_z\times\partial_i\bm {u} \rangle  = 0.
 \end{aligned}
\end{equation}
The penultimate term is signed, after integrating by parts
\begin{equation}
\label{eq:ohmic_preUid}
\begin{aligned}
    \langle \nabla \times \bm{u}, \nabla \times \nabla^2 \bm{u}\rangle  + \int \bm{n}_{\pm}\cdot ( (\nabla \times \bm{u}) \times \nabla^2 \bm{u}) \, \mathrm{d}S_{\pm}= - \int | \nabla^2 \bm{u}|^2 \, \mathrm{d}V,
    \end{aligned}
\end{equation}
where the boundary term vanishes by the stress-free boundary condition.
Integrating by parts the first term yields the more convenient form,
\begin{equation}
\label{eq:ohmic_preTid}
    \langle \nabla \times \bm{u},  \nabla \times T \bm{e}_z \rangle = \int \bm{u} \cdot \nabla \times \nabla \times T \bm{e}_z\, \mathrm{d}V,
\end{equation}
where $\nabla \times T \bm{e}_z$ is identically zero on the boundary. For the left hand side, integration by parts reveals 
\begin{equation}
\label{eq:ohmic_preJid}
\begin{aligned}
    -  \, \langle \nabla \times \bm{u}, \nabla \times  (\bm{J} \times \bm{B}) \rangle + \int \bm{n}_{\pm} \cdot (\bm{u} \times \nabla \times (\bm{J} \times \bm{B})) \, \mathrm{d}S_{\pm } = \langle \bm{u}, \nabla^2 ( \bm{J} \times \bm{B}) \rangle .
    \end{aligned}
\end{equation}
Pointwise, the integrand on the boundary $S_{\pm}$ is exactly
\begin{equation}
    \begin{aligned}
        \bm{n}_{\pm} \cdot \bm{u} \times ((\bm{B} \cdot \nabla) \bm{J} - (\bm{J} \cdot \nabla) \bm{B}) =  \bm{n}_{\pm} \cdot \bm{u} \times B_z \partial_z \bm{J} - \bm{u} \cdot (\bm{J} \cdot \nabla) ( \bm{B} \times \bm{n}).
    \end{aligned}
\end{equation}
On the right hand side, the first term vanishes because of \eqref{eq:BCextraJ} and the second term vanishes by the ferromagnetic boundary conditions; the boundary term above is zero. Substituting \eqref{eq:ohmic_preCid}, \eqref{eq:ohmic_preUid}, \eqref{eq:ohmic_preTid} and \eqref{eq:ohmic_preJid} into the identity \eqref{eq:sourcetermid} yields
\begin{equation}
\label{eq:ohmic_totalid}
    \langle \bm{u}, \nabla^2 ( \bm{J} \times \bm{B}) \rangle = \frac{\Rmod}{\Lambda}  \int \bm{u} \cdot \nabla \times \nabla \times T \bm{e}_z\, \mathrm{d}V - \frac{\Ek}{\Lambda}\int | \nabla^2 \bm{u}|^2 \, \mathrm{d}V.
\end{equation}


Combining  \eqref{eq:sourcetermid}, \eqref{eq:lorentz_term3} and \eqref{eq:J_eq_2} yields the identity for the current energy
\begin{equation}
    \begin{aligned}
    &\frac{1}{2} \frac{\mathrm{d}}{\mathrm{d}t} \int |\bm{J}|^2 \, \mathrm{d}V  + \frac{1}{\Rm}\int |\nabla \bm{J}|^2 \mathrm{d}V + \frac{\Ek}{\Lambda}\int | \nabla^2 \bm{u}|^2 \, \mathrm{d}V\\
    &=\int  \bm{u} \cdot (\bm{J} \times \nabla^2  \bm{B} - 2 \nabla \cdot (\bm{J} \times \nabla \bm{B})) \, \mathrm{d}V +  \frac{Ra}{\Lambda} \int \bm{u} \cdot \nabla \times \nabla \times T \bm{e}_z\, \mathrm{d}V.
    \end{aligned}
\end{equation}
Lemma \ref{lem:mixedtermbd} applies and the proof is complete.

\end{proof}
\begin{proof}[\underline{Proof of Theorem  \ref{theorem:ohmic}}]
    We shall now estimate the production term in Proposition \ref{prop2} which depends on the temperature. Integrating by parts in $\mathrm{d}z$ and using the boundary condition gives
    \begin{equation}
    \label{eq:3.8part1}
        \int \bm{u} \cdot \nabla \times \nabla \times T \bm{e}_z\, \mathrm{d}V = - \langle \partial_z \bm{u}, \overline{\nabla \times \nabla \times T \bm{e}_z} \rangle,
    \end{equation}
    where recall $\overline{\vphantom {i}\cdot \vphantom {i}}$ indicates the vertical primitive. The curl of the momentum equation \eqref{eq:momentum_eq} is
    \begin{equation}
        - \partial_z \bm{u} =  \Lambda \nabla \times (\bm{J} \times \bm{B}) + \Rmod \nabla \times T \bm{e_z} + \Ek \nabla \times \nabla^2 \bm{u}.
    \end{equation}
    Observe that the middle term is toroidal, and therefore is orthogonal to the poloidal field $\overline{\nabla \times \nabla \times T \bm{e}_z}$ such that
    \begin{equation}
         \langle \nabla \times T \bm{e_z}, \overline{\nabla \times \nabla \times T \bm{e}_z} \rangle = -  \langle \nabla \times \overline{T} \bm{e_z}, \nabla \times \nabla \times T \bm{e}_z \rangle = 0.
    \end{equation}
    For the last term, we integrate by parts
    \begin{equation}
    \begin{aligned}
        &\langle\nabla \times \nabla^2 \bm{u},  \overline{\nabla \times \nabla \times T \bm{e}_z} \rangle + \int \bm{n}_{\pm} \cdot ( \nabla^2 \bm{u} \times  \overline{\nabla \times \nabla \times T \bm{e}_z}) \, \mathrm{d}S_{\pm } \\
        &= \langle \nabla^2 \bm{u},  \nabla \times \overline{\nabla \times \nabla \times T \bm{e}_z} \rangle.
        \end{aligned}
    \end{equation}
    By the poloidal structure of $\overline{\nabla \times \nabla \times T \bm{e}_z}$, the boundary term vanishes. Pointwise on $S_{\pm}$ it holds
    \begin{equation}
    \begin{aligned}
        &\bm{n}_{\pm} \cdot ( \nabla^2 \bm{u} \times  \overline{\nabla \times \nabla \times T \bm{e}_z})\\ &= -\nabla^2 \bm{u} \cdot (  \bm{n}_{\pm} \times  \overline{\nabla \times \nabla \times T \bm{e}_z}) 
        = -\nabla^2 \bm{u} \cdot (    \overline{\nabla^{\perp}_{x,y}\partial_zT}) 
         = -\nabla^2 \bm{u} \cdot     \nabla^{\perp}_{x,y}T= 0,
        \end{aligned}
    \end{equation}
    where the last equality holds by the isothermal boundary conditions.
   Therefore, we have established the identity
    \begin{equation}
    \label{eq:penultJid}
    \begin{aligned}
 &\frac{\Rmod }{\Lambda} \int \bm{u} \cdot \nabla \times \nabla \times T \bm{e}_z\, \mathrm{d}V \\
 &=  \Rmod \, \langle  \nabla \times (\bm{J} \times \bm{B}), \overline{\nabla \times \nabla \times T \bm{e}_z}   \rangle + \frac{\Rmod\, \Ek}{\Lambda} \langle \nabla^2 \bm{u},  \nabla \times \overline{\nabla \times \nabla \times T \bm{e}_z} \rangle.
 \end{aligned}
    \end{equation}
    Because $\bm{B}$ and $\bm J$ are divergence free, it holds pointwise
    \begin{equation}
        \nabla \times (\bm{J} \times \bm{B}) = (\bm{B} \cdot \nabla) \bm{J} - (\bm{J} \cdot \nabla) \bm{B}.
    \end{equation}
    By H{\"o}lder's inequality, it then follows
    \begin{equation}
        \begin{aligned}
             \langle  \nabla \times (\bm{J} \times \bm{B}), \overline{\nabla \times \nabla \times T \bm{e}_z}   \rangle \leq \|\overline{\nabla \times \nabla \times T \bm{e}_z} \|_{L^2} \left(\|\bm{B} \|_{L^\infty } \|\nabla \bm{J} \|_{L^2 } + \|\bm{J} \|_{L^6 } \|\nabla\bm{B} \|_{L^3 }\right)
        \end{aligned}
    \end{equation}
We now recall the Sobolev embeddings $H^2(V) \hookrightarrow L^\infty(V)$, the Ladyzhenskaya and Gagliardo--Nirenberg inequalities. These estimates are,
    \begin{equation}
    \label{eq:manyem}
        \| \bm{J} \|_{L^6} \leq C_{\mathrm{S}}\| \nabla \bm{J} \|_{L^2}, \quad  \| \bm{B} \|_{L^\infty} \leq C_b(\| \bm{B} \|_{L^2} + \| \nabla \bm{J} \|_{L^2}),  \quad 
        \| \nabla \bm{B} \|_{L^3}^2 \leq C_{\mathrm{S}} \| \nabla  \bm{J} \|_{L^2} \| \bm{J} \|_{L^2}.
    \end{equation}
    We used the fact that $\bm{B}$ is divergence free, and recalled that the ferromagnetic boundary conditions and exact magnetic flux identity ensure the kernel of the curl operator is trivial. With the embeddings \eqref{eq:manyem}  and the Poincaré inequality, this implies 
    \begin{equation}
    \label{eq:penultJest}
        \langle  \nabla \times (\bm{J} \times \bm{B}), \overline{\nabla \times \nabla \times T \bm{e}_z}   \rangle \leq C_n \|\overline{\nabla \times \nabla \times T \bm{e}_z} \|_{L^2} \int |\nabla \bm{J} |^2 \, \mathrm{d}V,
    \end{equation}
    where 
\begin{equation}
    C_n = C_b(1+ C_{\mathrm{P}}^2 + C_{\mathrm{P}}^4)^{1/2}+ C^{3/2}_{\mathrm{S}} C_{\mathrm{P}}^{1/2}.
\end{equation}
On the other hand, by  Young's inequality it holds
\begin{equation}
    \label{eq:ensest_v2}
    \begin{aligned}
&\langle \nabla^2 \bm{u},  \nabla \times \overline{\nabla \times \nabla \times T \bm{e}_z} \rangle \\
&\leq \frac{ \Rmod  \,\Lambda  \,  }{ \, \Ek^{\gamma } } \|\nabla \times \overline{\nabla \times \nabla \times T \bm{e}_z}\|_{L^2}^2\int | \nabla \bm{J}|^2 \, \mathrm{d}V + \frac{  \Ek^{\gamma }}{4\Rmod \Lambda \| \nabla \bm{J}\|_{L^2 }^{2}}  \int | \nabla^2 \bm{u}|^2 \, \mathrm{d}V.
\end{aligned}
\end{equation}
Finally, by  Poincare's inequality and Young's inequality with $ \delta>0$, it holds
\begin{equation}
\label{eq:enest_fin}
    C_m \| \nabla \bm{u}\|_{L_2} \| \nabla \bm{J}\|_{L^2}^2 \leq\left( \frac{C_{m}^2 C_{\mathrm{P}}^2}{4\delta }\Ek^{(\gamma -1)} +  \delta \Ek^{(1-\gamma)} \| \nabla^2 \bm{u}\|_{L^2}^2\right)\| \nabla \bm{J}\|_{L^2}^2.
\end{equation}
Combining Proposition \ref{prop2} with \eqref{eq:penultJid}, \eqref{eq:penultJest}, \eqref{eq:ensest_v2} and \eqref{eq:enest_fin}, it follows that
\begin{equation}
    \begin{aligned}
            \frac{1}{2} \frac{\mathrm{d}}{\mathrm{d}t} \int |\bm{J}|^2 \, \mathrm{d}V  \leq& \,\Bigg (  C_n \Rmod \,\|\overline{\nabla \times \nabla \times T \bm{e}_z} \|_{L^2} +\Ek^{1-\gamma} \Rmod^2 \,\| \nabla \times \overline{\nabla \times \nabla \times T \bm{e}_z}\|_{L^2}^2  \\  &  \qquad + \frac{C_{m}^2 C_{\mathrm{P}}^2}{4\delta } \Ek^{(\gamma -1)}- \frac{1}{\Rm}\Bigg)\int  | \nabla \bm{J}|^2 \, \mathrm{d}V \\
     &+ \frac{\Ek}{\Lambda } \left( \frac{\delta \Lambda \| \nabla\bm{J}\|_{L^2}^2 }{ \Ek^{\gamma} } +  \frac{\Ek^{\gamma }}{4\Lambda  \| \nabla \bm{J}\|_{L^2}^2 } - 1\right) \int |\nabla^2 \bm{u}|^2 \, \mathrm{d}V.
    \end{aligned}
    \label{eq:final_J_ineq}
\end{equation}
Observe that the non-trivial choice of $\gamma$ which will give an analogous result for the inviscid case where $Ek \to 0$ is of $0<\gamma <1$.
Growth of magnetic enstrophy requires that the expression in the parentheses is nonnegative. Indeed, if the second parentheses in \eqref{eq:final_J_ineq} is positive, with
\begin{equation}
    x = \frac{4 \Lambda \| \nabla\bm{J}\|_{L^2}^2 }{ \Ek^{\gamma} },
\end{equation}
we get
\begin{equation}
    \frac{\delta x}{4} + \frac{1}{x} - 1 > 0.
\end{equation}
Solving for $x$ gives
\begin{equation}
    x> \frac{2}{\delta }( 1 + \sqrt{1-\delta}),\quad \text{or}\quad  x <\frac{2}{\delta }( 1 - \sqrt{1-\delta}).
\end{equation}
For $\delta \geq 1$, the condition is true for all $x$. For $\delta < 1$, a simple, though weaker, condition is that $\sqrt{1-\delta} \geq 1 - \delta$. This gives
\begin{equation}
    x > \frac{4}{\delta} - 2, \quad \text{or} \quad x < 2.
\end{equation}
The result is exactly the choice of $\delta $ such that 
\begin{equation}
\frac{C_{m}^2 C_{\mathrm{P}}^2 \, \Ek^{(\gamma -1)}}{4\delta } = \frac{1}{2 \Rm}.
\end{equation}
    
\end{proof}

\begin{proof}[Proof of Lemma \ref{lem:mixedtermbd}]
Let us rewrite the two terms. For the first term, because $\bm{B}$ and $\bm{J}$ are divergence free, it follows that pointwise
\begin{equation}
    \bm{J} \times \nabla^2 \bm{B} = - \bm{J}\times (\nabla \times \bm{J}) =   (\bm{J}\cdot \nabla) \bm{J} - \frac12 \nabla |\bm{J}|^2.
\end{equation}
Integrating by parts against divergence-free $\bm{u}$ yields and using the boundary conditions on $\bm{J}$ gives
\begin{equation}
\begin{aligned}
    \int \bm{u} \cdot (\bm{J} \times \nabla^2 \bm{B})\, \mathrm{d}V &= \int \bm{u} \cdot \left( \bm{J}\cdot \nabla\right) \bm{J} \, \mathrm{d}V\\
    &= \int \bm{u} \cdot (\nabla \cdot  \left( \bm{J}\otimes \bm{J} \right)) \,  \mathrm{d}V=- \int \nabla \bm{u} :   (\bm{J}\otimes \bm{J}) \, \mathrm{d}V.
    \label{eq:mixedJBterm}
\end{aligned}
\end{equation}
    For the second term, integration by parts yields
\begin{align}
   - 2\int \bm{u} \cdot (\nabla \cdot (\bm{J} \times \nabla \bm{B}) )\, \mathrm{d}V &= 2 \int \nabla\bm{u} :( \bm{J} \times \nabla \bm{B} )\, \mathrm{d}V.
\end{align}
Therefore,
\begin{equation}
\begin{aligned}
 \int  \bm{u} \cdot (\bm{J} \times \nabla^2  \bm{B} - 2 \nabla \cdot (\bm{J} \times \nabla \bm{B})) \, \mathrm{d}V  &= \int  \nabla\bm{u} :( 2\bm{J} \times \nabla \bm{B} - \bm{J}\otimes \bm{J}  ) \, \mathrm{d}V \\
&\leq 3\int |\nabla \bm{u}| |\nabla \bm{B}  |  |\bm{J}| \, \mathrm{d}V.
\end{aligned}
\end{equation}
Hölder's inequality implies
\begin{equation}
    \int |\nabla \bm{u}| |\nabla \bm{B}  |  |\bm{J}| \, \mathrm{d}V
\leq
\|\nabla \bm u\|_{L^2}
\|\nabla \bm B\|_{L^3}
\|\bm J\|_{L^6},
\end{equation}
and by $L^p$ interpolation we can write
\begin{equation}
    \| \nabla \bm{B} \|_{L^3} \leq \| \nabla \bm{B} \|_{L^2}^{1/2} \|  \nabla \bm{B}\|_{L^6}^{1/2}.
\end{equation}
Using the Sobolev estimate for $\nabla \bm{B}$ and $\bm{J}$,
\begin{equation}
    \|\nabla \bm B\|_{L^6}
\leq C_{\mathrm{S}} \|\nabla \nabla  \bm{B}\|_{L^2} = C_{\mathrm{S}}\|\nabla \bm J\|_{L^2},
\end{equation}
and the Poincar{\'e} inequality, we conclude that
\begin{equation}
 \int |\nabla \bm{u}| |\nabla \bm{B}  |  |\bm{J}| \, \mathrm{d}V \leq C_{\mathrm{S}}^{3/2} C_{\mathrm{P}}^{1/2} \|\nabla \bm{u} \|_{L^2} \| \nabla \bm{J} \|_{L^2}^2.
\end{equation}
Therefore, setting $C_m = 3C_{\mathrm{P}}^{1/2}C_{\mathrm{S}}^{3/2}$ completes the proof.
\end{proof}

\section{Proof of Theorems: Thermal Estimates}
\label{sec:temperature_estimates}
\subsection{Inviscid Case}
We first identify radii of the absorbing balls for temperature quantities which control the quantities of interest in \cref{sec:bounds}. 
The results are obtained by a standard use of energy methods on \eqref{eq:temperature_eq}. 

\begin{proposition}
\label{lemma:T_Lp}
 Suppose $T$ satisfies the temperature equation and homogeneous isothermal boundary conditions, for divergence-free $\bm{u}$ satisfying no-penetration boundary conditions.  
 Let $2 \leq p < \infty $ and suppose that $Q \in L^p$ is time independent. Then,
\begin{equation*}
     \limsup_{t \to \infty}\| T(t) \|_{L^p} \leq  \frac{p^2 \|Q \|_{L^p}}{4\pi^2(p-1)}.
\end{equation*}
\end{proposition}
\begin{proof}
Multiplying \eqref{eq:temperature_eq} by $|T|^{p-2}T$ where $p\geq 2$, integrating over the domain and using the boundary conditions,  incompressibility and integrating by parts gives
\begin{equation}
    \label{eq:T_energy_eq_Lp}
     \frac{1}{p} \frac{\mathrm{d}}{\mathrm{d}t} \int |T|^p \, \mathrm{d}V  =  \frac{1}{\Pe}\int Q\,|T|^{p-2}T \, \mathrm{d}V -  \frac{1}{\Pe}\int \nabla (|T|^{p-2}T)\cdot  \nabla T \,\mathrm{d}V.
\end{equation}
The dissipative term can be simplified by rearranging and using the Poincar\'e inequality,
\begin{align}
    \int \nabla (|T|^{p-2}T) \nabla T \,\mathrm{d}V &= (p-1)\int|\nabla T|^2 T^{p-2} \mathrm{d}V = (p-1)\int|T^{p/2 - 1}\nabla T |^2 \mathrm{d}V \nonumber \\ &= \frac{4(p-1)}{p^2} \int|\nabla |T|^{p/2} |^2 \mathrm{d}V \geq \frac{4\pi^2(p-1)}{p^2} \int |T|^p \mathrm{d}V.  \label{eq:dissipative_T_lp}
\end{align}
Substituting \eqref{eq:dissipative_T_lp} into \eqref{eq:T_energy_eq_Lp}, and using H\"olders inequality on the source term gives
\begin{equation}
     \frac{1}{p} \frac{\mathrm{d}}{\mathrm{d}t} \| T\|_{L^p}^p  \leq  \frac{ \|Q \|_{L^p}}{\Pe} \| T\|_{L^p}^{(p-1)} -  \frac{4\pi^2 (p-1)}{\Pe \, p^2} \| T\|_{L^p}^p.
\end{equation}
Differentiating the left hand side and dividing through by $\| T\|_{L^p}^{(p-1)}$, we get
\begin{equation}
   \frac{\mathrm{d}}{\mathrm{d}t} \| T\|_{L^p} \leq   \frac{\|Q \|_{L^p}}{\Pe}  -  \frac{4\pi^2(p-1)}{\Pe\, p^2} \| T\|_{L^p}.
\end{equation}
Using Gronwall's lemma finally gives the desired result
\begin{align}
     \| T(t)\|_{L^p} \leq  &\|T(0)\|_{L^p} \exp \left\{-\frac{4\pi^2 (p-1)t}{\Pe\, p^2} \right\} \nonumber \\&+ \frac{ p^2\|Q\|_{L^p}}{4\pi^2  (p-1)} \left(1-\exp \left\{-\frac{4\pi^2 (p-1)t}{\Pe\, p^2} \right\}\right),
\end{align}
after taking the limit superior as time $t \to \infty$.
\end{proof}

\begin{proposition}
\label{lemma:pointwise_t_laplace_T}
 Suppose $T$ satisfies the temperature equation and isothermal boundary conditions, for divergence-free $\bm{u}$ satisfying no-penetration boundary conditions.  
 Suppose that $Q \in H^1$ is time independent. Let,
 \begin{equation*}
     U_\infty Pe < \frac{\pi}{\sqrt{1+\pi^2}},
 \end{equation*}
 then,
\begin{equation*}
\limsup_{t\to\infty}\|\nabla^2T(t)\|_{L^2}\leq \left(\frac{ U_\infty \Pe \,  \|Q\|_{H^1}}{\pi  - \sqrt{1+\pi^2}  U_\infty \Pe}+\|Q\|_{L^2}\right).
\label{eq:laplace_T_absorbing_radius_improved}
\end{equation*}
\end{proposition}
\begin{proof}
Define
\begin{equation}
\vartheta:=\nabla^2T+Q.
\label{eq:Phi_definition}
\end{equation}
Tracing \eqref{eq:temperature_eq} to the boundaries $S_{\pm}$, the isothermal boundary conditions for $T$ and no-penetration boundary conditions for $\bm{u}$ yield
\begin{equation}
\vartheta |_{\, S_{\pm}} = \left(\nabla^2T+Q\right) |_{\, S_{\pm}}=0.
\label{eq:Phi_boundary_condition}
\end{equation}
Because $Q$ is time independent, applying the Laplacian to \eqref{eq:temperature_eq} gives
\begin{equation}
\partial_t\vartheta +\nabla^2(\bm u\cdot\nabla T)=\frac{1}{\Pe}\nabla^2\vartheta.
\label{eq:Phi_equation}
\end{equation}
Multiplying \eqref{eq:Phi_equation} by $\vartheta$, integrating over $V$, and using \eqref{eq:Phi_boundary_condition}, we obtain
\begin{align}
\frac12\frac{\mathrm d}{\mathrm dt}\|\vartheta\|_{L^2}^2&=-\int\vartheta\nabla^2(\bm u\cdot\nabla T),\mathrm dV+\frac{1}{\Pe}\int\vartheta\nabla^2\vartheta\mathrm dV \nonumber\\
&=\int\nabla\vartheta\cdot\nabla(\bm u\cdot\nabla T)\mathrm dV-\frac{1}{\Pe}\|\nabla\vartheta\|_{L^2}^2.
\label{eq:Phi_energy_identity}
\end{align}
Expanding the nonlinear term gives
\begin{align}
\int\nabla\vartheta\cdot\nabla(\bm u\cdot\nabla T)\mathrm dV&=\int\nabla\vartheta\cdot(\nabla\bm u)^T\nabla T\mathrm dV+\int\nabla\vartheta\cdot(\bm u\cdot\nabla)\nabla T\mathrm dV.
\label{eq:Phi_nonlinear_expansion}
\end{align}
By H\"older's inequality,
\begin{align}
\left|\int\nabla\vartheta\cdot\nabla(\bm u\cdot\nabla T)\mathrm dV\right|&\leq (\|\bm u\|_{L^3}\|\nabla\nabla T\|_{L^6} +\|\nabla\bm u\|_{L^2}\|\nabla T\|_{L^\infty} )  \|\nabla\vartheta\|_{L^2}
\label{eq:Phi_advection_holder}
\end{align}
Next, given that $C_\infty$ and $C_{\mathrm{S}}$ be admissible Sobolev--elliptic constants satisfying
\begin{equation}
\|\nabla T\|_{L^\infty}\leq C_\infty\|\nabla^2T\|_{H^1}, \qquad 
\|\nabla\nabla T\|_{L^6}\leq C_{\mathrm{S},0}\|\nabla^2T\|_{H^1}.
\label{eq:hessian_T_L6_embedding_laplace_bound}
\end{equation}
Thus  \eqref{eq:Phi_advection_holder} and \eqref{eq:hessian_T_L6_embedding_laplace_bound} motivate the definition \eqref{eq:Mdef} such that
\begin{equation}
\left|\int\nabla\vartheta\cdot\nabla(\bm u\cdot\nabla T)\mathrm dV\right|\leq U_{\infty}\|\nabla^2T\|_{H^1} \|\nabla\vartheta\|_{L^2}.
\label{eq:Phi_advection_elliptic}
\end{equation}
On the other hand, by applying the Poincar\'e inequality to the definition of $\vartheta$,
\begin{equation}
\|\nabla^2T\|_{H^1}\leq \left(1+ \frac{1}{\pi^2}\right)^{1/2}\|\nabla\vartheta\|_{L^2}+\|Q\|_{H^1}.
\label{eq:laplace_T_H1_initial}
\end{equation}
Therefore, from \eqref{eq:Phi_advection_elliptic},
\begin{equation}
\left|\int\nabla\vartheta\cdot\nabla(\bm u\cdot\nabla T)\mathrm dV\right|\leq\left(1+\frac{1}{\pi^2}\right)^{1/2}U_{\infty}\|\nabla\vartheta\|_{L^2}^2+U_{\infty}\|Q\|_{H^1}\|\nabla\vartheta\|_{L^2}.
\label{eq:Phi_advection_final}
\end{equation}
Combining \eqref{eq:Phi_energy_identity} and \eqref{eq:Phi_advection_final} yields
\begin{equation}
\frac12\frac{\mathrm d}{\mathrm dt}\|\vartheta\|_{L^2}^2\leq-\gamma\|\nabla\vartheta\|_{L^2}^2+U_{\infty}\|Q\|_{H^1}\|\nabla\vartheta\|_{L^2}.
\label{eq:Phi_pre_young}
\end{equation}
where 
\begin{equation}
    \gamma = \frac{1}{\Pe}-\left(1+\frac{1}{\pi^2}\right)^{1/2} U_{\infty}
\end{equation}
For any $\varepsilon > 0$, Young's inequality gives
\begin{equation}
U_{\infty}\|Q\|_{H^1}\|\nabla\vartheta\|_{L^2} \leq\varepsilon\gamma\|\nabla\vartheta\|_{L^2}^2 +\frac{U_{\infty}^2\|Q\|_{H^1}^2}{4\varepsilon\gamma}.
\label{eq:Phi_young}
\end{equation}
Substituting \eqref{eq:Phi_young} into \eqref{eq:Phi_pre_young} with $\varepsilon=1/2$ gives
\begin{equation}
\frac12\frac{\mathrm d}{\mathrm dt}\|\vartheta\|_{L^2}^2\leq-\frac12\gamma\|\nabla\vartheta\|_{L^2}^2+\frac{U_{\infty}^2\|Q\|_{H^1}^2}{2\gamma}.
\label{eq:Phi_after_young}
\end{equation}
Using Poincar\'e's inequality again,
\begin{equation}
\|\nabla\vartheta\|_{L^2}^2\geq\pi^2\|\vartheta\|_{L^2}^2,
\label{eq:Phi_reverse_poincare}
\end{equation}
we obtain
\begin{equation}
\frac{\mathrm d}{\mathrm dt}\|\vartheta\|_{L^2}^2\leq-\gamma\pi^2\|\vartheta\|_{L^2}^2+\frac{U_{\infty}^2\|Q\|_{H^1}^2}{\gamma}.
\label{eq:Phi_scalar_ODE}
\end{equation}
Gronwall's inequality gives
\begin{align}
\|\vartheta(t)\|_2^2&\leq\|\vartheta_0\|_{L^2}^2\exp\{-\gamma\pi^2t\}+\frac{U_{\infty}^2\|Q\|_{H^1}^2}{\gamma^2\pi^2}\left(1-\exp\{-\gamma\pi^2t\}\right).
\label{eq:Phi_gronwall}
\end{align}
Therefore,
\begin{equation}
\limsup_{t\to\infty}\|\vartheta(t)\|_{L^2}\leq\frac{U_{\infty}\|Q\|_{H^1}}{\gamma\pi}.
\label{eq:Phi_limsup}
\end{equation}
Finally, 
we have
\begin{equation}
\|\nabla^2T\|_{L^2}\leq\|\vartheta\|_{L^2}+\|Q\|_{L^2}.
\label{eq:Delta_T_from_Phi}
\end{equation}
Combining \eqref{eq:Phi_limsup} and \eqref{eq:Delta_T_from_Phi} gives
\begin{equation}
\limsup_{t\to\infty}\|\nabla^2T(t)\|_{L^2}^2\leq\left(\frac{U_{\infty}\|Q\|_{H^1}}{\gamma\pi}+\|Q\|_{L^2}\right)^2.
\end{equation}
This proves the result.
\end{proof}

\begin{proof}[\underline{Proof of Theorem \ref{lemma:nablaT_L3}}]
In fact 
\begin{equation}
    \| \nabla_{x,y} \overline{T}\|_{L^3} = \| \overline{\nabla_{x,y} T}\|_{L^3} \leq b_3 \| \nabla_{x,y} T\|_{L^3} \leq b_3 \| \nabla T\|_{L^3}
\end{equation}
where $b_3$ is the constant given in Lemma \ref{lemma2}.
Then, by $L^p$ interpolation we have that
\begin{equation}
    \| \nabla T \|_{L^3} \leq \| \nabla T \|_{L^2}^{1/2} \| \nabla T \|_{L^6}^{1/2}.
\end{equation}
Then, using the temperature boundary conditions, we can integrate by parts and use the Cauchy--Schwarz inequality to write that
\begin{equation}
    \int |\nabla T|^2 \mathrm{d}V = - \int T \nabla^2 T \mathrm{d}V \leq \| T \|_{L^2} \| \nabla^2 T \|_{L^2}.
    \label{eq:grad_T_estS_s}
\end{equation}
Using \eqref{eq:grad_T_estS_s} and the Sobolev embedding $H^1 \hookrightarrow L^6$ gives that
\begin{equation}
     \| \nabla T \|_{L^3} \leq  C_{\mathrm{S}}^{1/2} \| T \|_{L^2}^{1/4} \| \nabla^2 T \|_{L^2}^{3/4}.
\end{equation}
Taking long-time limit, we have that
\begin{equation}
    \| T \|_{L^2}\|\nabla^2 T \|_{L^2} \leq \limsup_{t\to\infty} \left(  \| T \|_{L^2} \|\nabla^2 T \|_{L^2} \right) \leq   \left(\limsup_{t\to\infty}\| T \|_{L^2} \right) \left( \limsup_{t\to\infty}\|\nabla^2 T \|_{L^2} \right).
\end{equation}

Therefore, using the bound on $\nabla^2 T \in L^2$ from Proposition \ref{lemma:pointwise_t_laplace_T} and on $T \in L^2$ from Proposition \ref{lemma:T_Lp} with $p=2$, the result follows.

\end{proof}
\begin{proof}[\underline{Proof of Theorem \ref{lemma:poltorT}}]
We shall establish
\begin{equation}
\label{eq:L2hL2est}
    \| \nabla_{x,y}^2 \overline{T} \|_{L^\infty_z L^2_{x,y}} \leq \| \nabla^2 T \|_{L^2}
\end{equation}
and the claim follows from an application of Proposition \ref{lemma:pointwise_t_laplace_T}.

A basic property of the integral is
\begin{equation}
    \left\| \int_0^z \nabla_{x,y}^2 T(z')\, dz'\right\|_{L^2_{x,y}} \leq  \int_0^z\|  \nabla_{x,y}^2 T(z')\|_{L^2_{x,y}}dz'.
\end{equation}
Taking the supremum in $z$, it holds
    \begin{equation}
            \| \nabla_{x,y}^2 \overline{T} \|_{L^\infty_z L^2_{x,y}} \leq \| \nabla_{x,y}^2 T \|_{L^2}.
    \end{equation}
    Now we establish
    \begin{equation}
    \label{eq:Lapbd}
        \| \nabla_{x,y}^2 T \|_{L^2} \leq \| \nabla^2 T \|_{L^2}.
    \end{equation}
    Observe that
\begin{equation}
\langle \nabla^2 T, \nabla^2 T \rangle = \langle \hord
T, \hord T \rangle + 2 \langle \partial_{z}^2 T,
\hord T \rangle + \langle \partial_{z}^2 T, \partial_{z}^2
 T \rangle. 
\end{equation}
 The cross term we may integrate by parts in $dz$,
 \begin{equation}
    \langle  \partial_{z}^2 T,  \hord T \rangle = -
    \langle  \partial_{z} T,  \hord \partial_{z} T \rangle \pm 
    \int
      (\partial_{z} T)  \hord T\,dS_{\pm }.
 \end{equation}
 For isothermal boundary conditions $\hord T$ vanishes on
 the boundary $S_{\pm}$. Therefore we may further integrate by parts the cross
 term in $\mathrm{d}S$ such that
  \begin{equation}
  \langle \partial_{z}^2 T, \hord T \rangle  = \langle \horg 
  \partial_{z} T, T  \partial_{z} T \rangle \geq 0.
 \end{equation}
 This establishes estimate \eqref{eq:Lapbd}, and the result follows.
\end{proof}


\begin{remark}
    For Rayleigh--B\'enard convection where the boundary conditions are inhomogeneous, one may consider  perturbations about the conductive state. One then finds an evolution equation equivalent to the above,  and the results hold with $u_z$ in place of $Q$. 
\end{remark}

\subsection{Viscous Case}
\label{sec:visccombined}
We first state and prove two lemmas that estimate norms of primitive variables, that are used in the results of \cref{sec:temp_classical_R}.

\begin{lemma}
\label{lemma2}
Let $1\leq p<\infty$ and  $\bm{f}\in L^p(\mathrm{d}V)$. 
Then,
\begin{equation*}
\|\overline{f}\|_{L^p} \leq
b_p\| f\|_{L^p}, \label{eq:sharp_vertical_primitive_bound}
\end{equation*}
where
\begin{equation*}
b_p = 
\begin{cases}
    1& p=1, \\
    \frac{p\sin(\pi/p)}{\pi(p-1)^{1/p}} & p > 1.
\end{cases}
\label{eq:sharp_volterra_constant}
\end{equation*}
\end{lemma}

\begin{proof}
If one considers $\overline{\,\cdot\,\vphantom{i}}$ as an operator on $L^p_z$, it is in fact Volterra and has the operator norm $b_p$ per \cite[Theorem~3.1]{bennewitz2002embedding}. Noting that $L^p(dV) = L^p_{x,y} L^p_z$, the result follows.
\end{proof}

\begin{lemma}
\label{lem:primitive_laplacian}
Let $f\in H^2 $ and suppose that $f \, |_{\, S_{\pm}}=0$.
Then
\begin{equation*}
\|\nabla \nabla_{x,y} \overline{f}\|_{L^2} \leq  (C_{\mathrm{P}} + b_2)\|\nabla^2 f\|_{L^2}.
\label{eq:primitive_laplacian_bound}
\end{equation*}
\end{lemma}

\begin{proof}
Because $\horg f$ vanishes at the boundary, it holds
\begin{equation}
\nabla \nabla_{x,y} \overline{f} = (\nabla_{x,y} \nabla_{x,y} \overline{f},  \nabla_{x,y} f).
\label{eq:primitive_laplacian_identity}
\end{equation}
We have the identity
\begin{equation}
    \|\nabla_{x,y} \nabla_{x,y} \overline{f}\|_{L^2} = \| \overline{\nabla_{x,y} ^2f}\|_{L^2}
\end{equation}
Therefore, by  Lemma \ref{lemma2} and Lemma \ref{surfp}
\begin{equation}
\begin{aligned}
    \|\nabla \nabla_{x,y} \overline{f}\|_{L^2} &\leq \| \overline{\nabla_{x,y} ^2f}\|_{L^2} + \|\nabla_{x,y} f\|_{L^2} \\
    &\leq ( b_2 + C_{\mathrm{P}}) \|\nabla_{x,y}^2 f\|_{L^2}
    \end{aligned}
\end{equation}
With the previous estimate \eqref{eq:Lapbd}, we are done.
\end{proof}
\begin{proof}[\underline{Proof of Theorem \ref{thm:viscous_classical}}]
Following the proof of Theorem \ref{thm:visc_main} utilising the poloidal-toroidal decomposition estimates Propositions \ref{prop:ptod} and \ref{prop:ptfm}, we find that dynamo action requires that
    \begin{equation*}
         \Rmod \, A_2 \| \nabla^2_{x,y} \overline{T}\|_{L^{\infty_z}L_{x,y}^2} +  \Rmod^2 \, \Ek^{1 - \gamma}\| \nabla \nabla_{x,y}\overline{T}\|^2_{L^2} \, \geq \frac{1}{\Rm}
    \end{equation*}
    or 
        \begin{equation*}
4\Lambda \int | \nabla \bm{B}|^2 \, \mathrm{d}V \leq \Ek^{\gamma}.
 \end{equation*}
 Applying the estimate \eqref{eq:L2hL2est} and Lemma \ref{lemma2}, the first requirement  is
    \begin{equation}
                 \Rmod \, A_2 \| \nabla^2 T\|_{L^2} +  \Rmod^2 \, \Ek^{1 - \gamma} C_{\mathrm{P},2}^2\| \nabla^2 T\|_{L^2}^2 \, \geq \frac{1}{\Rm}.
    \end{equation}
Defining $x = \Rmod_{\nu}\| \nabla^2 T\|_{L^2}$ and using  relation \eqref{eq:OG_rayleigh_relation} of the classical Rayleigh number $\Rmod_{\nu}$, the requirement for dynamo action is simply
\begin{equation}
     C_{\mathrm{P},2}^2 \frac{\Ek^{3-\gamma}}{\Pe^2} x^2+ A_2 \frac{\Ek}{\Pe}  x   -\frac{1}{\Rm} \geq 0.
\end{equation}
Solving the quadratic furnishes the result.
\end{proof}

\begin{proof}[\underline{Proof of Theorem~\ref{thm:viscous_combined_energy}}]
First, assume that $Q=Q(z)$ and let $\tau$ be the conductive state solving \eqref{eq:conductive_temperature_problem}. The evolution equation for  $\theta=T-\tau$ is
\begin{equation}
\partial_t\theta+\bm{u}\cdot\nabla\theta+u_z\tau'=Pe^{-1}\nabla^2\theta
\label{eq:temperature_perturbation_equation}
\end{equation}
and the perturbation vanishes at the boundary,
\begin{equation}
    \theta=0 \quad \text{on }S_{\pm}.
\end{equation}

Testing \eqref{eq:temperature_perturbation_equation} against $-\nabla^2\theta$ and integrating gives
\begin{equation}
\frac{1}{2}\frac{\mathrm{d}}{\mathrm{d}t}\int |\nabla\theta|^2 \mathrm{d}V =-Pe^{-1}\int |\nabla^2\theta|^2\mathrm{d}V + \int \nabla^2\theta\,\bm{u}\cdot\nabla\theta\,\mathrm{d}V+\int\tau' u_z \nabla^2\theta\,\mathrm{d}V.
\label{eq:temperature_gradient_energy_identity}
\end{equation}
By H\"older's inequality and the estimate
\begin{equation}
    \| \nabla \theta \|_{L^6} \leq C_{\mathrm{S},0} \| \nabla^2 \theta  \|_{L^2}
\end{equation}
we get
\begin{equation}
\left|\int\nabla^2 \theta\,\bm{u}\cdot\nabla\theta\,\mathrm{d}V\right|\leq \|\bm{u} \|_{L^3} \|\nabla \theta \|_{L^6} \| \nabla^2 \theta \|_{L^2} \leq C_{\mathrm{S},0}\|\bm{u}\|_{L^3}\|\nabla^2\theta\|_{L^2}^2.
\label{eq:temperature_advection_estimate}
\end{equation}
For the final term in \eqref{eq:temperature_gradient_energy_identity}, using Cauchy-Schwarz in the horizontal directions and then H\"olders in $z$ gives,
\begin{align}
   \left| \int \tau' u_z \nabla^2 \theta \mathrm{d}V\right| \leq \left|\int^{1}_0 \tau' \| u_z \|_{L^{2}_{x,y}} \| \nabla^2 \theta \|_{L^{2}_{x,y}} \mathrm{d}z \right| \leq \|\tau' \|_{L^2_z} \| u_z \|_{L^\infty_zL^{2}_{x,y}}  \| \nabla^2 \theta \|_{L^{2}}.
   \label{eq:first_heating_est}
\end{align}
Since $u_z=0$ on $S_{\pm}$, exploiting both boundary conditions, 
\begin{align}
    u_z = (1-z)\int^z_0 \partial_{z'} u_z \mathrm{d}z' - z\int^1_z\partial_{z'} u_z \mathrm{d}z' = \int^{1}_0 \xi_z(z') \partial_{z'} u_z \mathrm{d}z' 
\end{align}
where $\xi_z$ is defined as
\begin{equation}
 \xi_z(z') = 
\begin{dcases}
    1-z, & 0 \leq z' \leq z,\\
    -z, & z \leq z' \leq 1.
\end{dcases}
\end{equation}
Then, taking the horizontal $L^2$ norm of $u_z$, using Cauchy-Schwarz in $z$ and evaluating the integral of $|\xi_z(z')|^2$ gives
\begin{align}
    \| u_z \|_{L^{2}_{x,y}} & = \left\| \int^{1}_0\xi_z(z') \partial_{z'} u_z \mathrm{d}z' \right\|_{L^{2}_{x,y}} \leq \| \xi_z \|_{L^{2}_z} \|\partial_z u_z \|_{L^2}  \nonumber 
    \\&\leq  \| \xi_z \|_{L^{2}_z} \|\nabla \bm{u} \|_{L^2} = \sqrt{z(1-z)} \|\nabla \bm{u} \|_{L^2}  .\label{eq:u_z_est_c}
\end{align}
Taking the supremum yields $\| u_z \|_{L^{\infty }_zL^{2}_{x,y}} \leq \frac{1}{2}\| \nabla \bm{u}\|_{L^2} $. Substituting this bound  into \eqref{eq:first_heating_est}, multiplying through by $Pe$ and using Young's inequality gives
\begin{align}
     Pe\left| \int \tau' u_z \nabla^2 \theta \mathrm{d}V\right| &\leq Pe\|\tau' \|_{L^2_z} \| \nabla \bm{u} \|_{L^{2}}  \| \nabla^2 \theta \|_{L^{2}} \nonumber \\
     &\leq Ek(1-x) \int |\nabla \bm{u}|^2 \mathrm{d}V + \frac{Pe^2 \| \tau'\|^2_{L^2_z}}{4Ek(1-x)} \int |\nabla^2 \theta|^2 \mathrm{d}V
     .
    \label{eq:conductive_gradient_estimate}
\end{align}
Substituting \eqref{eq:conductive_gradient_estimate} and \eqref{eq:temperature_advection_estimate} into \eqref{eq:temperature_gradient_energy_identity} gives
\begin{align}
    \frac{Pe}{2} \frac{\mathrm{d}}{\mathrm{d}t} \int |\nabla \theta|^2 \mathrm{d}V \leq &\left(-1+C_{\mathrm{S},0} Pe \|\bm{u}\|_{L^3} + \frac{Pe^2 \| \tau'\|_{L^2_z}^2}{4Ek(1-x)}\right)\int |\nabla^2 \theta|^2 \mathrm{d}V \nonumber \\
    & +  Ek(1-x)\int |\nabla \bm{u}|^2 \mathrm{d}V
    \label{eq:temp_grad_est_combined}
\end{align}


Next, testing the momentum equation \eqref{eq:momentum_eq} against $\bm{u}$ and combining the result with $\bm{B}$ tested against the induction equation \eqref{eq:induction_eq} gives
\begin{equation}
\frac{1}{2}\frac{\mathrm{d}}{\mathrm{d}t}\int |\bm{B}|^2 \mathrm{d}V = \frac{Ra}{\Lambda}\int u_z\theta\,\mathrm{d}V-\frac{Ek}{\Lambda}\int |\nabla\bm{u}|^2 \mathrm{d}V-Rm^{-1}\int|\nabla\bm{B}|^2 \mathrm{d}V.
\label{eq:magnetic_energy_with_temperature_perturbation}
\end{equation}
For the production term, we shall use the previous estimate \eqref{eq:prodest_visc}
\begin{equation}
\left|\int u_z \theta \,\mathrm{d}V\right|\leq \Lambda A_p\|\horg \overline{\theta }\|_{L^p}\int |\nabla\bm{B}|^2 \mathrm{d}V +Ek\|\nabla\nabla_{x,y}\overline{T}\|_{L^2} \|\nabla\bm{u}\|_{L^2}
\label{eq:viscous_heat_transport}
\end{equation}

Applying Lemma~\ref{lem:primitive_laplacian}, estimate  \ref{eq:viscous_heat_transport} and Young's inequality with $x\in(0,1)$, gives
\begin{align}
\frac{1}{2}\frac{\mathrm{d}}{\mathrm{d}t}\int |\bm{B}|^2 \mathrm{d}V\leq&\, (Ra\,A_p\|\horg \overline{T}\|_{L^p}-Rm^{-1})\int |\nabla\bm{B}|^2 \mathrm{d}V - \frac{Ek}{\Lambda}(1-x)\int |\nabla\bm{u}|^2 \mathrm{d}V \nonumber \\ &+\frac{C^2_{\mathrm{P},2} Ra^2Ek}{4\Lambda x}\int |\nabla^2\theta|^2 \mathrm{d}V,
\label{eq:magnetic_energy_temperature_perturbation_bound}
\end{align}
where $C_{\mathrm{P},2} = C_{\mathrm{P}} + b_2$.

Next we form the equation for $\mathcal{E} = \Lambda \| \bm{B} \|_{L^2}^2 + Pe \| \nabla \theta \|_{L^2}^2$.
Combining \eqref{eq:temp_grad_est_combined}
with $\Lambda$ multiplying \eqref{eq:magnetic_energy_temperature_perturbation_bound} cancels the terms proportional to $Ek(1-x)\| \nabla \bm{u} \|_{L^2}^2$ and gives
\begin{align}
\frac{1}{2}\frac{\mathrm{d}\mathcal{E}}{\mathrm{d}t}\leq~& \Lambda (Ra\,A_p\|\horg \overline{T}\|_{L^p}-Rm^{-1})\int |\nabla \bm{B}|^2 \mathrm{d}V \nonumber \\
&+ Pe\left(\frac{C^2_{\mathrm{P},2} Ra^2Ek}{4 x  Pe}+\frac{Pe\|\tau'\|_{L^2_z}^2}{4 Ek(1-x)}-Pe^{-1}+C_{\mathrm{S},0}\|\bm{u}\|_{L^3}\right)\int |\nabla^2\theta|^2\mathrm{d}V,
\label{eq:combined_energy_differential_inequality}
\end{align}
For $a,b\geq0$,
\begin{equation}
\inf_{0<x<1}\left(\frac{a}{x}+\frac{b}{1-x}\right)=(\sqrt{a}+\sqrt{b})^2.
\label{eq:young_parameter_minimisation}
\end{equation}
Therefore, minimising the second bracket in \eqref{eq:combined_energy_differential_inequality} gives
\begin{equation}
\left(\frac{C_{\mathrm{P},2}\,Ra\sqrt{Ek}}{ 2\sqrt{Pe}}+\frac{\sqrt{Pe}\| \tau'\|_{L^2_z}}{2\sqrt{Ek}}\right)^2-Pe^{-1}+C_{\mathrm{S},0}\|\bm{u}\|_{L^3}.
\label{eq:optimised_temperature_coefficient}
\end{equation}
Consequently, non-decay of $\mathcal{E}$ requires either of the two brackets in \cref{eq:combined_energy_differential_inequality} to be non-negative. 
\end{proof}

\section{Discussion}
This work establishes necessary conditions for the instantaneous growth of magnetic energy in rotating convection. We consider the magnetohydrodynamic model of J.~B.~Taylor for the Earth's outer core,  and its viscous extension. We prove antidynamo theorems: quantitative conditions that are necessary for magnetic energy growth in rapidly rotating convection. 
Our analysis identifies the horizontal variation of the vertically integrated temperature as the quantity controlling the necessary conditions for dynamo action.
The bounds therefore extend the logic of classical kinematic antidynamo estimates to a fully nonlinear setting.

A finding of our work is the significance of the Archimedean velocity $\bm{u}^{\mathrm{A}}$, i.e., the buoyancy-driven part of the flow. In particular, magnetic energy growth requires, for $3 \leq p \leq \infty $
\begin{equation*}
    \Rm \,A_p\|\bm{u}^{\mathrm{A}}\|_{L^p}\geq 1,
\end{equation*}
where $A_p$ is an explicit constant given in Appendix \ref{sec:appendix_constant}. The constraint holds for Taylor's model and for the  inertialess model with finite viscosity $\Ek > 0$.
The precise statements are Theorems \ref{corollary_0} and \ref{corollary_0_visc}.
We refine our bound via poloidal-toroidal decomposition to replace $L^p$ with $L^\infty_z \dot{H}^1_{x,y}$ in the above constraint.

\subsection{Taylor Dynamo}
Taylor's constraint complicates the analysis of the inertialess model. The production of magnetic energy in its basic form depends  on  the vertical (poloidal) component of the flow $u_z$; however, by identities from Taylor's constraint, the magnetic energy production is driven by a purely horizontal (toroidal) part of the flow 
$$\bm{u}^{\mathrm{A}} = \Rmod \,\nabla_{x,y}^{\perp} \overline{T}.$$
 Crucially, our results are independent of the geostrophic component of the flow which is typically difficult to compute. The precise statement is Theorem \ref{theorem_1}.

Our analysis without nondimensionalising yields the requirement for dynamo action
    \begin{equation*}
        \Rmod_{T} = \frac{g \alpha d^{1-\frac{3}{p}} \| \nabla_{x,y} \overline{T} \|_{L^p} }{2 \Omega \eta} \geq A^{-1}_p .
        \label{eq:nabla_T_3_b}
    \end{equation*}
    Therefore, for flows normalised by the temperature scale given by a ( dimensional) $\|\nabla_{x,y} \overline{T} \|_{L^p}$, a lower bound for $\Rmod$ is $A^{-1}_p$.  
In terms of a classical Rayleigh number $\Rmod_{\eta}$ our constraint for dynamo action has the form 
\begin{equation*}
    Ra_\eta \gtrsim q^{-1}Ek_\eta^{-1}.
\end{equation*}
In long-time convection, the implicit constant in this requirement depends on the radii of absorbing balls. For uniform internal heating, see Corollaries \ref{cor1} and \ref{cor3}.


\subsection{Viscous Dynamo}

At finite Ekman number $\Ek > 0$, viscosity removes the computational difficulties associated with the inviscid model.
Our requirements for dynamo action hold with viscous corrections which vanish as $\Ek \to 0$; our dynamo criteria are not based on singular properties of Taylor's constraint.

The viscous corrections are explicitly given in Theorem \ref{thm:visc_main}. Writing Theorem \ref{thm:visc_main} in terms of a classical Rayleigh number $Ra_\nu$, our requirement for dynamo action takes the form
\begin{equation*}
    Ra_\nu \gtrsim Pe Ek^{-3/2}.
\end{equation*}

In addition to the magnetic energy, in Theorem \ref{theorem:ohmic} we provide a necessary condition for the growth of magnetic enstrophy, given $\Rm$ is sufficiently small. 
The magnetic energy and enstrophy growth conditions are  complementary. If  growth conditions of magnetic energy enstrophy are both not satisfied, then the magnetic field and its gradients must decay.

Our final result concerns the thermo-magnetic energy $  \mathcal{E} = \Lambda\|\bm{B}\|_2^2+ Pe\|\nabla\theta\|_2^2 $ in Theorem \ref{thm:viscous_combined_energy}. The conductive profile affects the second combined energy condition through $\|\tau'\|_{L^2_z}$, where the conductive state corresponds to $Q(z)=-\tau''(z)$. For the combined energy to grow, at least one of two conditions must be satisfied. The first is the magnetic energy condition of Theorem \ref{corollary_0_visc}. The second result arises from a balance of thermal diffusion against the advection of thermal gradients, the conductive gradient $\tau'$ and the viscous dissipation. 
In the case of $Pe \sim Ek^{1/2}$ for $Ek \ll 1$ where $\mathcal{E} \to \Lambda \| \bm{B} \|^2_2$ and non-decay requires $Ra_\nu \gtrsim Ek^{-1}$. For  $Pe \sim 1$ and conductive states such that $\| \tau' \|_{L^2_z} \Ek^{-1/2} \ll 1$, the non-decay requirement is $Ra_\nu \gtrsim  Ek^{-3/2}$ as before. Such conductive states are generated by  heating profiles $Q = Q(z)$ which are oscillatory or concentrated at the boundaries.
 
\input{fig/ekman}
Unlike Taylor's model, the viscous inertialess model  permits the onset of convection from a conductive solution. In particular, there exists a critical Rayleigh number $Ra_{\nu}^{*}$, above which the flow is linearly unstable to convection. At the onset of convection, for Elsasser numbers of $\mathcal{O}(\varepsilon)$, it is known that for boundary and internal heating, $Ra_{\nu}^{*} \sim Ek^{-4/3}$ \cite{chandrasekhar1953,arslan2024internally}. For comparison, in the magnetostrophic limit where the Elsasser number is $\mathcal{O}(1)$, of rotating convection subject to a vertical background magnetic field, $Ra_{\nu}^{*} \sim Ek^{-1}$. This is not the self-consistent dynamo onset problem considered here, but provides a useful reference scaling.
Indicative scalings are plotted in \cref{fig:onset_vs_combined} along with data from successful dynamo simulations in the plane layer.


\subsection{Scope and Future Outlook}

Several qualifications should be kept in mind when interpreting these results. First, our conditions are strictly necessary, but are not sufficient for dynamo action. A smooth solution to the inertialess models may continuously satisfy our constraints without exhibiting any magnetic energy growth. We impose strictly impenetrable,  stress-free and ferromagnetic  boundary conditions however several of our arguments apply to spherical domains, \emph{mutatis mutandis}.

The numerical values of the thresholds use admissible rather than sharp constants (appendix \ref{sec:appendix_constant}). The constants depend on the domain. The temperature absorbing radii introduce additional losses when the entire temperature gradient is estimated instead of its dynamically relevant horizontal structures. The bounds should consequently be regarded as quantitative but conservative. 
The results in \cref{sec:bounds} can be applied to different equations governing the temperature, provided the buoyancy force retains the same structure. As such the results can apply to double-diffusive convection \cite{huppert1981double,garaud2018double}. An interesting application would be to alternative models like Maxwell-Cattaneo convection \cite{hughes2021maxwell}, convection with anisotropic or nonlinear thermal diffusion, or anelastic convection \cite{anufriev2005boussinesq}.

It would be useful to evaluate the necessary conditions using numerical dynamo solutions, both to measure the gap between the inequalities and observed onset and to identify which estimates produce the largest losses \cite{jones2000convection,hughes2016strong,cattaneo2017dynamo,hughes2019force}. Finally, the study of magnetohydrodynamics with recent mathematical tools would be valuable, like methods for bounding the long-time averages of quantities, used to study turbulence \cite{Fantuzzi2022}. Establishing such long-time constraints would turn the instantaneous obstructions derived here into a nonlinear analogue of the classical antidynamo theory for sustained magnetic fields.

\section*{Acknowledgments}
The authors foremost acknowledge Andrew Jackson for fruitful discussions and for invitation to the `Fluid flow and magnetic dynamos' conference at the Centro Stefano-Franscini in Ascona where we had the honor of meeting John Bryan Taylor.

AA acknowledges funding from the ERC (agreement no. 833848-UEMHP). HG thanks Peter Constantin, Emmanuel Dormy and Zhongtian (Kevin) Hu for insightful discussions on Taylor's model.
HG was primarily supported by the NSF
(grant no. DGE-2039656) and by the Ford Foundation.  HG 
further acknowledges support of the Institut Henri Poincaré (UAR 839 CNRS-Sorbonne Université), and LabEx CARMIN (ANR-10-LABX-59-01) for his participation in the scientific trimester `Mathematical developments in geophysical fluid dynamics' at  IHP.

\subsection*{Disclosure on AI Usage}
Generative artificial intelligence tools were used during manuscript preparation strictly for literature search, typesetting, proofreading and as an aid in checking the algebra. All mathematical arguments, results, references and final wording were independently reviewed and verified by the authors, who take full responsibility for the content.

\appendix

\newpage
\section{Admissible Values of Constants}
\label{sec:appendix_constant}

\begin{table}[ht]
    \centering
    \renewcommand{\arraystretch}{2}
    \begin{tabular}{c|c}
        Label  & Admissible Value \\
        \hline

        $\lambda_1$
        &
        $\displaystyle
        \min\left\{
            \pi,\frac{2\pi}{L_x},\frac{2\pi}{L_y}
        \right\}$
        \\

        $C_{\mathrm{P}}$
        &
        $\displaystyle \frac{1}{\lambda_1}$
        \\

        $C_3$
        &
        $\displaystyle
        \frac{2^{2/3}}{\pi^{2/3}\sqrt{3}}$
        \\

        $C_{\mathrm{S}}$
        &
        $\displaystyle
        2\sqrt{2}\,C_3
        \left(
            1+
            \frac{\sqrt{1+L_x^{-2}+L_y^{-2}}}{\lambda_1}
        \right)$
        \\
        $C_{\mathrm{S},0}$&
        $\displaystyle 2\sqrt{2} C_3 \sqrt{2 + L^{-2}_x +  L^{-2}_y}$ \\
        
        $A_p$
        &
        $\displaystyle
        C_{\mathrm{P}}^{(p-3)/p}C_{\mathrm{S}}^{3/p}$
        \\
        $C_L$ & $\displaystyle \left(\frac{3\pi}{32} \max\left\{ \frac{L_x}{L_y},\frac{L_y}{L_x} \right\} \right)^{1/4}$
        \\
        $C_1$ & $\displaystyle(C_L C_{\mathrm{P}}^{1/2})^2$
        \\
        $C_2$ & $\displaystyle \left(\frac{1}{6\pi}+ \frac{C_{\mathrm{P}}}{2} \right)(L_x L_y C_L^{2})^{-1}$   
        \\
        $A_2$ & $\displaystyle C_1 \max\{3,C_2\}$
        \\
        $\lambda_{mnk}$
        &
        $\displaystyle
        \left(\frac{2\pi m}{L_x}\right)^2
        +
        \left(\frac{2\pi n}{L_y}\right)^2
        +
        k^2\pi^2$
        \\

        $C_\infty$
        &
        $\displaystyle
        \left[
        \frac{2}{L_xL_y}
        \sum\limits_{m,n\in\mathbb Z}
        \sum\limits_{k\geq1}
        \frac{1}
        {\lambda_{mnk}(1+\lambda_{mnk}-k^2\pi^2)}
        \right]^{1/2}$
        \\

        $C_b$
        &
        $\displaystyle
        \left[
        \frac{1}{L_xL_y}
        \sum\limits_{
            \substack{
                m,n\in\mathbb Z,\ k\in\mathbb N_0\\
                (m,n,k)\neq(0,0,0)
            }
        }
        \frac{2-\delta_{k0}}
        {1+\lambda_{mnk}+\lambda_{mnk}^2}
        \right]^{1/2}$
        \\
         $C_m$
        &
        $\displaystyle 3C^{3/2}_{\mathrm{S}}C^{1/2}_P$
        \\
        $C_n$
        &
        $\displaystyle
        C_b\left(1+C_{\mathrm{P}}^2+C_{\mathrm{P}}^4\right)^{1/2}
        +
        C_m/3$
    \end{tabular}
    \caption{Explicit admissible values for the constants used in the estimates.}
    \label{table:constants}
\end{table}

\begin{proof}[Proof of $C_{\mathrm{S}}$]
Let $\bm B$ be horizontally periodic, solenoidal, and satisfy ferromagnetic boundary conditions in addition to net zero flux. Then the constant $C_{\mathrm{S}}$ appears in the inequality $\| \bm B \|_{L^6(V)} \leq C_{\mathrm{S}} \| \nabla \bm{B} \|_{L^2(V)}$ where $V=[0,L_x]\times[0,L_y]\times[0,1]$.

We construct an extension to which the whole-space Sobolev
inequality can be applied. Extend $B_x$ and $B_y$ oddly across
$z=0,1$, extend $B_z$ evenly across these planes, and extend all
three components periodically in $x$ and $y$. Denote the resulting
field by $\widetilde{\bm B}$. The boundary conditions ensure that
$\widetilde{\bm B}\in H^1_{\mathrm{loc}}(\mathbb R^3)$.

Consider the enlarged box
\begin{equation}
    W=(-L_x,2L_x)\times(-L_y,2L_y)\times(-1,2).
\end{equation}
Choose a product cutoff
\begin{equation}
    \eta(x,y,z)=\eta_x(x)\eta_y(y)\eta_z(z)
\end{equation}
such that $\eta=1$ on $V$, $\eta=0$ on $\partial W$, and each
factor varies linearly across the two adjoining cells.

To estimate the cutoff more precisely, set
\[
    \ell_x=L_x,\qquad \ell_y=L_y,\qquad \ell_z=1.
\]
For $i\in\{x,y,z\}$ and $s\in[0,\ell_i]$, the three values of the
corresponding cutoff factor satisfy
\begin{equation}
\begin{aligned}
    \eta_i(s-\ell_i)&=\frac{s}{\ell_i},\\
    \eta_i(s)&=1,\\
    \eta_i(s+\ell_i)&=1-\frac{s}{\ell_i}.
\end{aligned}
\end{equation}
Consequently,
\begin{equation}
    \sum_{j=-1}^{1}\eta_i(s+j\ell_i)^2
    =
    1+\left(\frac{s}{\ell_i}\right)^2
    +\left(1-\frac{s}{\ell_i}\right)^2
    \leq 2,
    \label{eq:cutoff_weight_sum}
\end{equation}
and
\begin{equation}
    \sum_{j=-1}^{1}
    \left|\eta_i'(s+j\ell_i)\right|^2
    =
    \frac{2}{\ell_i^2}.
    \label{eq:cutoff_derivative_sum}
\end{equation}

Set
\[
    \bm F=\eta\widetilde{\bm B}
\]
on $W$, and extend $\bm F$ by zero outside $W$. Then
$\bm F\in H^1(\mathbb R^3;\mathbb R^3)$ and $\bm F=\bm B$ on $V$.
By the product rule,
\begin{equation}
    \nabla\bm F
    =
    \eta\nabla\widetilde{\bm B}
    +
    \widetilde{\bm B}\otimes\nabla\eta .
    \label{eq:extension_product_rule}
\end{equation}

Decomposing $W$ into its $27$ translated and reflected copies of
$V$, and using \eqref{eq:cutoff_weight_sum} in each coordinate,
gives
\begin{equation}
\begin{aligned}
    \|\eta\nabla\widetilde{\bm B}\|_{L^2(W)}^2
    &\leq
    2^3\|\nabla\bm B\|_{L^2(V)}^2\\
    &=8\|\nabla\bm B\|_{L^2(V)}^2.
\end{aligned}
\end{equation}
Hence,
\begin{equation}
    \|\eta\nabla\widetilde{\bm B}\|_{L^2(W)}
    \leq
    2\sqrt{2}\,
    \|\nabla\bm B\|_{L^2(V)}.
    \label{eq:weighted_gradient_term}
\end{equation}

Similarly, since
\[
\begin{aligned}
    |\nabla\eta|^2
    ={}&
    |\eta_x'|^2\eta_y^2\eta_z^2
    +\eta_x^2|\eta_y'|^2\eta_z^2\\
    &+\eta_x^2\eta_y^2|\eta_z'|^2,
\end{aligned}
\]
equations \eqref{eq:cutoff_weight_sum} and
\eqref{eq:cutoff_derivative_sum} imply
\begin{equation}
\begin{aligned}
    \|\widetilde{\bm B}\otimes\nabla\eta\|_{L^2(W)}^2
    \leq{}&
    8\left(
        L_x^{-2}+L_y^{-2}+1
    \right)
    \|\bm B\|_{L^2(V)}^2.
\end{aligned}
\end{equation}
Therefore,
\begin{equation}
    \|\widetilde{\bm B}\otimes\nabla\eta\|_{L^2(W)}
    \leq
    2\sqrt{2}\,
    \sqrt{1+L_x^{-2}+L_y^{-2}}\,
    \|\bm B\|_{L^2(V)}.
    \label{eq:weighted_cutoff_term}
\end{equation}

Combining \eqref{eq:extension_product_rule},
\eqref{eq:weighted_gradient_term}, and
\eqref{eq:weighted_cutoff_term}, we obtain
\begin{equation}
\begin{aligned}
    \|\nabla\bm F\|_{L^2(\mathbb R^3)}
    \leq
    2\sqrt{2}\Bigl(
        \|\nabla\bm B\|_{L^2(V)}
        &+
        \sqrt{1+L_x^{-2}+L_y^{-2}}\,
        \|\bm B\|_{L^2(V)}
    \Bigr).
\end{aligned}
\end{equation}
Using the Poincar\'e inequality then gives
\begin{equation}
\begin{aligned}
    \|\nabla\bm F\|_{L^2(\mathbb R^3)}
    \leq
    2\sqrt{2}
    \left(
        1+
        \frac{\sqrt{1+L_x^{-2}+L_y^{-2}}}{\lambda_1}
    \right)
    \|\nabla\bm B\|_{L^2(V)}.
    \label{eq:nabla_F_ineq}
\end{aligned}
\end{equation}

Talenti's sharp whole-space Sobolev inequality, applied to
$|\bm F|$, gives
\begin{equation}
    \|\bm F\|_{L^6(\mathbb R^3)}
    \leq
    C_3\|\nabla|\bm F|\|_{L^2(\mathbb R^3)}
    \leq
    C_3\|\nabla\bm F\|_{L^2(\mathbb R^3)}.
\end{equation}
Since $\bm F=\bm B$ on $V$, it follows from
\eqref{eq:nabla_F_ineq} that
\begin{equation}
    \|\bm B\|_{L^6(V)}
    \leq
    C_{\mathrm{S}}\|\nabla\bm B\|_{L^2(V)},
\end{equation}
with
\begin{equation}
    C_{\mathrm{S}}
    =
    2\sqrt{2}\,C_3
    \left(
        1+
        \frac{\sqrt{1+L_x^{-2}+L_y^{-2}}}{\lambda_1}
    \right).
\end{equation}
\end{proof}

\begin{remark}
    If a scalar or vector $f$ admits an $H^1_{loc}$ reflected extension to the enlarged domain and satisfies $\| f\|_2 \leq \frac{1}{\lambda_1}\| \nabla f\|_2$ then the cutoff calculation proves $\| f \|_6 \leq C_6 \| \nabla f \|_2$.
\end{remark}
\begin{remark}
    If a Poincar\'e inequality does not exist, as for the case of $\nabla^2 T$, then
    \begin{equation}
        \|F \|_6 \leq C_{\mathrm{S},0} \| F \|_{H^1}
    \end{equation}
    where
    \begin{equation}
        C_{\mathrm{S},0} := 2\sqrt{2} C_3 \sqrt{2 + L^{-2}_x +  L^{-2}_y}.
    \end{equation}
\end{remark}

\begin{proof}[Proof of $C_\infty$]
Define the normalized eigenfunctions of the horizontally periodic Dirichlet Laplacian by
\begin{equation*}
    e_{mnk}(x,y,z)
=
\left(\frac{2}{L_xL_y}\right)^{1/2}
\exp\left[
i\left(
\frac{2\pi m}{L_x}x+
\frac{2\pi n}{L_y}y
\right)
\right]
\sin(k\pi z).
\end{equation*}
They satisfy
\begin{equation*}
    -\Delta e_{mnk}
=
\lambda_{mnk}e_{mnk}.
\end{equation*}
Set$ f=\Delta T$. Writing
\begin{equation}
    T=\sum_{m,n\in\mathbb Z}\sum_{k=1}^{\infty}
\widehat T_{mnk}e_{mnk},
\qquad
f=\sum_{m,n\in\mathbb Z}\sum_{k=1}^{\infty}
\widehat f_{mnk}e_{mnk},
\end{equation}
Green's identity, horizontal periodicity, and the conditions
\begin{equation}
    T=e_{mnk}=0
\qquad\text{on }z=0,1
\end{equation}
give
\begin{equation}
    \widehat f_{mnk}
=
-\lambda_{mnk}\widehat T_{mnk}.
\end{equation}
Notice that this identity does not require \(f=\Delta T\) to vanish at the boundaries. Consequently,
\begin{equation*}
    \nabla T
=
-\sum_{m,n\in\mathbb Z}\sum_{k=1}^{\infty}
\frac{\widehat f_{mnk}}{\lambda_{mnk}}
\nabla e_{mnk}.
\end{equation*}
Applying the Cauchy-Schwarz inequality with weights
\(1+\lambda_{mnk}- k^2 \pi^2\) gives
\begin{align}
    |\nabla T(x,y,z)|^2
\leq
&\left[
\sum_{m,n\in\mathbb Z}\sum_{k=1}^{\infty}
\bigl(1+\lambda_{mnk} - k^2\pi^2\bigr)
|\widehat f_{mnk}|^2
\right] \nonumber \\ &\times
\left[
\sum_{m,n\in\mathbb Z}\sum_{k=1}^{\infty}
\frac{|\nabla e_{mnk}(x,y,z)|^2}
{\lambda_{mnk}^{\,2}\bigl(1+\lambda_{mnk} - k^2\pi^2\bigr)}
\right].
\end{align}
By Parseval's identity in the horizontal variables and the completeness of the sine basis in the vertical variable,
\begin{equation}
    \sum_{m,n\in\mathbb Z}\sum_{k=1}^{\infty}
\bigl(1+\lambda_{mnk} - k^2\pi^2\bigr)
|\widehat f_{mnk}|^2
=
\|f\|_2^2+\|\horg  f\|_2^2
\leq
\|f\|_{H^1}^2.
\end{equation}
Furthermore,
\begin{equation}
    |\nabla e_{mnk}(x,y,z)|^2
=
\frac{2}{L_xL_y}
\left[
\lambda_{mnk} - k^2\pi^2\sin^2(k\pi z)
+
k^2\pi^2\cos^2(k\pi z)
\right]
\leq
\frac{2}{L_xL_y}\lambda_{mnk}.
\end{equation}
It follows that
\begin{equation}
    |\nabla T(x,y,z)|^2
\leq
\frac{2}{L_xL_y}
\left[
\sum_{m,n\in\mathbb Z}\sum_{k=1}^{\infty}
\frac{1}
{\bigl(1+\lambda_{mnk} - k^2\pi^2\bigr)\lambda_{mnk}}
\right]
\|\Delta T\|_{H^1}^2.
\end{equation}
Taking the supremum over $V$ proves the result.
The series converges as the summand behaves as $\lambda_{mnk}^{-2}$ for large wavenumbers.
This completes the proof.
\end{proof}

\begin{proof}[Proof of $C_b$]
Set
\[
    \alpha_m=\frac{2\pi m}{L_x},
    \qquad
    \beta_n=\frac{2\pi n}{L_y}.
\]
The normalized eigenfunctions satisfying homogeneous Dirichlet
conditions at $z=0,1$ are
\begin{equation}
    e^D_{mnk}(x,y,z)
    =
    \sqrt{\frac{2}{L_x L_y}}\,
    e^{\mathrm i(\alpha_mx+\beta_ny)}
    \sin(k\pi z),
    \qquad k\geq1.
\end{equation}
The normalized eigenfunctions satisfying homogeneous Neumann
conditions are
\begin{equation}
    e^N_{mn0}(x,y,z)
    =
    \frac{1}{\sqrt{ L_x L_y}}\,
    e^{\mathrm i(\alpha_mx+\beta_ny)}
\end{equation}
for $k=0$, and
\begin{equation}
    e^N_{mnk}(x,y,z)
    =
    \sqrt{\frac{2}{L_x L_y}}\,
    e^{\mathrm i(\alpha_mx+\beta_ny)}
    \cos(k\pi z),
    \qquad k\geq1.
\end{equation}
All these functions satisfy
\begin{equation}
    -\nabla^2 e_{mnk}=\lambda_{mnk}e_{mnk}.
\end{equation}

The horizontal components have Dirichlet expansions,
\begin{equation}
    B_i
    =
    \sum_{m,n\in\mathbb Z}
    \sum_{k=1}^{\infty}
    \widehat B^i_{mnk}e^D_{mnk},
    \qquad i\in\{x,y\},
\end{equation}
whereas the vertical component has the Neumann expansion
\begin{equation}
    B_z
    =
    \sum_{\substack{
        m,n\in\mathbb Z,\ k\in\mathbb N_0\\
        (m,n,k)\neq(0,0,0)
    }}
    \widehat B^z_{mnk}e^N_{mnk}.
\end{equation}
The mode $(m,n,k)=(0,0,0)$ is absent because of the zero net-flux
condition.

For either horizontal component, the Cauchy--Schwarz inequality
gives
\begin{align}
    |B_i(x,y,z)|^2
    &\leq
    \left[
        \sum_{m,n\in\mathbb Z}
        \sum_{k=1}^{\infty}
        \bigl(1+\lambda_{mnk}+\lambda_{mnk}^2\bigr)
        |\widehat B^i_{mnk}|^2
    \right]
    \nonumber\\
    &\qquad\times
    \left[
        \sum_{m,n\in\mathbb Z}
        \sum_{k=1}^{\infty}
        \frac{|e^D_{mnk}(x,y,z)|^2}
        {1+\lambda_{mnk}+\lambda_{mnk}^2}
    \right].
    \label{eq:Bi_Cauchy}
\end{align}
Since
\[
    |e^D_{mnk}(x,y,z)|^2\leq\frac{2}{L_x L_y},
\]
the second factor in \eqref{eq:Bi_Cauchy} is bounded by
$C_b^2$.

The same argument for the vertical component gives
\begin{align}
    |B_z(x,y,z)|^2
    &\leq
    \left[
        \sum_{\substack{
            m,n\in\mathbb Z,\ k\in\mathbb N_0\\
            (m,n,k)\neq(0,0,0)
        }}
        \bigl(1+\lambda_{mnk}+\lambda_{mnk}^2\bigr)
        |\widehat B^z_{mnk}|^2
    \right]
    \nonumber\\
    &\qquad\times
    \left[
        \sum_{\substack{
            m,n\in\mathbb Z,\ k\in\mathbb N_0\\
            (m,n,k)\neq(0,0,0)
        }}
        \frac{|e^N_{mnk}(x,y,z)|^2}
        {1+\lambda_{mnk}+\lambda_{mnk}^2}
    \right].
\end{align}
The normalized Neumann eigenfunctions satisfy
\begin{equation}
    |e^N_{mnk}(x,y,z)|^2
    \leq
    \frac{2-\delta_{k0}}{L_x L_y}.
\end{equation}
Consequently,
\[
    |B_i(x,y,z)|^2
    \leq
    C_b^2\|B_i\|_{H^2_\Delta}^2
    \qquad
    \text{for }i\in\{x,y,z\},
\]
where
\[ 
\| \bm{B} \|^2_{H^2_{\Delta}}:= \|\bm{B} \|^2_2 + \| \nabla \bm B\|_2^2 + \|\nabla^2 \bm B \|_2
\]
Summing over the three components yields
\begin{equation}
\begin{aligned}
    |\bm B(x,y,z)|^2
    &=
    |B_x|^2+|B_y|^2+|B_z|^2\\
    &\leq
    C_b^2
    \left(
        \|B_x\|_{H^2_\Delta}^2
        +\|B_y\|_{H^2_\Delta}^2
        +\|B_z\|_{H^2_\Delta}^2
    \right)\\
    &=
    C_b^2\|\bm B\|_{H^2_\Delta}^2.
\end{aligned}
\end{equation}
Taking the essential supremum over $V$ proves
the embedding $\| \bm{B}\|_\infty\leq C_b \| \bm{B} \|_{H^2}$. The series defining $C_b$ converges because
its summand behaves as $\lambda_{mnk}^{-2}$ for large wavenumbers.
\end{proof}

\section{Poloidal-Toroidal Decomposition}
\label{sec:polotoro}
If magnetic flux is exact, a magnetic field $\bm{B}$ in the plane layer admits the  poloidal-toroidal decomposition
\begin{equation}
  \bm{B} = \nabla \times \nabla \times  \pol \bm{e}_{z} + \nabla \times
  \tor \bm{e}_{z} + \bm{B}^0 =: \bm{B}^{\pol} + \bm{B}^{\tor} + \bm{B}^0.
\end{equation}
Respectively, $\pol$ and $\tor$ are the scalar potentials for the poloidal and
toroidal parts of the magnetic field $\bm{B}^{\pol}$ and $\bm{B}^{\tor}$. In the plane layer, the representation is completed by the (horizontal) base field
\begin{equation}
    \bm{B}^0 = B^0_x (z) \, \bm{e}_x + B^0_y (z) \, \bm{e}_y.
\end{equation}
Clearly, for $\alpha \in \{ \pol, \tor, 0\}$ it holds 
\begin{equation}
\nabla \cdot  \bm{B}^{\alpha} = 0.
\end{equation}
Here, we establish orthogonality in $L^2_{\mathrm{sol}}$ and provide representation formulae.
\begin{proposition}
\label{prop:orth}
 If $\alpha \neq \beta \in \{\pol, \tor, 0\}$, then
 \begin{equation*}
 \langle \bm{B}^{\alpha}, \bm{B}^{\beta} \rangle = 0. 
 \end{equation*}
\end{proposition}
\begin{proof}
  Writing the volume integral as a vertical integral in $\mathrm{d}z$ of surface integrals in $\mathrm{d}x\, \mathrm{d}y$, it follows that
  \begin{equation}
  \begin{aligned}
    \langle \bm{B}^{\pol}, \bm{B}^{\tor} \rangle &= \int_{0}^{1} \mathrm{d}z \int \bm{B}^{\pol}
    \cdot \bm{B}^{\tor} \, \mathrm{d}x\, \mathrm{d}y \\
                                 &= \int_0^{1} \mathrm{d}z  \int \horg 
                                 \partial_{z} \pol \cdot \horg ^{\perp} \tor \, \mathrm{d}x \, \mathrm{d}y \\
                                 &= -\int_0^{1} \mathrm{d}z  \int
                                 \partial_{z} \pol \horg  \cdot \horg ^{\perp} \tor \, \mathrm{d}x \, \mathrm{d}y =
                                 0.
  \end{aligned}
  \end{equation}
   In the last line, we integrated by parts on each surface 
   using that $\horg $ is tangent to  the surface and that $\horg  \cdot  \horg ^{\perp} =
   0$ in general. 

   For $\alpha \in \{\pol, \tor\}$, it follows that
   \begin{equation}
       \begin{aligned}
            \langle \bm{B}^{0}, \bm{B}^{\alpha} \rangle &= \int_{0}^{1} \mathrm{d}z \,\bm{B}^{0} \cdot \int  
     \bm{B}^{\alpha } \, \mathrm{d}x\, \mathrm{d}y. \\
     &= \int_{0}^{1} \mathrm{d}z \left( B^{0}_x  \int  
     B^{\alpha }_x \, \mathrm{d}x\, \mathrm{d}y + B^{0}_y  \int  
     B^{\alpha }_y \, \mathrm{d}x\, \mathrm{d}y \right) = 0.
       \end{aligned}
   \end{equation}
   The last equality holds because in either case $\alpha  = \pol$ or $\alpha = \tor$, the components $B_x^\alpha$ and $B_y^\alpha$ are total derivatives in $\partial_x$ or $\partial_y$.
\end{proof}
A basic observation is that the toroidal part is
$\horg $-divergence free
\begin{equation}
\label{eq:divT}
\horg \cdot \bm{B}^{\tor} = \partial_x B^{\tor}_x + \partial_y B^{\tor}_y = 0,
\end{equation}
and the poloidal part is $\horg $-curl free
\begin{equation}
 \horg^{\perp} \cdot  \bm{B}^{\pol} = \partial_x B^{\pol}_y - \partial_y B^{\pol}_x = 0.
\end{equation}
Clearly, the base part $\bm{B}^0$ is annihilated by $\horg.$
\begin{proposition}
  The scalar potential for the poloidal part has the representation
  \begin{equation*}
   \pol(x,y,z) =  \pol(x,y,0) + \int_0^{z} \mathrm{d}z' \, (\hord)^{-1}
  (\horg \cdot \bm{B})(x, y, z').
  \end{equation*}
\end{proposition}
\begin{proof}
  By divergence theorem,
\begin{equation}
    \int \horg \cdot \bm{B} \, \mathrm{d}x \, \mathrm{d}y = 0.
\end{equation}
  Thus $(-\hord)^{-1}$ is well defined on every constant $z$ surface. Then
\begin{equation}
  \horg  \cdot  \bm{B}= \horg  \cdot  \bm{B}^{\pol} = \hord \partial_{z}\pol.
\end{equation}
  Inverting $\hord$ and integrating in  $\mathrm{d}z$  yields the claim.
\end{proof}

\begin{proposition}
  The scalar potential for the toroidal part has the representation
  \begin{equation*}
  \tor(x,y,z) = ( \hord)^{-1} (\horg^{\perp} \cdot \bm{B})(x,y,z).
  \end{equation*}
\end{proposition}
\begin{proof}
  Let us first verify that the $\horg $-curl has zero average on each
 constant $z$ surface. Indeed, after integrating by parts
\begin{equation}
  \int \horg ^{\perp} \cdot \bm{B} \, \mathrm{d}x \, \mathrm{d}y = \int \mathrm{d}y \int \partial_{x} B_{y} \, \mathrm{d}x - \int \mathrm{d}x \int \partial_{y}
  B_{x} \, \mathrm{d}y  = 0
\end{equation}
and so $(-\hord)^{-1}$ in the representation is well defined. Pointwise it
holds
\begin{equation}
\horg ^{\perp} \cdot \bm{B}^{\pol} = \horg ^{\perp} \cdot \horg  \partial_{z} \pol = 0.
\end{equation}
Therefore, as claimed
\begin{equation}
\horg ^{\perp} \cdot \bm{B} = \horg ^{\perp} \cdot \bm{B}^{\tor} = \horg ^{\perp} \cdot \horg ^{\perp}\tor =
\hord \tor .
\end{equation}
\end{proof}
On the other hand, the poloidal and toroidal part have zero horizontal mean, if $\alpha \in \{ \pol, \tor\}$ then
\begin{equation}
    \bm{e}_x\int B_x^{\alpha } \,\mathrm{d} x\,\mathrm{d}y + \bm{e}_y\int B_y^{\alpha } \,\mathrm{d} x\,\mathrm{d}y = 0.
\end{equation}
Almost tautologically,
\begin{proposition}
    The (horizontal) base part has the representation
    \begin{equation*}
    \bm{B}^0(z) = \bm{e}_x\int B_x(z) \,\mathrm{d} x\,\mathrm{d}y + \bm{e}_y\int B_y(z) \,\mathrm{d} x\,\mathrm{d}y .
\end{equation*}
\end{proposition}

\printbibliography

\end{document}

%% file: fig/ekman.tex
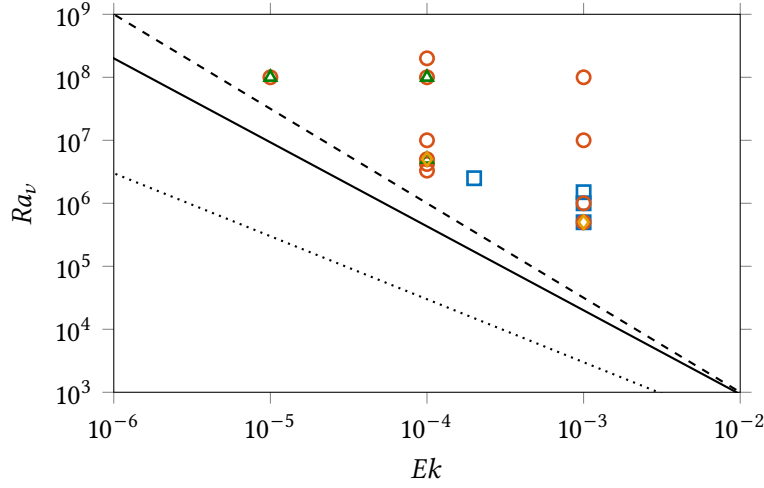
\begin{figure}[t]
\centering
\begin{tikzpicture}
\begin{loglogaxis}[
    width=0.78\textwidth,
    height=0.52\textwidth,
    xlabel={$Ek$},
    ylabel={$\Rmod_{\nu}$},
    xmin=1e-6, xmax=1e-2,
    ymin=1e3, ymax=1e9,
    legend style={at={(0.02,0.02)},anchor=south west,draw=none,fill=none},
    tick align = outside,
    minor tick style={draw=none},
]

\definecolor{matlabgreen}{rgb}{0,0.5,0} 
\definecolor{mypurple}{rgb}{1,0,1}
\definecolor{matlabblue}{RGB}{0,113,188}
\definecolor{matlabred}{RGB}{216,82,24}
\definecolor{mygrey}{rgb}{0.7,0.7,0.7}
\definecolor{paperorange}{RGB}{230,159,0}
\definecolor{papergreen}{RGB}{0,158,115}
\definecolor{paperpurple}{RGB}{117,112,179}
\definecolor{paperpink}{RGB}{204,121,167}

\def\Cconv{2.0}    
\def\Ccomb{1.0} 

\def\Ccombb{3.0}  

\renewcommand{\t}{\tiny}
\newcommand{\mrksz}{2.5pt}
\newcommand{\lw}{1pt}


\addplot[
    only marks,
    mark=square,
    mark size=\mrksz,
    color=matlabblue,
    mark options = { 
      line width = \lw,
    draw=matlabblue
    }
] coordinates {
    (1e-3,5e5)
    (2e-4,2.5e6)
    (1e-3,1e6)
    (1e-3,1.5e6)
};

\addplot+[
    only marks,
    mark=o,
    mark size=\mrksz,
    color=matlabred,
    mark options = { 
    line width = \lw,
    draw=matlabred,
    fill=none
    }
] coordinates {
    (1e-3,5e5)
    (1e-3,1e6)
    (1e-3,1e7)
    (1e-3,1e8)
    (1e-4,3.3e6)
    (1e-4,4.2e6)
    (1e-4,5e6)
    (1e-4,1e7)
    (1e-4,1e8)
    (1e-4,2e8)
    (1e-5,1e8)
};



\addplot+[
    only marks,
    mark=triangle,
    mark size=\mrksz,
    color=matlabgreen,
     mark options = { 
      line width = \lw,
    draw= matlabgreen,
    fill= none
    }
] coordinates {
    (1e-4,1e8)
    (1e-4,5e6)
    (1e-5,1e8)
};

\addplot+[
    only marks,
    mark=diamond,
    mark size=\mrksz,
    color=paperorange,
    mark options = { 
     line width = \lw,
    draw= paperorange,
    fill= none
    }
] coordinates {
    (1e-3,5e5)
    (1e-4,5e6)
};


\addplot[thick,domain=1e-10:1e-1,samples=300]
    {\Cconv * x^(-4/3)};

\addplot[thick,dashed,domain=1e-10:1e-1,samples=300]
    {\Ccomb * x^(-3/2)};

\addplot[thick,dotted,domain=1e-10:1e-1,samples=300]
    {\Ccombb * x^(-1)};

\end{loglogaxis}
\end{tikzpicture}
\caption{
Comparison between the convective onset scaling $Ek^{-4/3}$ (black solid line) with the  scaling required for dynamo action as discussed in  Remarks \ref{rem:scaling1} and \ref{rem:scaling2}: $Ek^{-3/2}$ (black dashed line) and  $Ek^{-1}$ (black dotted line). Dynamos from simulations are plotted using the classical Rayleigh number $\Ra_{\nu}$ against $Ek$, for fixed $q$. Blue squares represent data from Jones \& Roberts (2000) \cite{jones2000convection}, red circles from Cattaneo \& Hughes (2017) \cite{cattaneo2017dynamo}, orange diamonds from Hughes \& Cattaneo (2016)\cite{hughes2016strong} and green triangles from Hughes \& Cattaneo (2019) \cite{hughes2019force}. 
}
\label{fig:onset_vs_combined}
\end{figure}